\documentclass[11pt]{amsart}
\usepackage[top=1.2in,bottom=1.2in,left=1.2in,right=1.2in]{geometry}
\usepackage[inline]{enumitem}

\usepackage{subcaption,amssymb,pifont}
\usepackage{paracol}
\usepackage{pinlabel}
\usepackage{hyperref}
\hypersetup{
    colorlinks,
    citecolor=black,
    filecolor=black,
    linkcolor=black,
    urlcolor=black
}
\usepackage{multicol}

\usepackage{amsthm}

\usepackage{mathtools}

\usepackage{graphicx}
\usepackage{amssymb}
\usepackage{epstopdf}
\usepackage{hyperref}
\usepackage{xypic}

\newcommand\C{\mathbb{C}}
\newcommand\D{\mathbb{D}}
\newcommand\Q{\mathbb{Q}}
\newcommand\R{\mathbb{R}}
\newcommand\ZZ{\mathbb{Z}}
\renewcommand\H{\mathcal{H}}
\newcommand\M{\mathcal{M}}
\newcommand\I{\mathcal{I}}
\newcommand\SI{\mathcal{SI}}
\newcommand\T{\mathcal{T}}

\newcommand\K{\mathcal{K}}
\newcommand\BI{\mathcal{BI}}
\renewcommand\L{\mathcal{L}}

\renewcommand\cH{\mathcal{H}}

\newcommand\Sp{\text{Sp}}
\newcommand\GL{\text{GL}}

\DeclareMathOperator{\Mod}{Mod}
\DeclareMathOperator{\SMod}{SMod}
\DeclareMathOperator{\PMod}{PMod}

\DeclareMathOperator{\Hom}{Hom}

\DeclareMathOperator{\B}{B}
\DeclareMathOperator{\PB}{PB}
\DeclareMathOperator{\Z}{\mathcal{Z}}
\DeclareMathOperator{\PZ}{\mathcal{P}\hspace*{-.3ex}\mathcal{Z}}

\DeclareMathOperator{\Lift}{Lift}
\DeclareMathOperator{\Capmap}{Cap}

\DeclareMathOperator{\Res}{Res}
\DeclareMathOperator{\Ind}{Ind}

\DeclareMathOperator{\cd}{cd}

\newcommand{\p}[1]{\bigskip \noindent \emph{#1}.}

\renewcommand{\S}{\mathcal S}

\numberwithin{equation}{section}

\theoremstyle{plain}

\newtheorem{theorem}{Theorem}[section]
\newtheorem{proposition}[theorem]{Proposition}
\newtheorem{lemma}[theorem]{Lemma}
\newtheorem{corollary}[theorem]{Corollary}
\newtheorem{conjecture}[theorem]{Conjecture}
\newtheorem{question}{Question}

\title{Representation stability in the level 4 braid group}
\author{Kevin Kordek}
\author{Dan Margalit}
\thanks{}
\date{}

\begin{document}

\vspace*{-4.5ex}

\maketitle

\vspace*{-7ex}

\begin{abstract}
We investigate the cohomology of the level 4 subgroup of the braid group, namely, the kernel of the mod 4 reduction of the Burau representation at $t=-1$.  This group is also equal to the kernel of the mod 2 abelianization of the pure braid group.  We give an exact formula for the first Betti number; it is a quartic polynomial in the number of strands.  We also show that, like the pure braid group, the first homology satisfies uniform representation stability in the sense of Church and Farb.  Unlike the pure braid group, the group of symmetries---the quotient of the braid group by the level 4 subgroup---is one for which the representation theory has not been well studied; we develop its representation theory.  This group is a non-split extension of the symmetric group.

As applications of our main results, we show that the rational cohomology ring of the level 4 braid group is not generated in degree 1 when the number of strands is at least 15, and we compute all Betti numbers of the level 4 braid group when the number of strands is at most 4. We also derive a new lower bound on the first rational Betti number of the hyperelliptic Torelli group and on the top rational Betti number of the level 4 mapping class group in genus 2.  Finally, we apply our results to locate all of the 2-torsion points on the characteristic varieties of the pure braid group. 
\end{abstract}


\vspace*{0ex}

\section{Introduction}

For an integer $m\geq 0$, the \emph{level $m$ braid group} $\B_n[m]$ is a finite-index subgroup of the braid group $\B_n$. It is defined as the kernel of the composition 
\[
\B_n \to  \GL_n(\ZZ[t,t^{-1}]) \to \GL_n(\ZZ) \to \GL_n(\ZZ/m)
\]
where the first map is the (unreduced) Burau representation (see Birman's book \cite{birman}), the next map is evaluation at $t=-1$ and the last map is given by reducing entries mod $m$. 

Arnol'd \cite{arnold} showed that $\B_n[2]$ is equal to the pure braid group $\PB_n$. Brendle and the second author showed that $\B_n[4]$ is equal to $\PB_n^2$, the subgroup of $\PB_n$ generated by all squares of elements \cite{brendlemargalit}.  Equivalently, $\B_n[4]$ is the kernel of the mod 2 abelianization map 
\[
\PB_n \to H_1(\PB_n;\ZZ/2) \cong (\ZZ/2)^{n \choose 2}.
\]
As such $\B_n[4]$ is also isomorphic to the fundamental group of the mod 2 congruence cover of the complement of the braid arrangement $X_n$.   Brendle and the second author additionally showed that $\B_n[4]$ is equal to the subgroup of $\PB_n$ generated by all squares of Dehn twists \cite{brendlemargalit}.   Little else is known about the algebraic structure of $\B_n[m]$ when $m > 2$.

Our main result is Theorem~\ref{thm:reps}, which states that $H_1(\B_n[4];\C)$ satisfies uniform representation stability in the sense of Church and Farb.  The group of symmetries is $\Z_n = \B_n/\B_n[4]$, which is a non-split extension of the symmetric group $S_n$ by $H_1(\PB_n;\ZZ/2)$.   Theorem~\ref{thm:reps} gives the explicit decomposition of $H_1(\B_n[4];\C)$ into irreducible $\Z_n$-modules:
\[
H_1(\B_n[4]; \C)\cong 
\begin{cases}
V_2(1,(0)) & n =2 \\[0ex]
V_3(1,(0)) \oplus V_3(1,(1)) \oplus V_3(\rho_3,(0)) & n=3 \\[0ex]
V_n(1,(0)) \oplus V_n(1,(1)) \oplus V_n(1,(2)) \oplus V_n(\rho_3,(0)) \oplus V_n(\rho_4,(0))  & n\geq 4.\\
\end{cases}
\]
Each summand here is of the form
\[
V_n(\rho,\lambda) = \Ind_{\Z_n^I}^{\Z_n} \left(V_m(\rho) \boxtimes V_{n-m}(\lambda)\right)
\]
where $\Z_n^I$ is the stabilizer in $\Z_n$ of a set $I$ of pairs of elements of $\{1,\dots,n\}$, $V_m(\rho)$ is an irreducible representation of $\Z_m^I$, and $V_{n-m}(\lambda)$ is an irreducible representation of $S_{n-m}$.  See Theorem~\ref{thm:class} and the preceding discussion for the precise definitions.

Our first step towards proving Theorem~\ref{thm:reps} is to prove Theorem~\ref{thm:main}, which gives an explicit basis for $H_1(\B_n[4];\Q)$.  From this basis we obtain a formula for the first Betti number of $\B_n[4]$, which is a quartic polynomial in $n$:
\[
\displaystyle \dim H_1(\B_n[4];\Q) = 3{n \choose 4} + 3{n \choose 3} + {n \choose 2}.
\]

Church and Farb introduced the theory of representation stability and proved that $H_k(\PB_n;\Q)$ satisfies uniform representation stability \cite{churchfarb}.  By passing to $\PB_n$-invariants, our Theorem~\ref{thm:reps} recovers their result for $k=1$.  Church and Farb take advantage of the explicit basis for $H_k(\PB_n;\Q)$ provided by Arnol'd; our Theorem~\ref{thm:main} plays the role of the Arnol'd result.  They also employ the representation theory of $S_n$.  As part of our work we develop the representation theory of $\Z_n$ from the ground up (Section~\ref{sec:bnbar}).  Our results suggest that it is an interesting problem to study the stability of $H_k(\B_n[m])$ as $k$, $n$, and $m$ vary.

We derive a number of consequences of our methods and results, about the level 4 hyperelliptic mapping class group $\SMod_g[4]$, the braid Torelli group $\BI_g$, the level 4 mapping class group $\Mod_g[4]$, the characteristic varieties $V_d(X_n)$ of the braid arrangement $X_n$, and $\B_n[4]$ itself.  Specifically, we give the following applications.
\begin{enumerate}[leftmargin=*]
\item $H^*(\B_n[4];\Q)$ is not generated in degree 1 (Theorem~\ref{thm:mainD}).
\item $\B_n[4]$ is not generated by 4th powers of half-twists (Theorem~\ref{thm:2}).
\item All 2-torsion on $V_1(X_n)$ lies on central components,  and outside $V_2(X_n)$ (Theorem~\ref{thm:2torsion}).
\newcounter{enumi_saved}
\setcounter{enumi_saved}{\value{enumi}}
\end{enumerate}
We further give:
\begin{enumerate}[leftmargin=*]
\setcounter{enumi}{\value{enumi_saved}}
\item a new lower bound for the first Betti number of $\BI_g$ (Theorem~\ref{thm:mainF}),
\item a new lower bound for the top Betti number of $\Mod_2[4]$ (Proposition~\ref{prop:genus2prop}), and 
\item computations of all Betti numbers of $\B_n[4]$ for $n \leq 4$ (Theorem~\ref{thm:smallbetti}).
\setcounter{enumi_saved}{\value{enumi}}
\end{enumerate}
Finally, we obtain analogues of some of our results about $\B_n[4]$ for $\SMod_g[4]$:
\begin{enumerate}[leftmargin=*]
\setcounter{enumi}{\value{enumi_saved}}
\item we determine the first Betti number of $\SMod_g[4]$  (Corollary~\ref{cor:mainC}), and 
\item we show $H^*(\SMod_g[4];\Q)$ is not generated in degree 1 (Theorem~\ref{thm:mainE}).
\end{enumerate}

Representation stability has been studied for representations of Weyl groups (such as the symmetric groups and the hyperoctahedral groups), certain linear groups, and certain wreath products, among others; see the surveys by Farb, Khomenko--Kesari, and Wilson  \cite{farb, khomenko, wilsonnotes}.  The group $\Z_n$ seems to have not appeared before in the theory.  The general trend has to obtain representation stability from the finite generation of a module over a category associated to a sequence of groups.  The category for the groups $\{\Z_n\}$ is the subject of a forthcoming paper with Miller and Patzt.  With the current technology it does not appear to be possible to obtain our uniform representation stability from this categorical viewpoint.

Representation stability has also been studied extensively for various types of configuration spaces, beginning with the work of Church, Church--Farb, and Church--Ellenberg--Farb \cite{church, cef1, churchfarb}.  However, there is no general theory for the representation stability of the homology of covers of configuration spaces.  On the other hand, the representation stability for the homology of specific covers, such as orbit configuration spaces, have been studied, for example, by Bibby--Gadish and Casto \cite{bibby,casto}.   Congruence covers of complements of hyperplane arrangements are well studied; see, for example, the survey by Suciu \cite{suciu}.  However, we are not aware of any previously known general closed formulas for the Betti numbers of congruence covers of $X_n$.

The uniform representation stability of the $\{H_1(\B_n[4];\C) \}$ can also be interpreted as a result about the twisted homology of $\B_n$ with coefficients in the $V_n(\rho,\lambda)$.  Wahl and Randal-Williams have proven a general stability result for the homology of braid groups with twisted coefficients \cite{wahl}.  Their theorem applies to certain coefficient systems, called polynomial coefficient systems.   It seems to be an interesting (and difficult) problem to determine if the $V_n(\rho,\lambda)$ are polynomial in this sense.  Even if this were the case, their result would not imply our Theorem~\ref{thm:reps}; it would only imply that the multiplicities stabilize.

In the remainder of the introduction, we explain our results in detail and give an outline for the paper.  We begin with a discussion of Theorem~\ref{thm:main}.  This theorem can in theory be obtained from our main result, Theorem~\ref{thm:reps} (and the dimension count in Theorem~\ref{thm:main} does indeed follow immediately).  However, our entire approach to Theorem~\ref{thm:reps} is predicated on Theorem~\ref{thm:main}.


\subsection*{An explicit basis} In this paper we identify $\B_n$ with the mapping class group of the 2-dimensional disk $\D_n$ with $n$ marked points in the interior  \cite[Chapter 9]{farbmargalit}.  In general, for a surface $S$, possibly with boundary and possibly with marked points, we define the mapping class group $\Mod(S)$ as the group of homotopy classes of orientation-preserving homeomorphisms of $S$ that preserve the set of marked points and fix the boundary pointwise.

We label the marked points by $[n] = \{1,\dots,n\}$ and denote by $T_{ij}$ the (left) Dehn twist about the curve in $\D_n$ indicated in Figure~\ref{fig:twist} and by $\tau_{ij}$ the image of $T_{ij}^2$ in $H_1(\B_n[4];\Q)$.  The group $\PB_n$ acts on $H_1(\B_n[4];\Q)$ by conjugation; for $f \in \PB_n$ we denote by $f\tau_{ij}$ the image of $\tau_{ij}$ under the action of $f$.  

For $n \geq 4$ we define the set $\S$ to be $\S_1 \cup \S_2 \cup \S_3$ where
\begin{align*}
\S_1 &= \{ \tau_{ij} \mid i < j \}, \\
\S_2 &= \{ T_{ik}\tau_{ij},\ T_{jk}\tau_{ik},\  T_{ij}\tau_{jk} \mid i < j < k   \}, \text{ and} \\
\S_3 &= \{ T_{il}T_{jk} \tau_{ij},\  T_{ij}T_{k\ell} \tau_{ik},\ T_{ik}T_{j\ell} \tau_{i\ell} \mid i < j < k < \ell \}.
\end{align*}
For $n < 4$ we define $\S$ in the same way, except that we declare $\S_1$, $\S_2$, and $\S_3$ to be empty when $n$ is less than 2, 3, and 4, respectively.    In this paper we compose elements of $\B_n$ from right to left (functional notation).

\begin{theorem}\label{thm:main}
For all $n\geq 1$ the set $\S$ is a basis for $H_1(\B_n[4];\Q)$.  In particular,
\[
\displaystyle \dim H_1(\B_n[4];\Q) = 3{n \choose 4} + 3{n \choose 3} + {n \choose 2}.
\]
\end{theorem}

We do not know if the abelianization of $\B_n[4]$ is torsion free, which is to say that we do not know a complete description of $H_1(\B_n[4];\ZZ)$.  On the other hand, the proof of Theorem~\ref{thm:main} also works with $\ZZ/p\ZZ$ coefficients for any odd prime $p$, implying that any non-trivial torsion in the abelianization would have to be 2-primary.

As applications of Theorem~\ref{thm:main} we prove the following two theorems in Section~\ref{sec:smallbetti} and~\ref{sec:nongen}, respectively.  The first gives all Betti numbers for $\B_3[4]$ and $\B_4[4]$.

\begin{theorem}\label{thm:smallbetti}
For $n=3$ and $n=4$, the dimensions of $H_k(\B_n[4]; \Q)$ are as follows:
\[
\dim H_k(\B_3[4]; \Q) = \left\{
\begin{array}{lr}
1 & k = 0 \\
6 & k = 1 \\
5 & k = 2 \\
0 & k\geq 3
\end{array}
\right.
\qquad
\dim H_k(\B_4[4]; \Q) = \left\{
\begin{array}{lr}
1 & k = 0 \\
21& k = 1 \\
103 & k = 2 \\
83 & k = 3 \\
0 & k\geq 4
\end{array}
\right.
\]
\end{theorem}

The second application shows that, even though $H_1(\B_n[4];\Q)$ is generated by the images of 4th powers of half-twists, the group $\B_n[4]$ is not.  This is in contrast with $\B_n$ and $\PB_n$, each of which is generated by its simplest elements, half-twists and squares of half-twists.

\begin{theorem}
\label{thm:2}
Let $n\geq 3$ and suppose that $G$ is a subgroup of $\B_n[4]$ that contains $\BI_n$.  Then $G$ does not have a generating set consisting entirely of even powers of Dehn twists about curves surrounding 2 points.  In particular $\B_n[4]$ is not generated by 4th powers of half-twists.
\end{theorem}

We have the following related question.

\begin{question}
Does $\B_n[4]$ have a generating set whose cardinality is equal to the dimension of $H_1(\B_n[4];\Q)$ (or even a generating set whose cardinality is a quartic polynomial in $n$)?
\end{question}

\begin{figure}[t]
\centering
\includegraphics[scale=.85]{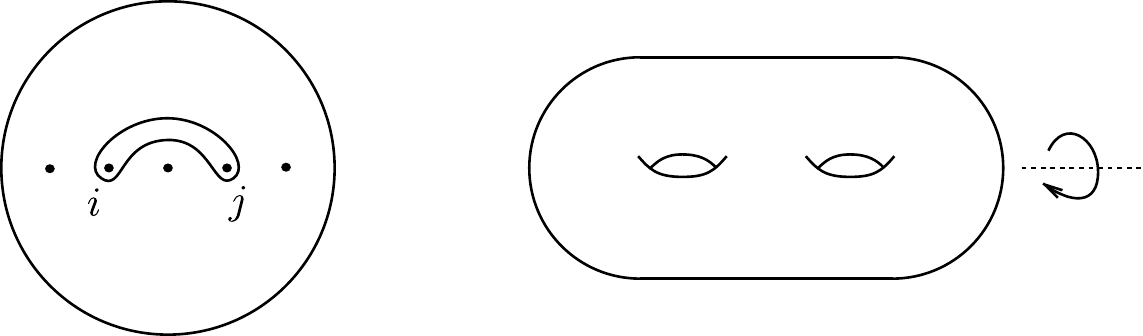}
\caption{\emph{Left:} $T_{ij}$ is the Dehn twist about the indicated curve; \emph{Right:} A hyperelliptic involution}
\label{fig:twist}
\end{figure}


\subsection*{Irreducible representations of \boldmath$\Z_n$} In order to state our main theorem, Theorem~\ref{thm:reps}, we must first describe the irreducible representations of the group $\Z_n = \B_n/\B_n[4]$.  The end result of our discussion here is Theorem~\ref{thm:class} below, which gives a naming system for the irreducible representations.

In order to state our classification of irreducible $\Z_n$-representations we require several definitions.  Let $[n]$ denote $\{1,\dots,n\}$ and let $[n]^{\underline 2}$ denote the set of unordered pairs of elements of $[n]$.  The standard action of $\B_n$ on $[n]$ (via the symmetric group $S_n$) induces an action on $[n]^{\underline 2}$.  We say that a subset $I$ of $[n]^{\underline 2}$ is full if the union of the elements of $I$ is $[n]$.  The symmetric group $S_n$ acts on the set of full subsets of $[n]^{\underline 2}$; let $\mathbb{I}_n$ be a set of orbit representatives.

Let $I \in \mathbb{I}_m$.  For $n \geq m$ we may regard $I$ as a subset of $[n]^{\underline 2}$.  We denote by $\B_n^I$ the stabilizer in $\B_n$ of the set $I$ and by $\Z_n^I$ the image in $\Z_n$.  We prove in Section~\ref{sec:proj} that there is a natural surjective map
\[
\Z_n^I \to \Z_m^I \times S_{n-m}.
\]
Next, let $\PZ_n$ denote the image of $\PB_n$ in $\Z_n$.  This group is isomorphic to $(\ZZ/2)^{n \choose 2}$, and the irreducible representations of $\PZ_n$ are in bijection with subsets of $[n]^{\underline 2}$; see Section~\ref{sec:iso}.  We denote the representation corresponding to $I \subseteq [n]^{\underline 2}$ by $V_I$.  We say that a representation of $\PZ_n$ is $I$-isotypic if it decomposes as a direct sum of copies of $V_I$.

We are ready to describe the irreducible representations of $\Z_n$ that appear in our classification.  The input for one of these representations consists of two pieces of data, an $I$-isotypic irreducible representation $\rho$ of some $\Z_m^I$ with $I \in \mathbb{I}_m$ and an irreducible representation of $S_{n-m}$; as usual we label the latter by its corresponding padded partition of $[n-m]$, call it $\lambda$.    With these in hand, we define a $\Z_n$-representation $V_n(\rho,\lambda)$ by the formula 
\[
V_n(\rho,\lambda) = \Ind_{\Z_n^I}^{\Z_n} \left(V_m(\rho) \boxtimes V_{n-m}(\lambda)\right)
\]
where $V_m(\rho)$ and $V_{n-m}(\lambda)$ are the representations corresponding to $\rho$ and $\lambda$ and $\Z_n^I$ acts via the surjection to $\Z_n^I \to \Z_m^I \times S_{n-m}$.  Of course if $\rho$ is isomorphic to $\rho'$ then $V_m(\rho,\lambda)$ is isomorphic to $V_m(\rho',\lambda)$.  In order to obtain a unique name for each representation, we fix one representative from each equivalence class once and for all.  

We observe that if we take $I=\emptyset$, then $V_n(\rho,\lambda)$ is the representation of $\Z_n$ that factors through the $S_n$-representation $V_n(\lambda)$.  We denote such a representation as $V_n(1,\lambda)$.  

\begin{theorem}
\label{thm:class}
The $V_n(\rho,\lambda)$ are irreducible $\Z_n$-representations.  Further, every irreducible $\Z_n$-representation is isomorphic to exactly one $V_n(\rho,\lambda)$.
\end{theorem}

The usual map $\B_n \to S_n$ induces a short exact sequence
\[
1 \to \PZ_n \to \Z_n \to S_n \to 1.
\]
If this sequence were split, we could hope to understand the representation theory of $\Z_n$ via the representation theory of $S_n$.  We prove, however, in Proposition~\ref{prop:nosplit} that it is not split.


\subsection*{Statement of the main theorem: representation stability}

The conjugation action of $\B_n$ on $\B_n[4]$ induces an action of $\B_n$ on $H_1(\B_n[4]; \C)$. Since this restricts to a trivial action of $\B_n[4]$, we have that $H_1(\B_n[4]; \C)$ is in a natural way a representation of $\Z_n$. 

Church and Farb defined representation stability for sequences of representations of $S_n~=~\B_n/\PB_n$. We extend their definition to our setting and show that the $H_1(\B_n[4];\C)$ are uniformly representation stable.

For each $n$ there is a standard inclusion $\B_n \to \B_{n+1}$ obtained by adding a strand.  We show in Lemma~\ref{lem:inc} that this induces inclusions $\B_n[4] \to \B_{n+1}[4]$ and $\Z_n \to \Z_{n+1}$.  Suppose we have a sequence of $\Z_n$-representations $V_n$ and maps $\varphi_n : V_n \to V_{n+1}$.  Following Church--Farb, we say that the sequence $\{V_n\}$ is consistent if for each $n$ the map $\varphi_n$ is equivariant with respect to the $\Z_n$-action.  The inclusions $\B_n[4] \to \B_{n+1}[4]$ induce maps $H_1(\B_n[4];\C) \to H_1(\B_{n+1}[4];\C)$.  With respect to these maps the $H_1(\B_n[4];\C)$ form a consistent sequence of $\Z_n$-representations (Lemma~\ref{lem:bn4cons}).

Further following Church--Farb, we say that a consistent sequence of $\Z_n$-representations $V_n$ satisfies \emph{representation stability} if
\begin{enumerate}
\item the maps $\varphi_n : V_n \to V_{n+1}$ are injective,
\item the span of the $\Z_{n+1}$-orbit of $\varphi_n(V_n)$ is equal to $V_{n+1}$, and
\item if we decompose each $V_n$ into irreducible $\Z_n$-representations 
\[
V_n \cong \bigoplus_{(\rho,\lambda)} c_{\rho,\lambda,n}V_n(\rho,\lambda)
\]
then each of the sequences of multiplicities $c_{\rho,\lambda,n} \geq 0$ is independent of $n$ for $n$ large.
\end{enumerate}
We say that the $V_n$ satisfy \emph{uniform representation stability} if  there is some $N$ so that every $c_{\rho,\lambda,n}$ is independent of $n$ for $n \geq N$.  

Let $I_3$ and $I_4$ be the subsets of $[3]^{\underline 2}$ and $[4]^{\underline 2}$ given by
\[
I_3 = \{\{1,3\},\{2,3\}\}, \quad \text{and} \quad I_4 = \{\{1,3\},\{2,3\}, \{1,4\},\{2,4\}\}.
\]  
We may assume that $I_3 \in \mathbb{I}_3$ and $I_4 \in \mathbb{I}_4$.  Let $\mu_2$ denote the multiplicative group $\{\pm 1\}$; we can regard $\mu_2$ as a subgroup of $\GL(\C)$.  In Section~\ref{sec:homs} we define specific homomorphisms $\rho_k : \Z_k^{I_k} \to \mu_2$, and hence representations of $\Z_k^{I_k}$ for $k \in \{3,4\}$.   Each $\rho_k$ is the sum of the winding numbers of the pairs of strands corresponding to the elements of $I_k$.

\begin{theorem}\label{thm:reps}
There are $\Z_n$-equivariant isomorphisms
\[
H_1(\B_n[4]; \C)\cong 
\begin{cases}
V_2(1,(0)) & n =2 \\[1ex]
V_3(1,(0)) \oplus V_3(1,(1)) \oplus V_3(\rho_3,(0)) & n=3 \\[1ex]
V_n(1,(0)) \oplus V_n(1,(1)) \oplus V_n(1,(2)) \oplus V_n(\rho_3,(0)) \oplus V_n(\rho_4,(0))  & n\geq 4.\\
\end{cases}
\]
Further, the sequence $\{H_1(\B_n[4]; \C)\}$ of $\Z_n$-modules is uniformly representation stable. 
\end{theorem}

The $\PZ_n$-invariants of $H_1(\B_n[4]; \C)$ is exactly $H_1(\PB_n;\C)$; this subspace corresponds to the summands $V_n(\rho,\lambda)$ in the statement of Theorem~\ref{thm:reps} with $\rho=1$.  Thus, the first statement of Theorem~\ref{thm:reps} recovers Church--Farb's description of $H_1(\PB_n;\C)$ as an $S_n$-representation.  

Also, representation stability for a sequence  $V_n$ of $\Z_n$-modules implies representation stability for the sequence of $\PZ_n$-invariants, which are $S_n$-modules.  In this way, the second statement of Theorem~\ref{thm:reps} recovers the representation stability of $H_1(\PB_n;\C)$  discovered by Church--Farb \cite{churchfarb}.

It appears to be an interesting problem to determine the character table of $\Z_n$.  

\subsection*{Level 4 hyperelliptic mapping class groups} Let $\Sigma_g$ be a closed orientable surface of genus $g$. Let $\Mod_g$ denote its mapping class group. This group is, for example, the (orbifold) fundamental group of the moduli space  $\M_g$ of Riemann surfaces of genus $g$.

The hyperelliptic mapping class group $\SMod_g$ is the centralizer in $\Mod_g$ of some fixed hyperelliptic involution; see Figure~\ref{fig:twist}.  The level $m$ mapping class group $\Mod_g[m]$ is the subgroup of $\Mod_g$ consisting of all elements that act trivially on $H_1(\Sigma_g; \ZZ/m)$.  The level $m$ hyperelliptic mapping class group $\SMod_g[m]$ is the intersection $\SMod_g \cap \Mod_g[m]$.  

As we will explain in Section~\ref{sec:ala}, there are isomorphisms $\B_{2g+1}[4]\cong \SMod_g[4] \times \ZZ$ for all $g \geq 1$.  Even in the absence of an inclusion $\Sigma_g \to \Sigma_{g+1}$, these isomorphisms give rise to maps $\SMod_g[4] \to \SMod_{g+1}[4]$ as follows:
\[
\SMod_g[4] \to \B_{2g+1}[4] \to \B_{2g+3}[4] \to \SMod_{g+1}[4],
\]
where the first map is inclusion into the first factor, the second map is the standard inclusion, and the third map is projection onto the first factor.  The induced maps $H_1(\SMod_g[4];\C) \to H_1(\SMod_{g+1}[4];\C)$ are injective and equivariant with respect to the $\Z_{2g+1}$- and $\Z_{2g+3}$-actions.  Since the $\ZZ$-factor of $\B_{2g+1}[4]$ is central, it follows that this factor corresponds to the trivial representation $V_{2g+1}(1,(0))$ in Theorem~\ref{thm:reps}.  We thus obtain the following consequence of Theorem~\ref{thm:reps}.

\begin{corollary}\label{cor:mainC}
For $g\geq 1$, there are $\Z_{2g+1}$-equivariant isomorphisms
\[
H_1(\SMod_g[4]; \C) \cong 
\begin{cases}
V_3(1,(1)) \oplus V_3(\rho_3,(0)) & g=1 \\[1ex]
V_{2g+1}(1,(1)) \oplus V_{2g+1}(1,(2)) \oplus V_{2g+1}(\rho_3,(0)) \oplus V_n(\rho_4,(0))  & g\geq 2.\\
\end{cases}
\]
In particular, we have
\[
\displaystyle \dim H_1(\SMod_g[4];\Q) = 3{2g+1\choose 4} + 3{2g+1\choose 3} + {2g+1 \choose 2} -1.
\]Further, the sequence $\{H_1(\SMod_g[4]; \C)\}$ of $\Z_{2g+1}$-modules is uniformly representation stable.  
\end{corollary}

In contrast to Corollary~\ref{cor:mainC}, $H_1(\Mod_g[m];\Q)=0$ for $g \geq 3$ and $m \geq 1$; see the paper by Hain \cite{hain2}.

Under the map $\B_{2g+1}[4] \to \SMod_g[4]$ from Section~\ref{sec:ala}, the basis elements from Theorem~\ref{thm:main} map to the classes of 4th powers of Dehn twists about nonseparating curves.

The group $\SMod_g$ is the orbifold fundamental group of the hyperelliptic locus $\H_g$ in  $\M_g$. The group $\SMod_g[4]$ is the fundamental group of any connected component $\H_g[4]$ of the hyperelliptic locus in the moduli space of  genus $g$ Riemann surfaces $C$ with level 4 structure,\footnote{There are $\displaystyle\frac{ 2^{g^2}(2^{2g}-1)\cdots (2^2-1)}{(2g+2)!}$ such components; they are mutually isomorphic.} i.e. a symplectic basis for $H_1(C; \ZZ/4)$. In fact, $\H_g[4]$ is a $K(\pi,1)$ space for $\SMod_g[4]$. Thus $H^i(\SMod_g[4];\Q) \cong H^i(\H_g[4];\Q)$ for all $j\geq 0$. As such, Corollary~\ref{cor:mainC} gives the first Betti number of $\H_g[4]$.  

For $g=2$ we have $\SMod_g = \Mod_g$.  In this case we have the following result.

\begin{proposition}\label{prop:genus2prop}
We have $\dim H_3(\Mod_2[4]; \Q) \geq 3068$.
\end{proposition}

Proposition~\ref{prop:genus2prop} improves on a special case of a result of Fullarton--Putman \cite[Theorem A]{fp}, which gives $\dim H_3(\Mod_2[4]; \Q) \geq 24$.


\subsection*{Albanese cohomology}
For a finitely generated group $\Gamma$, the \emph{Albanese cohomology} of $\Gamma$ is the subalgebra $H^*_{Alb}(\Gamma;\Q)$ of the rational cohomology algebra $H^*(\Gamma;\Q)$ generated by $H^1(\Gamma;\Q)$. In other words, $H^*_{Alb}(\Gamma;\Q)$ is the image of the cup product mapping
\[
\Lambda^*H^1(\Gamma;\Q)\rightarrow H^*(\Gamma;\Q).
\]
The term ``Albanese cohomology" was introduced by Church--Ellenberg--Farb in their work on representation stability \cite{cef1}. 

Arnol'd \cite{arnold} showed that the cohomology ring of $\PB_n$ is generated by degree 1 classes, in other words $H^*_{Alb}(\PB_n;\Q) = H^*(\PB_n;\Q)$. We prove the following contrasting result. 

\begin{theorem}\label{thm:mainD}
Let $n\geq 15$.  Then $H^*_{Alb}(\B_n[4];\Q)$ is a proper subalgebra of $H^*(\B_n[4];\Q)$.
\end{theorem}

Although Theorem~\ref{thm:mainD} asserts that there are cohomology classes in $H^*(\B_n[4];\Q)$ that are not cup products of classes in $H^1(\B_n[4];\Q)$, our proof does not produce examples of such classes. 
We will derive Theorem~\ref{thm:mainD} from the following slightly stronger result. 

\begin{theorem}\label{thm:mainE}
 For all $g\geq 7$ the Albanese cohomology $H^*_{Alb}(\emph{SMod}_g[4]; \Q)$ is a proper subalgebra of $H^*(\emph{SMod}_g[4]; \Q)$.
\end{theorem}

The original proofs that the cohomology groups of the pure braid groups and the pure string motion groups are representation stable take advantage of the fact that the cohomology algebras of both are generated in degree 1; see the papers by Church--Farb and Wilson \cite{churchfarb,wilson}. One would like to emulate this in the case of $\B_n[4]$ to prove that its higher cohomology groups of $\B_n[4]$ are representation stable. However, as Theorem~\ref{thm:mainD} shows, this approach cannot work. Instead, we propose the following. 

\begin{conjecture}\label{conjalb}
For each $k\geq 1$ the sequence of $\Z_n$-representations $H^k_{Alb}(\B_n[4]; \Q)$ is uniformly representation stable for $n\geq 4k$. 
\end{conjecture}

Conjecture~\ref{conjalb} may be compared with a result of Church--Ellenberg--Farb, which states that the Albanese cohomology of the Torelli group is a finitely generated FI-module, hence uniformly representation stable \cite[proof of Theorem 7.2.2]{cef1}.  


\subsection*{Hyperelliptic Torelli groups}

The \emph{braid Torelli group} $\BI_n$ is the level 0 subgroup of $\B_n$, i.e. the kernel of the Burau representation evaluated at $t=-1$ (the latter is sometimes called the integral Burau representation). This group, an infinite-index subgroup of $\B_n$, is even more mysterious than the $\B_n[m]$. For example, it is not known if this group is finitely generated when $n\geq 7$. It is, however, known that for $n = 2,3,4,5,6$ have that that $\BI_n$ is isomorphic to 1, $\ZZ$, $F_\infty$, $F_\infty \times \ZZ$, and $F_\infty \ltimes F_\infty$, respectively, where $F_\infty$ denotes the free group of countably infinite rank.  It is also known that $\BI_7$ is not finitely presented \cite[Theorem 1.3]{bcm}.  An appealing infinite generating set for $\BI_n$, consisting of all squares of Dehn twists about curves surrounding either 3 or 5 punctures, was identified by Brendle, Putman, and the second author \cite{brendlemargalitputman}.

The \emph{hyperelliptic Torelli group} $\SI_g$ is the subgroup of $\SMod_g$ whose elements act trivially on $H_1(\Sigma_g;\ZZ)$. There are isomorphisms $\BI_{2g+1}\cong \SI_g\times \ZZ$; see  \cite{pointpushing}.  Using Hain's description \cite{haininfpresentations} of the image of the second Johnson homomorphism, one can obtain a lower bound
\[
\displaystyle \dim H_1(\BI_{2g+1};\Q) \geq \frac{g(g-1)(4g^2+4g-3)}{3}+1.
\]
We will deduce an improved lower bound from Theorem~\ref{thm:main} as follows.

\begin{theorem}\label{thm:mainF}
For $g\geq 3$ we have
\[
\dim H_1(\BI_{2g+1};\Q)\geq \frac{1}{6}\left(20g^4+12g^3-5g^2+9g\right)
\]
\end{theorem}

The finiteness properties of the $\BI_n$ and the $\B_n[m]$ are related by the following proposition, which we prove in Section~\ref{hyptorellisection}.

\begin{proposition}
\label{prop:unbound}
Let $n$ be odd.  If the sequence $\left(\dim H_1(\B_n[m];\Q)\right)_{m=1}^{\infty}$ is unbounded then $H_1(\BI_n;\Q)$ is infinite dimensional, and in particular $\BI_n$ is not finitely generated.
\end{proposition}

We make the following conjecture.

\begin{conjecture}\label{conj:fg}
For fixed $n \geq 4$, the sequence $\left(\dim H_1(\B_n[m];\Q)\right)_{m=1}^{\infty}$ is unbounded.  In particular, $H_1(\BI_n;\Q)$ is infinite dimensional and $\BI_n$ is not finitely generated.
\end{conjecture}

Theorem~\ref{thm:main} provides some evidence for Conjecture~\ref{conj:fg}; with the equalities
\[
\dim H_1(\B_n;\Q) = 1\ \ \ \  \dim H_1(\PB_n;\Q) = {n \choose 2}
\]
it implies
\[
\dim H_1(\B_n;\Q) \ll \dim H_1(\B_n[2];\Q) \ll \dim H_1(\B_n[4];\Q).
\]
From the above descriptions of $\BI_n$ for $2 \leq n \leq 6$ we have that $H_1(\BI_n;\Q)$ is not finite dimensional for $4 \leq n \leq 6$.  Brendle, Childers, and the second author showed that $H_{g-1}(\BI_{2g+1};\Q)$ is infinite dimensional \cite[Theorem 1.3]{bcm}.  It is not known if any of the other $H_{k}(\BI_{n};\Q)$ are infinite dimensional.

The group $\BI_{2g+1}$ is isomorphic to the direct product of $\ZZ$ with fundamental group of any component of the branch locus of the period map on Torelli space \cite{hainsurvey}.  Therefore Conjecture~\ref{conj:fg} implies that this fundamental group is infinitely generated.


\subsection*{Torsion points on the characteristic variety for the braid arrangement}

Let $X_n\subset \C^n$ denote the complement of the braid arrangement
\[
X_n = \C^n \setminus  \{
(x_1,\dots,x_n) \in \C^n \mid x_i = x_j \text{ some } i,j
\}
\]
There is a one-to-one correspondence between homomorphisms $\PB_n \cong \pi_1(X_n)\rightarrow \C^{\times}$ and 1-dimensional complex local systems over $X_n$; to a homomorphism $\rho: \PB_n\rightarrow \C^{\times}$ we associate the 1-dimensional local system $\C_{\rho}$ with monodromy $\rho$. Thus the space of all such local systems can be identified with the algebraic torus $\Hom(\PB_n, \C^{\times})$, which we identify with $(\C^{\times})^{n \choose 2}$ via
$\rho \to \left(\rho(T_{12}),\ldots , \rho(T_{n-1,n})\right)$.

For $d\geq 1$, the $d$th characteristic variety $V_d(X_n)$ of $X_n$ is the subvariety of $(\C^{\times})^{n \choose 2}$ consisting of all $\rho\in \Hom(\PB_n, \C^{\times})$ such that $\dim H^1(X_n; \C_{\rho})\geq d$. One reason for studying the characteristic varieties of $X_n$ is that they give fine information about the topology of its abelian covers. A general theorem of Arapura~\cite{arapura} implies that $V_d(X_n)$ is a union of algebraic subtori, possibly with some components translated away from the identity ${\bf 1}$ by finite-order elements. Following Cohen--Suciu \cite{cohensuciu}, we denote by $\check{V}_d(X_n)$ the union of the components of $V_d(X_n)$ that contain {\bf 1} (these are the so-called central components). 

Let $\rho_I: \PZ_n\rightarrow \mu_2$ denote the homomorphism giving the irreducible representation $V_I$ of $\PZ_n$ (see Section~\ref{sec:bnbar}) and denote the unique extension $\PB_n\rightarrow \mu_2\subset \C^{\times}$ by $\rho_I$ as well.  

\begin{theorem}\label{thm:2torsion}
Let $n\geq 3$. For $d\geq 2$, the characteristic variety $V_d(X_n)$ contains no 2-torsion. The set of 2-torsion points on $V_1(X_n)$ is $\{\rho_{g(I_3)}\}_{g\in S_n}\cup \{\rho_{g(I_4)}\}_{g\in S_n}$, which is contained in $\check{V}_1(X_n)$.
\end{theorem}

For $n \leq 4$, all known components of the characteristic varieties of $X_n$ contain {\bf 1} but it is an open problem to determine whether this holds for general $n$. Arapura's theorem implies that any component of $V_d(X_n)$ not containing {\bf 1} must contain some point of finite order. Therefore, if one could show that every torsion point on $V_d(X_n)$ were contained in $\check{V}_d(X_n)$, it would follow that $\check{V}_d(X_n) = V_d(X_n)$.  By Theorem~\ref{thm:2torsion}, any translated components of $V_d(X_n)$, should they exist, would have to be translated by an element of order at least 3.


\subsection*{Outline of the paper}

The remainder of the paper essentially has three parts.  The first part, Sections~\ref{doublecoverssection}--\ref{upperboundsection}, is devoted to the proof of Theorem~\ref{thm:main}.  In Section~\ref{doublecoverssection} we define maps 
\begin{align*}
\psi &: \B_n[4]\rightarrow \ZZ^{{n \choose 2}}, \\
\psi_{i\infty} &: \B_n[4]\rightarrow \ZZ^{{2n-1 \choose 2}}  \quad 1 \leq i \leq n, \text{ and}\\
\psi_{ij} &: \B_n[4]\rightarrow \ZZ^{{2n-2 \choose 2}} \quad 1 \leq i < j \leq n
\end{align*}
that we will use to detect nontrivial classes in $H_1(\B_n[4];\Q)$.  The map $\psi$ is simply the restriction of the abelianization of $\PB_n$.  This map is clearly not sufficient for our purposes, since the dimension of $H_1(\B_n[4];\Q)$ asserted by Theorem~\ref{thm:main} is much larger than that of $H_1(\PB_n;\Q) = {n \choose 2}$.  The $\psi_{ij}$ are designed to detect elements of $H_1(\B_n[4];\Q)$ coming from the commutator subgroup of $\PB_n$.  Roughly, they are defined as follows: lift an element of $\B_n[4]$ to a double cover of $\D_n$ and apply the abelianization of the pure braid group of the cover.  We show at the end of Section~\ref{doublecoverssection} that these maps do indeed detect commutators of pure braids.

With the $\psi_{ij}$ in hand, we complete the proof of Theorem~\ref{thm:main} for $n=3$ in Section~\ref{3strandsection}.  That the given basis elements are independent is proved using the $\psi_{ij}$ and the fact that the first Betti number is 6, which comes easily from the equality $\B_3[4] = \PB_3^2$ and the isomorphism $\PB_3 \cong F_2 \times \ZZ$.  

The proof of Theorem~\ref{thm:main} for $n \geq 4$ is carried out in the next three sections.  In Section~\ref{4punctureslowerbound} we use the $\psi_{ij}$ to show that the basis elements from Theorem~\ref{thm:main} are linearly independent; this step is similar in spirit to the $n=3$ case, and is complicated mainly by the large number of homology classes being considered.  Then in Section~\ref{spanning} we give an infinite spanning set for $H_1(\B_n[4];\Q)$ whose elements have a particularly simple form: they are the images of squares of Dehn twists about curves surrounding two marked points.  There is an obvious spanning set for $H_1(\B_n[4];\Q)$ coming from the generating set for $\B_n[4]$ given by Brendle and the second author, namely the squares of Dehn twists, and our spanning set is a subset of this one.  Finally, in Section~\ref{upperboundsection} we complete the proof of Theorem~\ref{thm:main}.  This section is the technical heart of the proof.  The idea is to whittle down the spanning set from Section~\ref{spanning} using a series of relations in $H_1(\B_n[4];\Q)$.  The relations are obtained using a combination of the squared lantern relation, a Jacobi identity, the Witt--Hall identity, and the standard Artin relations for $\PB_n$. 

The second part of the paper is dedicated to the proof of Theorem~\ref{thm:reps}.  As above, this first requires an investigation of the representation theory of $\Z_n$.  In particular, the statement of representation stability requires a naming system for its irreducible representations.  This is carried out in Section~\ref{sec:reppf}.  The main difficulty stems from the fact that $\Z_n$ does not split as a semi-direct product over $S_n$.  We then prove Theorem~\ref{thm:reps} in Section~\ref{sec:urs} by exhibiting the irreducible representations from the statement of the theorem as explicit submodules and by verifying the three parts of the definition of uniform representation stability.  These submodules are the spans of the $\Z_n$-orbits of elements $a_{ij}$, $x_3$, and $x_4$.  The main obstacle towards proving Theorem~\ref{thm:reps} is simply locating the elements $x_3$ and $x_4$ in the first place.  

Finally, the third part of the paper gives the proofs of the various applications of our main results.  First in Section~\ref{sec:nongen} we quickly dispense with Theorem~\ref{thm:2} as a consequence of Theorem~\ref{thm:main}.  Then in Section~\ref{sec:alb} we Theorem~\ref{thm:mainE} and then use this to prove Theorem~\ref{thm:mainD}.  For Theorem~\ref{thm:mainE}, the basic idea is to compare the dimension of $H^1(\SMod_g[4];\Q)$ (Corollary~\ref{cor:mainC}) to the Euler characteristic of $\SMod_g[4]$.  The latter is an enormous negative number, signaling the presence of large amounts of cohomology in odd degrees. A careful comparison of the odd Betti numbers of $\SMod_g[4]$ with the dimensions of the odd graded pieces of the exterior algebra $\Lambda^*H^1(\SMod_g[4];\Q)$ then gives the result.  At the end of Section~\ref{sec:alb} we prove Proposition~\ref{prop:genus2prop} and Theorem~\ref{thm:smallbetti}.

Next, in Section~\ref{hyptorellisection} we prove Theorem~\ref{thm:mainF}.  The idea is to show that there is a surjective map $H_1(\BI_{2g+1};\Q) \to H_1(\SMod_g[4];\Q)$ and that the direct sum of this map with the second Johnson homomorphism is surjective.  The result is then obtained by adding together the dimensions of the targets of these two maps.  Finally, in Section~\ref{sec:tors} we prove Theorem~\ref{thm:2torsion} as an application of Theorem~\ref{thm:reps}.

\subsection*{Acknowledgments}
The authors would like to thank Lei Chen, Jordan Ellenberg, and Benson Farb for encouraging us to examine the role of representation stability in the homology of the level 4 braid group.  We would also like to thank Santana Afton, Weiyan Chen, Nir Gadish, Marissa Loving, Jeremy Miller, Christopher O'Neill, Peter Patzt, Andrew Putman, and Oscar Randal-Williams for helpful comments and conversations.  This material is based upon work supported by the National Science Foundation under Grant No. DMS - 1057874.


\section{Abelian quotients from double covers}\label{doublecoverssection}

The goal of this section is to define and describe the homomorphisms
\begin{align*}
\psi &: \B_n[4]\rightarrow \ZZ^{{n \choose 2}}, \\
\psi_{i\infty} &: \B_n[4]\rightarrow \ZZ^{{2n-1 \choose 2}}  \quad 1 \leq i \leq n, \text{ and}\\
\psi_{ij} &: \B_n[4]\rightarrow \ZZ^{{2n-2 \choose 2}} \quad 1 \leq i < j \leq n
\end{align*}
discussed in the introduction.  We will denote the induced maps on $H_1(\B_n[4];\Q)$ by the same symbols.  We will use these homomorphisms in Sections~\ref{3strandsection} and~\ref{4punctureslowerbound} to detect non-zero homology classes in $H_1(\B_n[4]; \Q)$.  

The map $\psi$ will simply be defined as the restriction of the abelianization of $\PB_n$.   As discussed in the introduction, the $\psi_{ij}$ will be defined in terms of 2-fold covers of $\D_n$.  In Section~\ref{sec:covers} we describe the 2-fold covers used.  Then in Section~\ref{sec:psi def} we define the $\psi_{ij}$ and in Section~\ref{sec:psi comp} we compute the images under the $\psi_{ij}$ of each square of a Dehn twist in $\B_n[4]$.  Finally we in Section~\ref{sec:nat} we give a naturality (equivariance) formula for the $\psi_{ij}$ and use this formula to compute several examples.  

One of the examples we compute at the end of the section is $\psi_{1\infty}([T_{23}^2,T_{12}])$.   In particular we show it is non-zero.  Of course $\psi$ evaluates to zero on the commutator subgroup of $\PB_n$, and so this computation verifies that the $\psi_{ij}$ are indeed giving more information than $\psi$.

\subsection{Double covers of the disk.}\label{sec:covers} Denote the set of marked points of $\D_n$ by $P$.  There is a correspondence
\[
H_1(\D_n,P \cup \partial \D_n;\ZZ/2)  \longleftrightarrow  \{\text{2-fold branched covers of } (\D_n,P) \}/\sim
\]
Here, a branched cover over $(\D_n,P)$ is a branched cover over $\D_n$ where the set of branch points lies in $P$.

The above correspondence can be explained as a sequence of three correspondences as follows.  First, it is a consequence of Lefschetz duality that $H_1(\D_n,P \cup \partial \D_n;\ZZ/2)$ is isomorphic to $H^1(\D_n^\circ;\ZZ/2)$, where $\D_n^\circ$ is the surface obtained from $\D_n$ by removing the $n$ marked points.  Second, by basic covering space theory, the latter is in bijective correspondence with the equivalence classes of 2-fold covers of $\D_n^\circ$.  Third, 2-fold covers over $\D_n^\circ$ are in bijection with 2-fold branched covers of $\D_n$ with branch set in $P$; we pass from one to the other by adding/subtracting the points of $P$ and its preimage.  The stated correspondence follows.

As above, let $[n]$ denote the set $\{1,\dots,n\}$ and let $[n]^{\underline 2}$ denote the set of pairs of elements of $[n]$.  Also let $[n]_\infty$ denote $[n] \cup \{\infty\}$ and let $[n]_\infty^{\underline 2}$ be the set of pairs of elements of $[n]_\infty$.  There is a natural bijection between $[n]$ and $P$, where $i$ corresponds to the $i$th marked point.  If we think of $\partial \D_n$ has having the label $\infty$ then there is a further bijection between $[n]_\infty$ and the set of connected components of $P \cup \partial \D_n$.  There is a map
\[
\{ [n]_\infty^{\underline 2} \} \to \{\text{2-fold branched covers of } (\D_n,P) \}/\sim
\]
defined as follows.  For an element of $[n]_\infty^{\underline 2}$ we obtain a nontrivial element of $H_1(\D_n,P \cup \partial \D_n;\ZZ/2)$ by choosing an arc between the corresponding components of $P \cup \partial \D_n$ (the arc should be disjoint from $P \cup \partial \D_n$ on its interior).  This homology class, hence the resulting equivalence class of covers, is independent of the choice of arc.  We refer to any resulting cover of $\D_n$ as an $(ij)$-cover of $\D_n$.  An $(i\infty)$-cover of $\D_n$ is a disk with $2n-1$ marked points and any other $(ij)$-cover is an annulus with $2n-2$ marked points.  

As elements of $[n]_\infty^{\underline 2}$ only give equivalence classes of branched covers over $\D_n$, it will be helpful to fix specific $(ij)$-covers once and for all, as follows.  First we fix a copy of $\D_n$ once and for all, as the closed unit disk in the plane with the marked points along the $x$-axis.

Then for each $i$ we let $\alpha_{i\infty}$ be the vertical arc in $\D_n$ connecting the $i$th marked point to the upper boundary of $\D_n$.  And for each $\{i,j\} \in [n]^{\underline 2}$ we let $\alpha_{ij}$ be the semi-circular arc that connects the $i$th and $j$th marked points and lies above the $x$-axis.  

We construct the specific $(ij)$-covers by taking two copies of $\D_n$, cutting each along the corresponding $\alpha_{ij}$-arc and then gluing the two cut disks together.  Each cut disk corresponds to a fundamental domain for the deck group.   We think of the $\alpha_{ij}$ as branch cuts.  

In the $(i\infty)$-cover $\tilde \D_n$, each marked point of $\D_n$ has two pre-images except for the $i$th, which has one.  We label the preimage in $\tilde \D_n$ of the $i$th marked point with $i$.  For each $j \neq i$ we label the preimages of the $j$th marked point in the first and second fundamental domains of $\tilde \D_n$ with $j$ and $j'$, respectively.  For an $(ij)$-cover with $j \neq \infty$ the marked points in $\tilde \D_n$ are labeled similarly.  We denote by $[n]'$ the set of symbols $\{1',\dots,n'\}$.  So the labels of the marked points of $\tilde \D_n$ lie in $[n] \cup [n]'$.

We remark that the cover $\tilde \D_n$, and hence the labeling of its marked points, is sensitive to the homotopy class of each $\alpha_{ij}$, not just the corresponding class in $H_1(\D_n,P \cup \partial \D_n;\ZZ/2)$.


\subsection{Homomorphisms from double covers}\label{sec:psi def} Our next goal is to define $\psi$ and the $\psi_{ij}$.  The $\psi_{ij}$ will be defined as follows: given an element of $\B_n[4]$, lift it to the corresponding cover $\tilde \D_n$, and then take an abelian quotient of the pure mapping class group $\PMod(\tilde \D_n)$.  In general, for a surface $S$ with marked points, $\PMod(S)$ is the subgroup of $\Mod(S)$ given by the kernel of the action on the set of marked points.  

We will carry out the plan described in the previous paragraph by explicitly describing the lifting maps and the abelian quotients.  The latter will be aided by another homomorphism, called the capping homomorphism, which we define along the way.

\p{Lifting} We begin with the lifting homomorphism.  Consider an element of $[n]_\infty^{\underline 2}$ and let $\tilde \D_n$ denote the corresponding branched cover.  There is a homomorphism
\[
\Lift : \PB_n \to \Mod(\tilde \D_n),
\]
defined as follows.  Let $f \in \PB_n \cong \PMod(\D_n)$ and let $\phi : \D_n \to \D_n$ be a representative homeomorphism fixing the boundary.  Because $f$ lies in $\PB_n$ it fixes the element of $H_1(\D_n,P \cup \partial \D_n;\ZZ/2)$ corresponding to $\tilde \D_n$.  Hence $\phi$ lifts to a homeomorphism of $\tilde \D_n$.  There is a unique lift that induces the identity map on $\partial \tilde \D_n$;   let $\tilde f$ be the corresponding element of $\Mod(\tilde \D_n)$.    Then $\Lift$ is defined by
\[
\Lift(f) = \tilde f.
\]

\begin{lemma}
\label{lem:lift}
Let $n \geq 2$, let $\{i,j\} \subset [n]_\infty$, and let $\tilde \D_n$ denote the corresponding branched cover of $\D_n$.  The lifting homomorphism $\Lift : \PB_n \to \Mod(\tilde \D_n)$ restricts to a homomorphism
\[
\B_n[4] \to \PMod(\tilde \D_n).
\]
\end{lemma}

\begin{proof}

As mentioned in the introduction, it is a theorem of Brendle and the second author that $\B_n[4]$ is equal to the subgroup of $\PB_n$ generated by squares of Dehn twists.  Therefore, it is enough to show that if $c$ is a simple closed curve in $\D_n$ then $\Lift(T_c^2)$ lies in $\PMod(\tilde \D_n)$.

The preimage $\tilde c$ in $\tilde \D_n$ of a simple closed curve $c$ in $\D_n$ is a 2-fold cover of $c$.  In particular it has one or two components.  In the first case, $T_c^2$ lifts to the Dehn twist about $\tilde c$.  In the second case, $T_c^2$ lifts to the product of the squares of the Dehn twists about the components of $\tilde c$.  In both cases, the lift lies in $\PMod(\tilde \D_n)$, as desired.
\end{proof}

\p{Capping} We now proceed to the capping homomorphism.  Let $S$ be a surface with boundary.  We choose one distinguished component of the boundary of $S$.  Let $\hat S$ be the surface obtained by gluing a disk to this component.  There is a homomorphism 
\[
\Capmap : \Mod(S) \to \Mod(\hat S),
\]
defined as follows: given an element $f$ of $\Mod(S)$ we can represent it by a homeomorphism of $S$ that fixes the boundary, and then extend this homeomorphism to $\hat S$ in such a way that the extension is the identity on the complement of $S$.  The resulting mapping class is $\Capmap(f)$.  

\p{The abelianization of the pure braid group} The abelianization of the pure braid group is:
\[
\PB_n/[\PB_n,\PB_n] \cong H_1(\PB_n;\ZZ) \cong \ZZ^{n \choose 2}.
\]
The abelianization map
\[
\PB_n \to \ZZ^{n \choose 2}
\]
can be described as follows.  There are $n \choose 2$ forgetful homomorphisms
\[
\PB_n \to \PB_2 \cong \ZZ
\] 
obtained by forgetting all but two of the marked points in $\D_n$, and the abelianization of $\PB_n$ is the direct sum of these homomorphisms.

We now give a slightly different, and more natural, description of the abelianization of $\PB_n$.  We will denote the element $\{i,j\}$ of $[n]^{\underline 2}$ by $(ij)$ (so $(ji)=(ij)$).   Let $\ZZ\{ (ij) \}_{i<j}$ denote the free abelian group on the set of elements of $[n]^{\underline 2}$.  We can write the abelianization of $\PB_n$ as 
\[
\PB_n\rightarrow \ZZ\{ (ij)\}_{i<j},
\]
where the $(ij)$-factor of $\ZZ\{ (ij) \}_{i<j}$ corresponds to the map $\PB_n \to \PB_2$ where all marked points except the $i$th and the $j$th are forgotten.  This notation will be especially useful when the labels of the marked points are not natural numbers, as is the case in our $(i\infty)$- and $(ij)$-covers.

For any subset $A = \{i_1, i_2, \cdots, i_k\}$ of $\{1,\dots,n\}$ there is an associated element of the free abelian group $\ZZ\{ (ij)\}_{i<j}$.  This element may be denoted by $(A)$ or $(i_1 i_2 \cdots i_k)$ and it is defined as
\[
(i_1 i_2 \cdots i_k) = \sum_{p < q} (i_p i_q).
\]
For example, $(123) = (12) + (13) + (23)$.  In this notation it makes sense to interpret $(\emptyset)$ as the identity.  This language makes it convenient to describe the image of a Dehn twist under the abelianization of $\PB_n$: if $c$ is a simple closed curve in $\D_n$ surrounding the marked points $\{i_1,\dots,i_k\}$, then
\[
T_c \mapsto (i_1 i_2 \cdots i_k).
\]
The braid group $\B_n$ acts on $\PB_n$, hence its abelianization, by conjugation.  The action of a particular element $f \in \B_n$ depends only on its image in the quotient $\B_n/\PB_n$, which is isomorphic to the symmetric group on $[n]$.  If the image of $f$ is the permutation $\sigma$ then the action $f_*$ on $\ZZ\{ (ij)\}_{i<j}$ is given by $f_*(ij) = (\sigma(i)\sigma(j))$.

\p{The definition of $\psi$} As advertised, we define 
\[
\psi : \B_n[4] \to \ZZ\{ (ij)\}_{i<j} \cong \ZZ^{n \choose 2}
\]
as simply the restriction of the abelianization of $\PB_n$.  From our description of the latter we immediately obtain a formula for the image of the square of an arbitrary Dehn twist under the map $\psi$: if $c$ is a simple closed curve in $\D_n$ and $A$ is the set of labels of marked points in the interior of $c$, then 
\[
\psi(T_c^2) = 2 (A).
\] 
If $f \in \B_n$ maps to $f_*$ in the symmetric group on $[n]$ then
\[
\psi(f \cdot T_c^2) = 2(f_*(A)).
\]
Our goal in the remainder of this section is to define the $\psi_{ij}$ and obtain similar formulas for the image of a square of a Dehn twist.

\p{The definitions of the $\psi_{ij}$} We are finally ready to define the $\psi_{ij}$.  First, for $i \in [n]$ we define $\psi_{i\infty}$ as the composition
\[
\psi_{i\infty} : \B_n[4] \stackrel{\Lift}{\to} \PMod(\tilde \D_n) \to \ZZ\{ (k\ell)\}_{\{k,\ell\} \subseteq L_{i\infty}} \cong \ZZ^{2n-1 \choose 2},
\]
where $\tilde \D_n$ is the $(i\infty)$-cover of $\D_n$ and $L_{i\infty}$ is the set of labels of the marked points of $\tilde \D_n$.  The existence of the first map is ensured by Lemma~\ref{lem:lift}.  The second map is the abelianization.  The isomorphism at the end comes from the fact that the $(i\infty)$-cover is a disk with $2n-1$ marked points, with $L_{i\infty} = [n] \cup [n]' \setminus i'$.

For $\{i,j\} \in [n]^{\underline 2}$ we denote by $\tilde \D_n$ the $(ij)$-cover of $\D_n$ and by $L_{ij}$ the corresponding set of marked points.  Then we define $\psi_{ij}$ in the analogous way.  The only difference is that we must apply the capping homomorphism (as above, we cap the component of the boundary lying in the second fundamental domain, where the marked points are labeled with primed numbers):
\[
\psi_{ij} : \B_n[4] \stackrel{\Lift}{\to} \PMod(\tilde \D_n) \stackrel{\Capmap}{\to} \PB_{2n-2} \to \ZZ\{ (k\ell)\}_{\{k,\ell\} \subseteq L_{ij}} \cong \ZZ^{2n-2 \choose 2}.
\]

\p{The induced maps} Since $\psi$ and the $\psi_{ij}$ are maps to abelian groups, they also induce maps 
\begin{align*}
\psi &: H_1(\B_n[4];\Q) \rightarrow \Q^{{n \choose 2}}, \\
\psi_{i\infty} &: H_1(\B_n[4];\Q) \rightarrow \Q^{{2n-1 \choose 2}}  \quad 1 \leq i \leq n, \text{ and}\\
\psi_{ij} &: H_1(\B_n[4];\Q) \rightarrow \Q^{{2n-2 \choose 2}} \quad 1 \leq i < j \leq n
\end{align*}
As shown here, we refer to the induced maps by the same symbols as the original maps.

\subsection{Computations}\label{sec:psi comp} Our next goal is to describe the images under the $\psi_{ij}$ of a square of a Dehn twist, first for squares of Artin generators and then for arbitrary squares of Dehn twists.  The statement of the first lemma requires some notation.  Let $\{i,j\}$ be an element of $[n]_\infty^{\underline 2}$ with $i < j$ and let $\{k,\ell\}$ be an element of $[n]^{\underline 2}$ with $k < \ell$.  We say that $\{i,j\}$ and $\{k,\ell\}$ are \emph{linked} if
\[
i < k < j < \ell \quad \text{ or } \quad k < i < \ell < j
\]
and we say that they are \emph{unlinked} if
\[
i < k < \ell < j \quad \text{ or } \quad k < i < j < \ell.
\]
Also, we write $[n\setminus k,\ell]$ for $[n] \setminus \{k,\ell\}$.  

\begin{figure}
\includegraphics[scale=1]{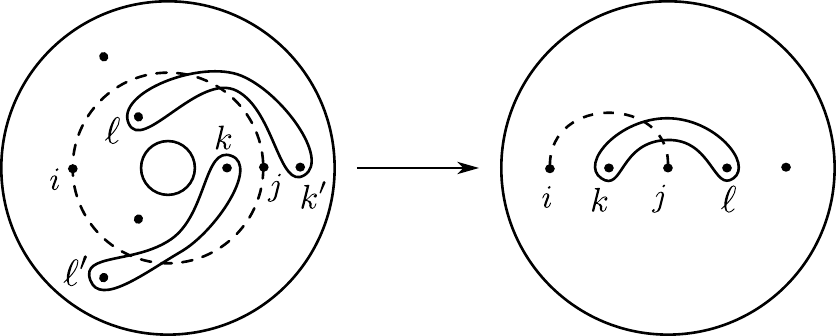}
\caption{A representative case for the proof of Lemma~\ref{lem:psi Artin}}
\label{fig:psi artin}
\end{figure}

\begin{lemma}
\label{lem:psi Artin}
Let $k,\ell \in [n]$ with $k < \ell$ and let $i,j \in [n]_\infty$ with $i < j$.  Then
\[
\psi_{ij}(T_{k\ell}^2) = 
\begin{cases}
(k\ell\ell') &  \{i,j\} \cap \{k,\ell\} = \{k\} \\
(kk'\ell) &  \{i,j\} \cap \{k,\ell\} = \{\ell\} \\
2(k\ell) + 2(k'\ell') & \{i,j\},\{k,\ell\} \text{ unlinked} \\
2(k\ell') + 2(k'\ell) & \{i,j\},\{k,\ell\} \text{ linked} \\
2\left(\{k,\ell\} \cup [n\setminus k,\ell ]'\right)  + 2([n\setminus k,\ell ]') &  \{i,j\}=\{k,\ell\} \\
\end{cases}
\]
\end{lemma}

\begin{proof}

Let $c$ be the curve in $\D_n$ corresponding to the Artin generator $T_{k\ell}$.  As discussed in the proof of Lemma~\ref{lem:lift}, the preimage in $\tilde \D_n$ is either a single curve $\tilde c$ or a pair of curves $\tilde c_1, \tilde c_2$.  In the first case $\Lift(T_c^2)$ is equal to $T_{\tilde c}$ and in the second case it is equal to $T_{\tilde c_1}^2 T_{\tilde c_1}^2$.

By the way the $\psi_{ij}$ are defined, and because we already have a formula for the image of a Dehn twist in the abelianization the pure braid group, it remains to determine the preimage of $c$ in each case.  In fact, the only relevant feature of the preimage of $c$ is the set of marked points that it surrounds.

There are nine cases, as the fifth case of the lemma only makes sense for $(ij)$-covers with $i,j \neq \infty$.  A representative picture for the case where $j \neq \infty$ and $\{i,j\}$ and $\{k,\ell\}$ are linked is shown in Figure~\ref{fig:psi artin}.  The curve $c$ is the boundary of a regular neighborhood of the arc $\alpha = \alpha_{k\ell}$.  Therefore the preimage of $c$ is the boundary of a regular neighborhood of the preimage of $\alpha$.  The path lift of $\alpha$ starting at the marked point $k$ crosses the preimage of $\alpha_{ij}$ and ends at $\ell'$; this path lift is hence a connected component of the preimage of $\alpha$.  Similarly, the other component of the preimage of $\alpha$ is an arc connecting the marked points $k'$ and $\ell$.  So the preimage of $c$ is a pair of curves, one surrounding the marked points $k$ and $\ell'$ and one surrounding the marked points $k'$ and $\ell$.  It follows that $\psi_{ij}(T_{k\ell}^2) = 2(k\ell') + 2(k'\ell)$, as in the statement of the lemma.  The other cases are handled similarly.
\end{proof}

\p{Closed formulas} We also have a closed formula for an arbitrary $\psi_{ij}(T_c^2)$.  The formula has three parts, depending on how many of $\{i,j\}$ lie in the interior of $c$.  Since we will not require these general formulas in the sequel, we do not supply the proofs.

Let $A$ be the set of labels of the marked points lying in the interior of $c$.  In the case where $A \cap \{i,j\} = \{i\}$ the formula is:
\[
\psi_{ij}(T_c^2) = \left(A \cup \left(A-\{i\}\right)'\right).
\]

We now suppose that $A \cap \{i,j\} = \emptyset$.  In this case there is a natural partition of $A$ into two subsets: two elements of $A$ are in the same subset if an arc that lies in the interior of $c$ and connects the corresponding marked points intersects $\alpha_{ij}$ in an even number of points.  If we denote the two subsets of $A$ by $A_1$ and $A_2$, the formula is
\[
\psi_{ij}(T_c^2) = 2(A_1 \cup A_2') + 2(A_1' \cup A_2).
\]

In the third and final case, where $\{i,j\} \subseteq A$, we need to define two subsets $B$ and $C$ of $[n] \setminus A$.  An element of $[n] \setminus A$ lies in $B$ if and only if an arc that lies in the exterior of $c$ and connects that marked point to $\partial \D_n$ crosses $\alpha_{ij}$ in an even number of points.  Then $C$ is the complement of $B$ in $[n]\setminus A$.  We have in this case
\[
\psi_{ij}(T_c^2) = 2(A \cup C \cup B' \cup (A \setminus \{i,j\})') + 2(C \cup B').
\]
It is straightforward to check that in the case where $T_c = T_{k\ell}$ our formulas here agree with Lemma~\ref{lem:psi Artin}.  For instance, in the third case we have $A = \{k,\ell\}$, $B = [n] \setminus \{k,\ell\}$, and $C = \emptyset$.  Thus our formula gives that $\psi_{k\ell}(T_{k\ell}^2)$ is
\[
2(A \cup C \cup B' \cup (A \setminus \{i,j\})') + 2(C \cup B') = 2(\{k,\ell\} \cup [n \setminus k,\ell]') + 2([n \setminus k,\ell]'),
\]
as per Lemma~\ref{lem:psi Artin}.

\subsection{Naturality.}\label{sec:nat} Our final task in this section to give a formula for the image under $\psi_{ij}$ of
\[
T_{k\ell} \cdot f = T_{k\ell} f T_{k\ell}^{-1}
\]
in terms of the image of $f$, where $f \in \B_n[4]$ and $T_{k\ell}$ is an Artin generator for $\PB_n$.  Since we already have a formula for each $\psi_{ij}(T_{k\ell}^2)$ (Lemma~\ref{lem:psi Artin}), this will give a formula for the image under $\psi_{ij}$ of an arbitrary $T_{k\ell} \cdot T_{pq}^2$.

Let $\{i,j\} \in [n]_\infty^{\underline 2}$ with $i < j$ and $k,\ell \in [n]$ with $k < \ell$.  We define a permutation $\iota_{k\ell}^{ij}$ of the set $[n] \cup [n]'$
as follows:
\[
\iota_{k\ell}^{ij} = 
\begin{cases}
(\ell\ \ell') &  \{i,j\} \cap \{k,\ell\} = \{k\} \\
(k\ k') &  \{i,j\} \cap \{k,\ell\} = \{\ell\} \\
id & \text{otherwise}.
\end{cases}
\]
In this formula we are using cycle notation for the symmetric group on $[n] \cup [n]'$, so $(k\ k')$ and $(\ell\ \ell')$ are  transpositions.

Let $\{i,j\} \in [n]_\infty^{\underline 2}$, let $\tilde \D_n$ be the corresponding branched cover of $\D_n$, and let $L_{ij}$ be the set of labels of the marked points of $\tilde \D_n$, as in Section~\ref{sec:psi def}.  We may regard $\iota_{k\ell}^{ij}$ as a permutation of $L_{ij}$, and as such it acts on the abelianization of $\PMod(\tilde \D_n)$.  We abuse notation and write the corresponding automorphism of the abelianization as $\iota_{k\ell}^{ij}$.  

\begin{lemma}
\label{lem:psi act}
Let $\{i,j\} \in [n]_\infty^{\underline 2}$ with $i < j$ and $\{k,\ell\} \in [n]^{\underline 2}$ with $k < \ell$.  Let $f$ be an element of $\B_n[4]$.  Then
\[
\psi_{ij}(T_{k\ell} \cdot f) = \iota_{k\ell}^{ij} \left(\psi_{ij}(f)\right).
\]
\end{lemma}  

\begin{proof}

For concreteness, we suppose that $j \neq \infty$; the case where $j =\infty$ is essentially the same.  As in Section~\ref{sec:psi def}, the set $L_{ij}$ is $[n] \cup [n]' \setminus \{i',j'\}$.  Let $\psi$ denote the abelianization of $\PMod(\tilde \D_n) \cong \PB_{2n-2}$.  We have that $\Capmap \circ \Lift(T_{k\ell})$ is an element of $\Mod(\tilde \D_n) \cong \B_{2n-2}$.  As such it acts on $\PMod(\tilde \D_n)$ by conjugation.  As in Section~\ref{sec:psi def} we denote the induced action on the abelianization by $\Capmap \circ \Lift(T_{k\ell})_*$.  We have:
\begin{align*}
\psi_{ij}(T_{k\ell} \cdot f) & = \psi(\Capmap \circ \Lift(T_{k\ell} \cdot f)) \\
& = \psi(\Capmap \circ \Lift(T_{k\ell}) \cdot \Capmap \circ \Lift(f)) \\
& = \Capmap \circ \Lift(T_{k\ell})_* \psi(\Capmap \circ \Lift(f)) \\
& = \Capmap \circ \Lift(T_{k\ell})_* \psi_{ij}(f).
\end{align*}
It remains to check that $\Capmap \circ \Lift(T_{k\ell})_*$ is equal to $\iota_{k\ell}^{ij}$.

If $\{i,j\} \cap \{k,\ell\}$ is not a singleton, then the simple closed curve $c$ in $\D_n$ corresponding to $\{k,\ell\}$ has an even number of intersections with the arc $\alpha_{ij}$.  Thus the preimage of $c$ in $\tilde \D_n$ is a pair of curves.  As in the proof of Lemma~\ref{lem:lift}, it follows that $\Capmap \circ \Lift(T_{k\ell})$ lies in $\PB_{2n-2}$.  This agrees with the fact that $\iota_{k\ell}^{ij}$ is trivial in this case.

Suppose that on the other hand $\{i,j\} \cap \{k,\ell\} = \{k\}$.  In this case the simple closed curve $c$ in $\D_n$ corresponding to $\{k,\ell\}$ intersects the arc $\alpha_{ij}$ in a single point.  So the preimage of $c$ in $\tilde \D_n$  is a single curve $\tilde c$ surrounding the points labeled $k$, $\ell$, and $k'$.  And $\Capmap \circ \Lift(T_{k\ell})$ interchanges the points labeled $k$ and $k'$ (the square of $\Capmap \circ \Lift(T_{k\ell})$ is the Dehn twist about $\tilde c$).  This again agrees with the definition of $\iota{k\ell}^{ij}$.  The case $\{i,j\} \cap \{k,\ell\} = \{\ell\}$ is exactly the same, with the roles of $k$ and $\ell$ interchanged.
\end{proof}

We now give three sample computations with Lemmas~\ref{lem:psi Artin} and~\ref{lem:psi act}.  First, for any $n \geq 3$ we have:
\begin{align*}
\psi_{1\infty}(T_{12} \cdot T_{23}^2) &= \iota_{12}^{1\infty}\left(\psi_{1\infty}(T_{23}^2)\right) \\
& = \iota_{12}^{1\infty} \left( 2(23) + 2(2'3') \right) \\
&= 2(23') + 2(2'3).
\end{align*}
Second, for any $n \geq 3$ we have:
\begin{align*}
\psi_{12}(T_{23}\cdot T_{12}^2) &= \iota_{23}^{12}\left(\psi_{12}(T_{12}^2)\right) \\
&= \iota_{23}^{12} \left( 2(123'4' \cdots n') + 2(3'4' \cdots n')  \right) \\
&= 2(1234' \cdots n') + 2(34' \cdots n').
\end{align*}
Third, for any $n \geq 3$ we have:
\begin{align*}
\psi_{1\infty}([T_{23}^2,T_{12}]) &= \psi_{1\infty}(T_{23}^2(T_{12} \cdot T_{23}^{-2})) \\
 &= \psi_{1\infty}(T_{23}^2) \iota_{12}^{1\infty}\left(\psi_{1\infty}(T_{23}^{-2})\right) \\
&= 2(23)+2(2'3') -   \iota_{12}^{1\infty} (2(23)+2(2'3')) \\
&= 2(23) + 2(2'3')-2(2'3)-2(23').
\end{align*}

As in the introduction, we refer to the image of $T_{k\ell}^2$ in $H_1(\B_n[4];\Q)$ by $\tau_{k\ell}$.  Regarding the $\psi_{ij}$ as maps defined on $H_1(\B_n[4];\Q)$, we can reinterpret the above calculations as:
\begin{align*}
\psi_{1\infty}(T_{12}\tau_{23}) = 2(23') &+ 2(2'3),   \quad \psi_{12}(T_{23}\tau_{12}) = 2(1234' \cdots n') + 2(34' \cdots n'), \quad \text{and} \\
&\psi_{1\infty}((1-T_{12})\tau_{23}) = 2(23) + 2(2'3')-2(2'3)-2(23').
\end{align*}


\section{Proof of Theorem~\ref{thm:main} for three strands}\label{3strandsection}

In this section we prove Theorem~\ref{thm:main} for the case of $B_3[4]$.  This result is stated as Corollary~\ref{cor:3b cor} below.  We begin by showing that $\dim H_1(B_3[4];\Q) = 6$ (Proposition~\ref{prop:3a}).  Then we show that a set $\S'$ closely related to the set $\S$ from Theorem~\ref{thm:main} forms a basis (Proposition~\ref{prop:3b}).  We then use this to prove Corollary~\ref{cor:3b cor}.  

Below, we denote by $\PMod_{0,m}$ the pure mapping class group of a sphere with $m$ marked points. 

\begin{proposition}
\label{prop:3a}
The dimension of $H_1(B_3[4];\Q)$ is equal to 6.
\end{proposition}

\begin{proof}

We may glue a disk with one marked point to the boundary of $\D_3$ in order to obtain a sphere with four marked points.  There is a resulting capping homomorphism
\[
\PB_3 \to \PMod_{0,4}.
\]
This map is defined analogously to the capping homomorphism in Section~\ref{sec:psi def}.  The kernel is the infinite cyclic group generated by $T_\partial$, the Dehn twist about the boundary of $\D_3$ \cite[Proposition 3.19]{farbmargalit}.

Since $\B_3[4] = \PB_3^2$ and since $\langle T_\partial \rangle \cap \B_3[4] = \langle T_{\partial}^2\rangle$ we may restrict the capping homomorphism to obtain a short exact sequence 
\[
1\rightarrow \langle T_{\partial}^2\rangle \rightarrow B_3[4]\rightarrow \PMod_{0,4}^2\rightarrow 1
\]
where $\PMod_{0,4}^2$ denotes the subgroup of $\PMod_{0,4}$ generated by all squares.  This extension gives rise to an exact sequence in homology
\[
\Q\langle T_{\partial}^2\rangle \rightarrow H_1(B_3[4];\Q)\rightarrow H_1(\text{PMod}^2_{0,4};\Q)\rightarrow 0
\]
We analyze the terms on the right and left in turn.  We claim that the term on the right is isomorphic to $\Q^5$.  By the Birman exact sequence \cite[Theorem 4.6]{farbmargalit} and the fact that $\PMod_{0,3}$ is trivial \cite[Proposition 2.3]{farbmargalit} we have that $\PMod_{0,4}$ is a free group of rank 2.  The index of $\PMod_{0,4}^2$ in $\PMod_{0,4}$ is 4 (it is the kernel of the mod 2 abelianization).  By the Nielsen--Schreier formula the former is a free group of rank 5 and the claim follows.

We next claim that the first map is injective.  The sequence of inclusions $\langle T_{\partial}^2\rangle \to \B_3[4] \to \PB_3$ induces maps on homology $\Q\langle T_{\partial}^2\rangle \rightarrow H_1(\B_3[4];\Q)\rightarrow H_1(\PB_3;\Q)$.  It follows from the discussion in Section~\ref{sec:psi def} that the composition is nontrivial, and the claim follows.

It follows from the two claims that $\dim H_1(B_3[4];\Q) = 1 + 5 = 6$. 
\end{proof}

Consider the set $\S' = \S_1 \cup \S_2'$ where
\[ \S_1 = \{ \tau_{12},\ \tau_{13},\ \tau_{23}\} \quad \text{and} \quad 
\S_2' = \{(1-T_{13})\tau_{12},\ (1-T_{23})\tau_{13},\ (1-T_{12})\tau_{23}\}
\]

\begin{proposition}
\label{prop:3b}
The set $\S'$ is a basis for $H_1(B_3[4];\Q)$.
\end{proposition}

\begin{proof}

The first step is to compute the values of $\psi$ on the elements of $\S'$. We find that
\begin{align*}
\psi(\tau_{12}) &= 2(12) & \psi((1-T_{13})\tau_{12} ) & =0\\
\psi(\tau_{13}) &= 2(13) & \psi((1-T_{23})\tau_{13}) & = 0 \\
\psi(\tau_{23}) &= 2(23) & \psi((1-T_{12})\tau_{23} ) &= 0. 
\end{align*}
Thus $\{\tau_{12},\tau_{13},\tau_{23}\}$ is linearly independent and the span of $\S_2'$ is contained in the kernel of $\psi$.  Thus it suffices to show that $\S_2'$ is linearly independent.  Let $\delta_{ij}$ denote $(ij) + (i'j') - (ij') - (i'j)$.  The images of the elements of $\S_2'$ under the $\psi_{i\infty}$ are as shown in the following table (the top-right entry was computed as an example in Section~\ref{sec:nat}):

\setlength{\tabcolsep}{15pt}
\renewcommand{\arraystretch}{1.25}
\begin{center}
\begin{tabular}{|l|c|c|c|}
\hline
& $(1-T_{13})\tau_{12}$ & $(1-T_{23})\tau_{13}$ & $(1-T_{12})\tau_{23}$ \\ \hline
$\psi_{1\infty}$ & 0 & 0 & $2\delta_{23}$ \\ \hline
$\psi_{2\infty}$ & 0 & $-2\delta_{13}$ & 0 \\ \hline
$\psi_{3\infty}$ & $2\delta_{12}$ & 0 & 0 \\ \hline
\end{tabular}
\end{center}

\bigskip

Since each element of $\S_2'$ has nontrivial image under exactly one $\psi_{i\infty}$ it follows that $\S_2'$ is linearly independent, and the proposition follows.
\end{proof}

As in the introduction, let $\S = \S_1 \cup \S_2$, where
\[
\S_2 = \{ T_{13}\tau_{12}, \ T_{23}\tau_{13}, \ T_{12}\tau_{23}\}.
\]
The set $\S$ lies in the span of $\S'$ and vice versa.  Since the two sets have the same cardinality, we have the following corollary of Proposition~\ref{prop:3b}, which is the $n=3$ case of Theorem~\ref{thm:main}.

\begin{corollary}
\label{cor:3b cor} 
The set $\S$ is a basis for $H_1(B_3[4];\Q)$.
\end{corollary}


\section{A linearly independent set}\label{4punctureslowerbound}

Let $\S$ be the subset of  $H_1(\B_n[4]; \Q)$ from Theorem~\ref{thm:main}.  The goal of this section is to prove the  ``lower bound'' for Theorem~\ref{thm:main}, namely, that $\S$ is linearly independent (Proposition~\ref{prop:low} below).  The proof will make use of the homomorphisms introduced in Section~\ref{doublecoverssection}. 

As in Section~\ref{3strandsection} we will prove that $\S$ is linearly independent by showing that a slightly different set $\S'$ is linearly independent. Specifically, let $\S' = \S_1 \cup \S_2' \cup \S_3'$, where
\begin{align*}
\S_1 &= \{\tau_{ij} \mid 1\leq i < j \leq n \},\\
\S_2' &= \{(1-T_{jk})\tau_{ij}, (1-T_{jk})\tau_{ik}, (1-T_{ij})\tau_{jk} \mid 1\leq i < j < k \leq n \}, \text{ and}\\
\S_3' &= \{ (1-T_{i\ell})(1-T_{jk})\tau_{ij}, (1-T_{ij})(1-T_{k\ell})\tau_{ik}, (1-T_{ik})(1-T_{j\ell})\tau_{i\ell} \mid 1\leq i < j < k < \ell \leq n \}.
\end{align*}
For a group $G$ and $g,h \in G$ we denote by $[g,h]$ the commutator $ghg^{-1}h^{-1}$.  The elements $(1-T_{k\ell})\tau_{ij}$ are the images of the commutators $[T_{k\ell},T_{ij}^2]$ in $H_1(\B_n[4];\Q)$.  Similarly, the $(1-T_{pq})(1-T_{k\ell})\tau_{ij}$ are the images of $[T_{pq},[T_{k\ell},T_{ij}^2]]$.

\begin{proposition}\label{prop:low}
For all $n\geq 4$ the sets $\S$ and $\S'$ are linearly independent. 
\end{proposition}

In the proof we use a homomorphism $F_n$ defined as follows.   Since $\B_n[4]$ is the subgroup of $\PB_n$ generated by all squares, there are well-defined maps $\B_n[4] \to \B_{n-k}[4]$ for all $0 \leq k < n$ obtained by forgetting $k$ of the marked points in $\D_n$.  There are thus induced maps $H_1(\B_n[4];\Q) \to H_1(\B_{n-k}[4];\Q)$.  By introducing formal variables $\varepsilon_{ijk\ell}$, $\varepsilon_{ijk}$ and $\varepsilon_{ij}$ we can combine the maps with $2 \leq n-k \leq 4$ into a single homomorphism $F_n$.  When $n=4$ it will also be convenient to use a function $\bar F_4$ that is defined in the same way as $F_4$ except without the terms corresponding to quadruples:
\begin{equation*}\label{psieq}
\bar F_4: H_1(B_n[4];\Q)\rightarrow \left(\bigoplus_{i<j<k}H_1(B_3[4];\Q)\otimes \varepsilon_{ijk}\right)\bigoplus\left(\bigoplus_{i<j}H_1(B_2[4];\Q)\otimes \varepsilon_{ij}\right),
\end{equation*}
where $\varepsilon_{ijk}$ corresponds to the map $\B_n[4]\rightarrow \B_3[4]$ obtained by forgetting the marked point not labeled by $i$, $j$, or $k$ and $\varepsilon_{ijk}$ corresponds to the map $\B_n[4]\rightarrow \B_2[4]$ obtained by forgetting the marked point not labeled by $i$ or $j$.

\begin{proof}[Proof of Proposition~\ref{prop:low}]

We proceed in two steps, first dealing with the case $n=4$ and then the general case.  In both cases it suffices to show that $\S'$ is independent, since (as in the proof of Corollary~\ref{cor:3b cor}) each element of $\S$ lies in the span of $\S'$ and vice versa.  

We now proceed with the proof for the case $n=4$.  We first claim that $\S_1$ is linearly independent.  For each $i < j$ we have
\[
\bar F_4(\tau_{ij}) = \tau_{ij}\otimes \varepsilon_{ij} + \sum \tau_{ij}\otimes \varepsilon_{abc},
\]
where the sum is over all triples $(a,b,c)$ with $a<b<c$ such that $i,j \in \{a,b,c\}$.  Since each $\varepsilon_{ij}$ only appears in the image of $\tau_{ij}$, and since $H_1(\B_2[4];\Q)$ is non-zero, the claim follows.

We next claim that $\S_1 \cup \S_2'$ is linearly independent.  We begin by computing the images of the elements of $\S_2'$ under $\bar F_4$.  For fixed $i<j<k$, we have
\begin{align*}
\bar F_4((1-T_{jk})\tau_{ij}) &= (1-T_{jk})\tau_{ij}\otimes \varepsilon_{ijk},  \\
\bar F_4((1-T_{jk})\tau_{ik}) &=(1-T_{jk})\tau_{ik} \otimes \varepsilon_{ijk}, \text{ and}\\
\bar F_4((1-T_{ij})\tau_{jk}) & = (1-T_{ij})\tau_{jk} \otimes \varepsilon_{ijk}. 
\end{align*}
First of all, since there are no $\varepsilon_{ij}$ terms here, it is enough to check that $\S_2'$ is independent.  Second, since exactly three of the twelve elements of $\S_2'$ have an $\varepsilon_{ijk}$ term in their images, it is enough to check that these three elements are linearly independent.  By replacing $i$, $j$, and $k$ with 1, 2, and 3, we see that this is equivalent to the statement that the set $\S_2'$ from the proof of Proposition~\ref{prop:3b} is independent.  Applying that proposition, the claim is proven.

The set $\S_3'$ lies in the kernel of $\bar F_4$, and so to prove the $n=4$ case it suffices to show that $\S_3'$ is linearly independent.  To do this we compute the images under $\psi_{12}$, $\psi_{13}$, and $\psi_{14}$:

\setlength{\tabcolsep}{7.5pt}
\renewcommand{\arraystretch}{1.25}
\begin{center}
\begin{tabular}{|l|c|c|c|}
\hline
& $(1-T_{14})(1-T_{23})\tau_{12}$ & $(1-T_{12})(1-T_{34})\tau_{13}$ & $(1-T_{13})(1-T_{24})\tau_{14}$ \\ \hline
$\psi_{12}$ & $4\delta_{34}$ & 0 & 0 \\ \hline
$\psi_{13}$ & 0 & $4\delta_{24}$ & 0 \\ \hline
$\psi_{14}$ & 0 & 0 & $4\delta_{23}$ \\ \hline
\end{tabular}
\end{center}
As in the proof of Proposition~\ref{prop:3b} the symbol $\delta_{ij}$ denotes the element $(ij) + (i'j') - (ij') - (i'j)$.

It is clear from the table that the image of $\S_3'$ is independent, and so it remains to verify the entries of the table.  The calculations for the three rows are similar.  We show details for only the first one.  For this we have
\[
(1-T_{14})(1-T_{23})\tau_{12} = \tau_{12} -T_{14} \tau_{12} - T_{23} \tau_{12}+ T_{14}T_{23}\tau_{12}
\]
By Lemma~\ref{lem:psi Artin}, the image of $\tau_{12}$ under $\psi_{12}$ is $2(123'4') + 2(3'4')$.  By Lemma~\ref{lem:psi act}, the images of $T_{14} \tau_{12}$, $T_{23} \tau_{12}$, and $T_{14}T_{23}\tau_{12}$ are $2(123'4') + 2(3'4')$, $2(123'4) + 2(3'4)$, and $2(1234) + 2(34)$.  The calculation in the table follows.  This completes the proof of the $n=4$ case. 

We now proceed to the general case.  The only elements of $\S'$ detected by the components of $F_n$ corresponding to pairs of marked points are the elements of $\S_1$.  The set $\S_1$ is linearly independent for the same reason it is in the proof of the $n=4$ case, namely, the fact that $H_1(\B_2[4];\Q)$ is nontrivial.  It then remains to show that $\S_2' \cup \S_3'$ is linearly independent.  The only elements of the latter detected by the components of $F_n$ corresponding to triples of marked points are the elements of $\S_2'$.  By applying Proposition~\ref{prop:3b} to each choice of triple, we conclude that $\S_2'$ is linearly independent.  It then remains to check that $\S_3'$ is linearly independent.  This follows by evaluating $F_n$ on $\S_3'$ and applying the $n=4$ case of the proposition.  This completes the proof.
\end{proof}


\section{A spanning set in terms of Artin generators}\label{spanning}

The main goal of this section is to prove the following proposition.  In the statement $\{T_{k\ell}\}$ is the set of Artin generators for $\PB_n$.  

\begin{proposition}\label{prop:span}
For all $n\geq 3$, the $\Q[\overline{PB}_n]$-module $H_1(B_n[4];\Q)$ is generated by $\{ \tau_{ij} \}$.
Equivalently, the vector space $H_1(B_n[4];\Q)$ is spanned by
\[
\T = \{(1-T_1)\cdots (1-T_m)\tau_{ij} \mid m \geq 0, \ T_1,\dots,T_m \in \{T_{k\ell}\} \}.
\]
\end{proposition}

We now explain why the two statements of the proposition are equivalent.  Clearly any element of $\T$ lies in the module spanned by the $\tau_{ij}$.  For the other direction, suppose we have an element of $H_1(B_n[4];\Q)$ of the form $T\tau_{ij}$ with $T \in \PB_n$.  We first write this as $T_1 \cdots T_m \tau_{ij}$ with each $T_i \in \{T_{k\ell}\}$ (note that no inverses are needed because the actions of $T_i$ and $T_i^{-1}$ are the same).  Then we may inductively apply the formula $T\tau _{ij}= -(1-T)\tau_{ij} + \tau_{ij}$ in order to express the original element as a linear combination of elements of $\T$.

\p{The boundary twist} Let $T_\partial$ denote the element of $\PB_n$ corresponding to the Dehn twist about $\partial \D_n$, and let $\tau_\partial$ denote the image of $T_\partial^2$ in $H_1(B_n[4];\Q)$.  We first express $\tau_\partial$ as a linear combination of elements $T\tau_{ij}$ as in the statement of Proposition~\ref{prop:span}.  Then we introduce the squared lantern relation and use it to prove the proposition.

\begin{lemma}\label{boundarylemma}
For all $n \geq 2$ we have 
\begin{equation*}
\tau_{\partial} = 2^{-{n\choose 2}}(1+T_{12})\cdots (1+T_{n-1,n})\sum_{i<j}\tau_{ij}.
\end{equation*}
In particular $\tau_\partial$ lies in the span of the $\tau_{ij}$.  
\end{lemma}

\begin{proof}

The steps of the proof are:
\begin{enumerate}
\item Every $\Z_n$-invariant element of $H_1(\B_n[4];\Q)$ is a multiple of $\tau_\partial$.
\item The following element of $H_1(\B_n[4];\Q)$ is $\Z_n$-invariant:
\[
x = (1+T_{12})(1+T_{13})\cdots (1+T_{n-1,n})\sum_{i<j}\tau_{ij}.
\]
\item $x = 2^{{n \choose 2}} \tau_\partial$.
\end{enumerate}
For the first step let $H_1(\B_n[4];\Q)^{\Z_n}$ and $H_1(\B_n[4];\Q)_{\Z_n}$ denote the spaces of $\Z_n$-invariants and $\Z_n$-coinvariants of $H_1(\B_n[4];\Q)$, respectively. Since $\Z_n$ is finite there are isomorphisms \[
H_1(B_n[4];\Q)^{\Z_n} \cong H_1(\B_n[4];\Q)_{\Z_n} \cong H_1(\B_n;\Q) \cong \Q,
\]
where the last isomorphism is induced by the signed word length homomorphism $\B_n \to \ZZ$.  
We conclude that, up to scale, there is a unique $\Z_n$-invariant element of $H_1(B_n[4];\Q)$. 
The image of $T_\partial^2$ under the signed word length homomorphism is $4{n \choose 2}$.  This is because $T_\partial$ can be written as a product of $T_{ij}$ where each $T_{ij}$ appears exactly once and because the signed word length of each $T_{ij}$ is 2.  Thus the image of $\tau_\partial$ of in $H_1(\B_n;\Q)\cong \Q$ is equal to $4{n \choose 2} \neq 0$, so any $\Z_n$-invariant element of $H_1(B_n[4];\Q)$ is a multiple of $\tau_\partial$.  

We now proceed to the second step.  Let $\sigma \in \B_n$ and let $\sigma_*$ be the induced permutation of $[n]^{\underline 2}$.  We will use two facts.  The first fact is that the action of $\sigma$ on the image of $\{T_{ij}\}$ in $\PZ_n$ is a permutation.  The second fact is that $\sigma \tau_{ij}$ is equal to $T\tau_{\sigma_*(ij)}$ for some $T \in \PZ_n$.  Both statements hold because $\sigma \cdot T_{ij}^k$ is conjugate in $\PB_n$ to $T_{\sigma_*(ij)}^k$.  

By the previous paragraph, the action of $\sigma$ on the $ij$-term of $x$ is given by
\[
\sigma \cdot (1+T_{12})\cdots (1+T_{n-1,n})\tau_{ij} = (1+T_{12})\cdots (1+T_{n-1,n}) T \tau_{\sigma_*(ij)}.
\]
Here we have used the fact that $\Q[\PZ_n]$ is commutative and so the permutation of $\{T_{ij}\}$ induced by $\sigma$ is irrelevant.

To complete the proof of the second step, it is then enough to show for $T \in \PZ_n$ that
\[
(1+T_{12})\cdots (1+T_{n-1,n}) T = (1+T_{12})\cdots (1+T_{n-1,n}) 
\]
in $\Q[\PZ_n]$.  If $T = T_{ij}$, then this follows from the equality
\[
(1+T_{ij})T_{ij} = T_{ij} + T_{ij}^2 = T_{ij} + 1 = 1 + T_{ij}
\]
in $\Q[\PZ_n]$ and the commutativity of the latter.  If $T$ is a product of more than one $T_{ij}$ then we apply this equality inductively.  This completes the proof of the second step.  

We now proceed to the third step.  As above, the image of $\tau_\partial$ in $H_1(\B_n;\Q)\cong \Q$ is equal to $4{n \choose 2}$.   We similarly compute the image of $x$ to be $4{n \choose 2}2^{{n \choose 2}}$.  Thus $x = 2^{{n \choose 2}} \tau_\partial$, as desired.  The lemma follows.
\end{proof}

We are now ready to prove Proposition~\ref{prop:span}.

\begin{proof}[Proof of Proposition~\ref{prop:span}]

We will use the theorem of Brendle and the second author that $\B_n[4]$ is equal to the subgroup of $\PB_n$ generated by squares of Dehn twists.  Because of this theorem it is enough to show that for any curve $c$ the image of $T_c^2$ in $\B_n[4]$ lies in the $\Q[\PZ_n]$-submodule of $H_1(B_n[4];\Q)$ generated  by $\{ \tau_{ij} \}$.  We first prove this in the special case where $c$ is $c_k$, the round circle in $\D_n$ surrounding the first $k$ marked points.

Fix some $k \leq n$.  The standard inclusion $\B_k \to \B_n$ induces an inclusion $f : \B_k[4] \to \B_n[4]$.  The latter further induces a map $f_* : H_1(\B_k[4];\Q) \to H_1(\B_n[4];\Q)$.  The map $f_*$ sends each $\tau_{ij}$ in $H_1(\B_k[4];\Q)$ to $\tau_{ij}$ in $H_1(\B_n[4];\Q)$.

By Lemma~\ref{boundarylemma}, the element $\tau_\partial \in H_1(\B_k[4];\Q)$ lies in the submodule of  $H_1(\B_k[4];\Q)$ generated by the $\tau_{ij}$.  By the previous paragraph, it follows that $f_*(\tau_\partial) \in H_1(\B_n[4];\Q)$ lies in the submodule of  $H_1(\B_n[4];\Q)$ generated by the $\tau_{ij}$.  But since $f(T_\partial^2) = T_{c_k}^2$ we have that $f_*(\tau_\partial)$ is the class of $T_{c_k}^2$, and so this completes the proof of the special case.

Let $c$ be an arbitrary curve in $\D_n$.  Say that $c$ surrounds $k$ marked points of $\D_n$.  There is a braid $\sigma \in \B_n$ with $\sigma(c_k) = c$ (this is a special case of the change of coordinates principle \cite[Section 1.3]{farbmargalit}).  Thus the image of $T_c^2$ in $H_1(\B_n[4];\Q)$ is obtained by applying the action of $\sigma$ to the image of $T_{c_k}$.   But the $\Q[\PZ_n]$-submodule of $H_1(B_n[4];\Q)$ generated  by $\{ \tau_{ij} \}$ is invariant under the action of $\B_n$ (as in the proof of Lemma~\ref{boundarylemma}, each $\sigma \tau_{ij}$ is equal to $T\tau_{k\ell}$ for some $k$, $\ell$ and $T \in \PB_n$), so the proposition follows.
\end{proof}


\section{Basis and dimension}\label{upperboundsection}

The goal of this section is to complete the proof of Theorem~\ref{thm:main}, which states that the set $\S$ from the introduction is a basis for $H_1(\B_4[4];\Q)$.  

Proposition~\ref{prop:low} states that the subset $\S'$ of $H_1(\B_n[4];\Q)$ is linearly independent.  The cardinality of $\S'$ is 
\[
3{n \choose 4} + 3{n \choose 3} + {n \choose 2}.
\]
Proposition~\ref{prop:span} gives a spanning set $\T$ for $H_1(\B_n[4];\Q)$ that contains $\S'$.  So our task in this section is to show that the elements of $\T \setminus \S'$ lie in the span of $\S'$.  In fact, we will see that the elements of $\T \setminus \S'$ are all multiples of elements of $\S'$.

Before proving Theorem~\ref{thm:main} we state and prove Lemma~\ref{lem:lantern}, which will allow us to eliminate many of the elements from the spanning set $\T$ for $H_1(\B_n[4];\Q)$ given in Proposition~\ref{prop:span}.  Next we state and prove Lemma~\ref{lem:comm}, which gives an expression for the class of $[T_{ij},T_{jk}]$ in $H_1(\B_n[4]; \Q)$ in terms of the $\tau_{ij}$.  We then state and prove Lemma~\ref{lem:jacobi}, a pair of algebraic identities used to prove the subsequent Lemma~\ref{lemma:key}, which gives certain equalities between elements of $\T$.  We then finally proceed to the proof of the theorem.

\begin{lemma}\label{lem:lantern}
Let $n\geq 2$.   The following statements hold in $H_1(\B_n[4];\Q)$.
\begin{enumerate}[leftmargin=*,itemsep=2ex]
\item \label{lem:lanternpart1}
For $i<j<k$ we have  
\[
T_{ik} \tau_{ij} = T_{jk} \tau_{ij} \label{5.1}, \qquad T_{ij}\tau_{ik} = T_{jk}\tau_{ik}, \qquad \text{and} \qquad T_{ij}\tau_{jk} = T_{ik} \tau_{jk}. 
\]
\item \label{lem:lanternpart2}
For pairwise distinct $i,j,k,\ell$ we have
\[
T_{ij}\tau_{k\ell} = \tau_{k\ell}.
\]
\end{enumerate}
\end{lemma}

\begin{proof}

We begin with the first statement.  For $i<j<k$, we have the following standard Artin relations in $\PB_n$:
\[
[T_{ik}T_{jk},T_{ij}] = [T_{jk}T_{ij}, T_{ik}] = [T_{ij}T_{ik}, T_{jk}] = 1.
\]
This implies that for $i<j<k$ we have
\[
[T_{ik}T_{jk},T^2_{ij}] = [T_{jk}T_{ij}, T^2_{ik}] = [T_{ij}T_{ik}, T^2_{jk}] = 1.
\]
These relations are also expressible as 
\[
(T_{ik}T_{jk})\cdot T_{ij}^2 = T_{ij}^2 \qquad
(T_{ij}T_{jk})\cdot T_{ik}^2 = T_{ik}^2 \qquad
(T_{ij}T_{ik})\cdot T_{jk}^2 = T_{jk}^2.
\]
Since an element of $\PB_n$ and its inverse have the same action on $H_1(\B_n[4];\Q)$ the above relations take the following form in $H_1(\B_n[4];\Q)$:
\[
T_{ik}\tau_{ij} = T_{jk}\tau_{ij} \qquad
T_{ij}\tau_{ik} = T_{jk}\tau_{ik}\qquad
T_{ij}\tau_{jk} = T_{ik}\tau_{jk}.
\]
This completes the proof of the first statement.

We now proceed to the second statement.  Let $i$, $j$, $k$, and $\ell$ be distinct.  Let us also assume that $i < j$ and $k < \ell$.  There are six possible configurations for $\{i,j,k,\ell\}$, two of which are linked and four of which are unlinked (see Section~\ref{sec:psi comp} for the definition of linked). In the unlinked cases the result follows from the fact that Dehn twists about disjoint curves commute.  Thus it remains only to consider the linked cases $i < k < j < \ell$ and $k < i < \ell < j$.  The two cases are essentially the same, so we deal only with the first.  

If $i < k < j < \ell$ then in $\PB_n$ we have the standard Artin relation $(T_{j\ell}T_{ij}T_{j\ell}^{-1})\cdot T_{k\ell} = T_{k\ell}$ which (as above) gives the relation
\[
(T_{j\ell}T_{ij}T_{j\ell}^{-1})\cdot T_{k\ell}^2 = T_{k\ell}^2.
\]
In $H_1(\B_n[4];\Q)$ this takes the form $T_{ij}\tau_{k\ell} = \tau_{k\ell}$, as desired.  
\end{proof}

In the next lemma we use $\bar T$ to denote the image of $T \in \B_n[4]$ in $H_1(\B_n[4]; \Q)$.

\begin{lemma}\label{lem:comm}
For $n\geq 3$ and $i<j<k$, the following holds in $H_1(\B_n[4]; \Q)$:
\[
\overline{[T_{ij},T_{jk}]} = \frac{1}{2}\left((1-T_{ik})\tau_{ij} + (1-T_{ij})\tau_{ik} - (1-T_{ij})\tau_{jk} \right).
\]
\end{lemma}

\begin{proof}

We first prove the lemma for the $n=3$ case and then use this to obtain the general case.

Brendle and the second author proved a relation in $\B_3[4]$ called the squared lantern relation \cite[Proposition 4.2]{brendlemargalit}.  In terms of the Artin generators for $\PB_n$ this relation can be written as
\[
[T_{12}, T_{12}\cdot T_{13}] = T_{12}^2T_{23}^2T_{13}^2T_{\partial}^{-2}
\]
Conjugating both sides by $T_{12}^{-1}$ and using the fact that $T_\partial$ is central yields the relation
\[
[T_{12}, T_{13}] = T_{12}^2\left(T_{12}^{-1}\cdot T_{23}^2\right)\left(T_{12}^{-1}\cdot T_{13}^2\right)T_{\partial}^{-2}\]
(this commutator is conjugate to the one in the statement).  Thus the following identity holds in $H_1(\B_3[4]; \Q)$:
\[
\overline{[T_{12},T_{13}]} = \tau_{12}+T_{12}\tau_{13} + T_{12}\tau_{23} - \tau_{\partial}.
\]
Next we claim that
\[
\tau_{\partial} = \frac{1}{2}((1+T_{13})\tau_{12} + (1+T_{12})\tau_{23} + (1+T_{12})\tau_{13}).
\]
By Lemma \ref{boundarylemma} we have that
\[
\tau_{\partial} = \frac{1}{8}(1+T_{12})(1+T_{13})(1+T_{23})(\tau_{12} + \tau_{13} + \tau_{23}).
\]
Lemma \ref{lem:lantern} implies that $T_{12}\tau_{23}= T_{13}\tau_{23}$ and therefore that 
\[
(1+T_{12})(1+T_{13})(1+T_{23})\tau_{23} = 2(1+T_{12})^2\tau_{23} = 4(1+T_{12})\tau_{23}.
\] 
Similar calculations show that 
\[
(1+T_{12})(1+T_{13})(1+T_{23})\tau_{12} = 4(1+T_{13})\tau_{12}
\]
and
\[
(1+T_{12})(1+T_{13})(1+T_{23})\tau_{13} = 4(1+T_{12})\tau_{13}.
\]
We therefore have that 
\begin{align*}
\tau_{\partial} &=\frac{1}{8}\left(4(1+T_{13})\tau_{12}+ 4(1+T_{12})\tau_{13}+(1+T_{12})\tau_{23}\right)\\
&= \frac{1}{2}((1+T_{13})\tau_{12} + (1+T_{12})\tau_{23} + (1+T_{12})\tau_{13}),
\end{align*} 
whence the claim.

Combining the claim and the above expression for $\overline{[T_{12},T_{13}]}$ we obtain
\begin{align*}
\overline{[T_{12},T_{13}]} & =  \tau_{12}+T_{12}\tau_{13} + T_{12}\tau_{23}  - \frac{1}{2}\left((1+T_{13})\tau_{12} + (1+T_{12})\tau_{23} + (1+T_{12})\tau_{13}\right)\\
&= \frac{1}{2}((1-T_{13})\tau_{12}-(1-T_{12})\tau_{13} - (1-T_{12})\tau_{23}).
\end{align*}
If we apply the braid generator $\sigma_1$ to the first and last expressions above we obtain
\begin{align*}
\overline{[T_{12},T_{23}]} &= \frac{1}{2}((1-T_{23})\tau_{12}-(1-T_{12})\tau_{23} - (1-T_{12})T_{12}\tau_{13}) \\
&= \frac{1}{2}((1-T_{23})\tau_{12}-(1-T_{12})\tau_{23} + (1-T_{12})\tau_{13}) \\ 
&= \frac{1}{2}((1-T_{13})\tau_{12}-(1-T_{12})\tau_{23} + (1-T_{12})\tau_{13})
\end{align*}
The second equality is obtained by multiplying the terms $(1-T_{12})$ and $T_{12}$ and the third equality is obtained from the equality $(1-T_{23})\tau_{12} = (1-T_{13})\tau_{12}$ from Lemma~\ref{lem:lantern}.  

For $n\geq 3$ and $i < j < k$ there exists an embedding $f: \D_3\hookrightarrow \D_n$ such that the images of $T_{12},T_{13}$, and $T_{23}$ under the induced map $f_*: \PB_3\rightarrow \PB_n$ are $T_{ij}, T_{ik}$, and $T_{jk}$, respectively. Applying $f_*$ to the expression for $\overline{[T_{12},T_{23}]}$ in the $n=3$ case yields the lemma.
\end{proof}

For a group $G$, the commutator subgroup $[G,G]$ is a subgroup of $G^2$, the subgroup of $G$ generated by squares of elements of $G$.  For $g \in G^2$ we denote by $\bar g$ the image in $H_1(G^2;\Q)$.  The group $G$ acts $G^2$ by conjugation.  This induces an action of $G$ on $H_1(G^2;\Q)$, which descends to an action of $\bar G = G/G^2$: for $g \in G^2$ and $h \in G$ we have $h \bar g = \overline{hgh^{-1}}$.  

\begin{lemma}
\label{lem:jacobi}
Let $G$ be a group, and let $x,y,z \in G$.  We have the following identities in $H_1(G^2;\Q)$, thought of as a $\Q[\bar G]$-module.
\begin{align*}
\text{Witt--Hall: }\ \ &\overline{[x,yz]} = \overline{[x,y]} + y\overline{[x,z]}  \\
\text{Jacobi: }\ \ &(1-x)\overline{[y,z]}-(1-y)\overline{[x,z]} + (1-z)\overline{[x,y]} = 0 
\end{align*}
\end{lemma}

\begin{proof}

The Witt--Hall identity for groups is the equality
\[
[x,yz] = [x,y]\left(y[x,z]y^{-1}\right)
\]
in $G$, which can can checked by simply expanding both sides.  Since $[G,G] \leqslant G^2$ we obtain from this the Witt--Hall identity in the statement.

We now proceed to the Jacobi identity.  The strategy is to express $[x,[y,z]]$ in two ways and to set the resulting expressions equal to each other.  On one hand, since
\[
[x,[y,z]] = (x[y,z]x^{-1})[y,z]^{-1}
\]
we have
\[
\overline{[x,[y,z]]} = (x-1)\overline{[y,z]}.
\]
On the other hand, writing
\[
[x,[y,z]] = [x,(yz)(zy)^{-1}]
\]
we obtain
\begin{align*}
\overline{[x,[y,z]]} &= \overline{[x,yz]} + yz\overline{[x,(zy)^{-1}]} \\
&= \overline{[x,y]} + y\overline{[x,z]} + yz\overline{[x,(zy)^{-1}]}\\
& = \overline{[x,y]} + y\overline{[x,z]} + yz\left((zy)^{-1}\overline{[zy,x]}\right)\\
& = \overline{[x,y]} + y\overline{[x,z]}  - \overline{[x,zy]}\\
& = \overline{[x,y]} + y\overline{[x,z]} - \left(\overline{[x,z]} + z\overline{[x,y]}\right)  \\
& = -(1-y)\overline{[x,z]} +(1-z)\overline{[x,y]},
\end{align*}
where the first, second, and fifth equalities use the Witt--Hall identity, the third equality uses the relation $[a,b^{-1}] = b^{-1}[b,a]b$, and the fourth equality uses the fact that $G^2$, hence $[G,G]$ acts trivially on $H_1(G^2;\Q)$.  The Jacobi identity follows. 
\end{proof}

\begin{lemma}\label{lemma:key}
For all $p<q<r<s$ we have
\begin{align*}
(1-T_{ps})(1-T_{qr})\tau_{pq} &=\ \ \, (1-T_{ps})(1-T_{qr})\tau_{rs}\\
(1-T_{pq})(1-T_{rs})\tau_{pr} &= -(1-T_{pq})(1-T_{rs})\tau_{qs}\\
(1-T_{pr})(1-T_{qs})\tau_{ps} &=\ \ \, (1-T_{pr})(1-T_{qs})\tau_{qr}
\end{align*}
\end{lemma}

\begin{proof}

We first prove the lemma in the case $n=4$ and then use this to obtain the general case.  When $n$ is 4, we have that $p$, $q$, $r$, and $s$ are 1, 2, 3, and 4.   Thus the statement of the lemma reduces to the following three specific equalities:
\begin{align*}
(1-T_{14})(1-T_{23})\tau_{12} &=\ \ \,(1-T_{14})(1-T_{23})\tau_{34}\\
(1-T_{12})(1-T_{34})\tau_{13} &=-(1-T_{12})(1-T_{34})\tau_{24}\\
(1-T_{13})(1-T_{24})\tau_{14} &=\ \ \,(1-T_{13})(1-T_{24})\tau_{23}
\end{align*}

We will prove the second equality using the Jacobi identity (Lemma~\ref{lem:jacobi}) and then use the $\B_n$-action to derive the other two.  Specifically, by inserting $x = T_{12}$, $y=T_{23}$, $z=T_{34}$ into the Jacobi identity and using the fact that $[T_{12},T_{34}] = 1$ in $\PB_4$, we obtain the following equality in $H_1(\B_n[4];\Q)$:
\begin{equation*}\label{keyrelationv2}
(1-T_{12})\overline{[T_{23},T_{34}]} = - (1-T_{34})\overline{[T_{12},T_{23}]}.
\end{equation*}
Applying Lemma~\ref{lem:comm} twice with $(i,j,k)$ equal to $(1,2,3)$ and $(2,3,4)$, and inserting the results into both sides of the above equation and simplifying, we obtain the second equality.

Acting on both sides of the second equality with $\sigma_2$ and $\sigma_3$, respectively, yields the first and third equalities.  This completes the proof of the lemma in the case $n=4$.  

We now address the general case.  Let $f: \{1,2,3,4\} \to \{p,q,r,s\}$ be the unique increasing map.  
There is an embedding $\D_4 \to \D_n$ so that the induced homomorphism $\B_4 \to \B_n$ maps $T_{ij}$ to $T_{f(i)f(j)}$.  It follows that the induced homomorphism on homology maps $\tau_{ij}$ to $\tau_{f(i)f(j)}$.  The images of the equalities from the $n=4$ case are the desired equalities.
\end{proof}

We are now ready to prove Theorem~\ref{thm:main}. 

\begin{proof}[Proof of Theorem~\ref{thm:main}]

By Proposition~\ref{prop:low} the set $\S$ is linearly independent.  To prove the theorem, we must show that $\S$ spans $H_1(\B_n[4]; \Q)$.  Since $\S$ has cardinality $3{n \choose 4} + 3{n \choose 3} + {n \choose 2}$, it suffices to show that there is some spanning set of this size.

The second statement of Proposition~\ref{prop:span} states that $H_1(\B_n[4]; \Q)$ is spanned as $\Q$-vector space by the set
\[
\T = \{(1-T_1)\cdots (1-T_m)\tau_{ij} \mid m \geq 0, T_1,\dots,T_m \in \{T_{k\ell}\} \}.
\] 
Let $\T_0 = \T$.  The goal is to successively eliminate elements of $\T_0$ until we obtain a set $\T_3$ with exactly $3{n \choose 4} + 3{n \choose 3} + {n \choose 2}$ elements.  The elements of $\T_3$ will in fact all be scalar multiples of the elements of the set $\S'$ from Proposition~\ref{prop:low}.

We first claim that $H_1(\B_n[4]; \Q)$ is spanned by the subset $\T_1$ of $\T_0$ containing all elements of the form
\[
\prod_{k \neq i,j} (1-T_{ik})^{\varepsilon_k} \, \tau_{ij}
\]
where each $\varepsilon_k$ lies in $\{0,1\}$.  Each element of $\T_1$ is an element of $\T_0$ where the corresponding $m$ is at most $n-2$ (but not all such elements of $\T_0$ lie in $\T_1$).  

To prove the claim we consider an element
\[
x = (1-T_1)\cdots (1-T_m)\tau_{ij}
\]
of $\T_0$ and consider a single term $(1-T_p)$ of the product.  Suppose that $T_p$ is the Artin generator $T_{k\ell}$.  If $x$ is non-zero then the product $(1-T_p)\tau_{ij}$ must be non-zero.  By the second statement of Lemma~\ref{lem:lantern} the intersection of $\{k,\ell\}$ with $\{i,j\}$ must contain exactly one element.  Using the first statement of Lemma \ref{lem:lantern} we may assume without loss of generality that the intersection is $\{i\}$.  This leaves exactly $n-2$ possibilities for $\{k,\ell\}$, namely, the sets $\{i,\ell\}$ with $\ell \neq i,j$.  We also have that $(1-T_p)^2 = 2(1-T_p)$ in $\Q[\PZ_n]$, and so may further assume that each term in the product appears at most once.  The claim now follows.

We next claim that $H_1(\B_n[4]; \Q)$ is spanned by the subset $\T_2$ of $\T_1$ consisting of elements where there are only two or fewer factors of the form $(1-T_{ik})$.  In other words, $\T_2$ consists of elements of the following form
\[
\tau_{ij}, \qquad (1-T_{ik_1})\tau_{ij}, \qquad (1-T_{ik_2})(1-T_{ik_1})\tau_{ij}
\]
where each $k_1$ and $k_2$ lies outside $\{i,j\}$ and in each product $k_1 \neq k_2$.  To prove the claim, it suffices to show that an element of $\T_1$ of the form
\[
(1-T_{ik_3})(1-T_{ik_2})(1-T_{ik_1})\tau_{ij}
\]
lies in the span of $\T_1$, where $k_3$ does not lie in $\{k_1,k_2\}$.  Since the latter is equal to
\[
(1-T_{ik_2})(1-T_{ik_1})\tau_{ij} -T_{ik_3}(1-T_{ik_2})(1-T_{ik_1})\tau_{ij}
\]
and the first of these terms already lies in $\T_2$ it suffices to show that the second term
\[
T_{ik_3}(1-T_{ik_2})(1-T_{ik_1})\tau_{ij}
\]
lies in the span of $\T_2$.  The basic strategy is to use Lemmas~\ref{lem:lantern} and~\ref{lemma:key} to convert the latter into an element of the form
\[
\pm T_{ik_3}(1-T_{\star\star})(1-T_{\star\star})\tau_{k_1k_2}.
\]
By Lemma~\ref{lem:lantern} and the fact that $i$, $k_1$, $k_2$, and $k_3$ are all distinct we have that $T_{ik_3}\tau_{k_1k_2} = \tau_{k_1k_2}$.  If both of the $T_{\star\star}$ terms are of the form $T_{k_1\star}$ then the given element lies in $\T_2$ (up to sign).  If either of the $T_{\star\star}$ terms are of the form $T_{k_2\star}$, then we may apply Lemma~\ref{lem:lantern} to replace it with $T_{k_1\star}$, leading to the previous case.    If either $T_{\star\star}$ term is not of the form $T_{k_1\star}$ or $T_{k_2\star}$, then the corresponding product $(1-T_{\star\star})\tau_{k_1k_2}$ equals 0.  In all cases, the given element lies in the span of $\T_2$.  

In order to convert $T_{ik_3}(1-T_{ik_2})(1-T_{ik_1})\tau_{ij}$ into the desired form, we proceed in two steps.  The first step is to replace either $T_{ik_1}$ or $T_{ik_2}$ with a different Artin generator, so that the result is one of the six types of elements listed in the statement of Lemma~\ref{lemma:key}.  Here is how we do this.  The disjoint sets $\{i,j\}$ and $\{k_1,k_2\}$ are either linked or unlinked.  If they are unlinked then we can replace $T_{ik_1}$ with $T_{jk_1}$ by Lemma~\ref{lem:lantern}.  If they are linked then we can replace $T_{ik_2}$ with $T_{jk_2}$ by the same lemma.

Since we have converted the given element $T_{ik_3}(1-T_{ik_2})(1-T_{ik_1})\tau_{ij}$ into one of the six forms in the statement of Lemma~\ref{lemma:key}, we can apply the corresponding equality from Lemma~\ref{lemma:key} and we obtain an element of the desired form.  The claim is now proved.

Our final claim is that $H_1(\B_n[4]; \Q)$ is spanned by the subset $\T_3$ of $\T_2$ consisting of all of the $\tau_{ij}$, all of the terms of the elements of the form $(1-T_{ik_1})\tau_{ij}$, and among the elements of the form $(1-T_{ik_2})(1-T_{ik_1})\tau_{ij}$, only those that satisfy
\[
i = \min\{i,j,k_1,k_2\}.
\]
This claim follows from Lemma~\ref{lemma:key}.  Indeed, of the six types of elements in the statement of that lemma, there are three types that do not satisfy the condition $i = \min\{i,j,k_1,k_2\}$, and in each case the element on the other side of the equality does satisfy the condition.

To complete the proof, it remains to check that the cardinality of $\T_3$ is $3{n \choose 4} + 3{n \choose 3} + {n \choose 2}$.  The number of $\tau_{ij}$ with $i < j$ is $n \choose 2$, the number of $(1-T_{ik_1})\tau_{ij}$ with $i < j$  and $k \notin \{i,j\}$ is $3{n \choose 3}$, and the number of $(1-T_{ik_2})(1-T_{ik_1})\tau_{ij}$ with $i < j$, with $k_1 < k_2$ and with $k_1,k_2 \notin \{i,j\}$ is $3{n \choose 4}$.  Adding these three terms together gives the desired result.
\end{proof}


\section{Representation theory of $\Z_n$}
\label{sec:bnbar}

In this section we prove Theorem~\ref{thm:class}, which states that the $V_n(\rho,\lambda)$ are irreducible representations of $\Z_n$ and moreover that every irreducible representation of $\Z_n$ is isomorphic to exactly one $V_n(\rho,\lambda)$.

In Section~\ref{sec:proj} we define the map $\Z_n^I \to \Z_m^I\times S_{n-m}$ used in the definition of the $V_n(\rho,\lambda)$ and prove that it is surjective (Lemma~\ref{lem:product}).  Then in Section~\ref{sec:iso} we give a complete criterion for a representation of $\Z_n$ to be irreducible (Proposition~\ref{prop:irred}) and use this to show that the $V_n(\rho,\lambda)$ are irreducible.  Finally in Section~\ref{sec:class} we complete the proof of Theorem~\ref{thm:class}.

\subsection{Projection maps}\label{sec:proj} 

Our definition of the $V_n(\rho,\lambda)$ was predicated on the existence of a map $\Z_n^I \to \Z_m^I \times \Z_{n-m}$.  In this section we prove Lemma~\ref{lem:product}, which gives such a map.

Let $I$ be an element of $\mathbb{I}_m$ and let $n \geq m$.  By the definition of $\mathbb{I}_m$ the union of the elements of $I$ is $[m]$.  As in the introduction, we may regard $I$ as a subset of $[n]^{\underline 2}$.  There are forgetful maps $f_1 : \B_n^I \to \B_m^I$ and $f_2 : \B_n^I \to \B_{n-m}$ obtained by forgetting the last $n-m$ strands and the first $m$ strands, respectively.  Since the $f_i$ take squares of pure braids to squares of pure braids, and since $\B_n[4] = \PB_n^2$, there are induced maps
\[
F_1: \Z_n^I\rightarrow \Z_m^I \quad \text{and} \quad F_2: \Z_n^I\rightarrow \Z_{n-m}.
\]
Let $P$ the composition of $F_1 \times F_2$ with the natural surjection $\B_m^I \times \B_{n-m} \to \B_m^I \times S_{n-m}$. Let $K_{n,m}$ be the subgroup of $\PZ_n$ generated by the images of the $T_{ij}$ with $j > m$.

In the proof of the lemma, we will need the following isomorphism:
\[
\overline \PB_n \cong \bigoplus_{[n]^{\underline 2}} \ZZ/2.
\]
This isomorphism follows from the description of the abelianization of $\PB_n$ in Section~\ref{sec:psi def} and the fact that $\B_n[4]$ is the kernel of the mod 2 abelianization of $\PB_n$. 

\begin{lemma}
\label{lem:product}
The map
\[
P : \Z_n^I \to \Z_m^I\times S_{n-m}
\]
is surjective with kernel $K_{n,m}$. 
\end{lemma}

\begin{proof}

We first show that $P$ is surjective.  Let $(g,\sigma) \in \Z_m^I\times S_{n-m}$.  Let $\iota : \B_m^I \times \B_{n-m} \leqslant \B_n$ be the natural inclusion, induced by disjoint embeddings $\D_m \to \D_n$ and $\D_{n-m} \to \D_n$.  Let $\tilde \sigma$ be a lift of $\sigma$ to $\B_{n-m}$.  Then $P \circ \iota (g,\tilde \sigma) = (g,\sigma)$.  Thus $P$ is surjective.

It remains to determine the kernel of $P$.  First, we observe that $K_{n,m}$ is contained in the kernel.  Since the stated generating set for $K_{n,m}$ has ${n \choose 2} - {m \choose 2}$, elements, and since these elements are part of the standard basis for $\PZ_n$, it follows that $K_{n,m}$ has cardinality $2^{{n \choose 2} - {m \choose 2}}$.  Computing the cardinalities of $\Z_n$, $\Z_m^I$, and $\Z_{n-m}$, we see that the kernel of $P$ must have cardinality $2^{{n \choose 2} - {m \choose 2}}$. The result follows. 
\end{proof}

\subsection{$I$-Isotypic representations and a criterion for irreducibility}\label{sec:iso} 

We will give in this section a characterization of the irreducible representations of $\Z_n$, Proposition~\ref{prop:irred} below.  As a consequence, we deduce in Corollary~\ref{cor:irred} that the $V_n(\rho,\lambda)$ are irreducible.

Our characterization uses the notion of an $I$-isotypic representation, and so we begin with this idea.  It follows from the above description of $\overline \PB_n$ that 
\[
H^1(\overline \PB_n;\mu_2) \cong \prod_{[n]^{\underline 2}} \mu_2,
\]
and so elements of $H^1(\PZ_n;\mu_2)$ are labeled by subsets of $[n]^{\underline 2}$ (recall $\mu_2=\{\pm 1\}$).  
We may identify $H^1(\PZ_n;\mu_2)$ with $\Hom(\PZ_n, \mu_2)$, and we denote the homomorphism corresponding to $I \subseteq [n]^{\underline 2}$ by $\rho_I$.  We denote the corresponding 1-dimensional representation of $\PZ_n$ over $\C$ by $V_I$.

Let $\Gamma$ be a subgroup of $\Z_n$ that contains $\PZ_n$; for instance $\Gamma = \Z_n^I$ for some $I$. Let $V$ be a representation of $\Gamma$ over $\C$ and let $I \subseteq [n]^{\underline 2}$. We will say that a subspace $W$ of $V$ is $I$-isotypic if it is a $\PZ_n$-submodule of $V$ and there is a $\PZ_n$-module isomorphism $W \cong V_I^{\oplus m}$ for some $m\geq 1$.

\begin{lemma}\label{lem:symm}
Let $\Gamma$ be a subgroup of $\Z_n$ that contains $\PZ_n$, and let $V$ be a representation of $\Gamma$ over $\C$. If $W \subset V$ is $I$-isotypic, then for all $\sigma\in \Gamma$ we have that $\sigma W$ is $\sigma(I)$-isotypic. 
\end{lemma}

\begin{proof}

Let $v \in W$, let $\sigma \in \Gamma$, and let $T_{ij}$ denote the image of an Artin generator for $\PB_n$ in $\PZ_n$.  It suffices to show that $T_{ij}(\sigma v)$ is equal to $\rho_{\sigma(I)}(T_{ij})\left(\sigma v\right)$.
We indeed have:
\begin{align*}
T_{ij}&\left(\sigma v\right)  = \sigma\left(\sigma^{-1}T_{ij}\sigma\right)v =  \sigma T_{\sigma^{-1}\{i,j\}}v \\  &=\sigma  \rho_I(T_{\sigma^{-1}\{i,j\}})\left(v\right)
 =\rho_I(T_{\sigma^{-1}\{i,j\}})\left(\sigma  v\right)
 = \rho_{\sigma(I)}(T_{ij})\left(\sigma v\right),
\end{align*}
as desired.
\end{proof}

\begin{proposition}\label{prop:irred}
Let $W$ be a representation of $\Z_n$ over $\C$. Then $W$ is irreducible if and only if there exists an $I \subseteq [n]^{\underline 2}$ and an irreducible, $I$-isotypic $\Z_n^I$-submodule $W_I \subset W$ so that we have a $\Z_n$-module isomorphism
\[
W\cong \Ind_{\Z_n^I}^{\Z_n}W_I.
\]
\end{proposition}

\begin{proof}

First assume $W$ is irreducible. Let $ \Res ^{\bar  \B_n}_{\PZ_n}W = \bigoplus_{I} W_I$ be the decomposition into isotypic subspaces. By Lemma~\ref{lem:symm} we have that $gW_I$ is $g(I)$-isotypic for each $g\in \Z_n$. Thus $gW_I = W_{g(I)}$ and the $\Z_n$-action permutes the $W_I$. Since $W$ is irreducible, the induced action on the set of indices $I$ is transitive. Hence for any choice of $I$ there is an isomorphism of $\Z_n$-modules $W \cong  \Ind_{\Z_n^{I}}^{\Z_n}W_I$. Since $W$ is irreducible, $W_I$ is an irreducible $\Z_n^{I}$-module. 

For the other direction, assume that $W$ is a $\Z_n$-module of the form $\Ind_{\Z_n^I}^{\Z_n}W_I$ for some irreducible, $I$-isotypic $\Z_n^I$-module $W_I$.  Let $W'$ be the irreducible $\Z_n$-submodule of $W$ that contains $W_I$.  For any $g\in \Z_n$ we have $gW_I\subseteq W'$.   Since $gW_I$ is $g(I)$-isotypic (Lemma~\ref{lem:symm}), $W'$ contains the direct sum $\bigoplus_{g\in \Z_n/\Z_n^I}gW_I$.  This direct sum is isomorphic to the $\Z_n$-module $\Ind_{\Z_n^I}^{\Z_n}W_I$, which we assumed to be isomorphic to $W$.  Thus $W'=W$, as desired.
\end{proof}

\begin{corollary}
\label{cor:irred}
Each $\Z_n$-representation $V_n(\rho,\lambda)$ is irreducible.
\end{corollary}

\begin{proof}

Fix some $V_n(\rho,\lambda)$.  From the definition there is an $m \leq n$, a full subset $I$ of $[m]^{\underline 2}$, an irreducible $\Z_m^I$-representation $V_m(\rho)$, and an irreducible $S_{n-m}$-representation $V(\lambda)$ so that 
\[
V_n(\rho,\lambda) = \Ind_{\Z_n^I}^{\Z_n} \left(V_m(\rho) \boxtimes V_{n-m}(\lambda)\right).
\]
Since we are working over an algebraically closed field of characteristic 0, and since $V_m(\rho)$ and $V_{n-m}(\lambda)$ are irreducible, and since the action of $\Z_n^I$ on $V_m(\rho) \boxtimes V_{n-m}(\lambda)$ factors through the surjective map $P$ from Lemma~\ref{lem:product}, the tensor product $V_m(\rho) \boxtimes V_{n-m}(\lambda)$ is an irreducible $\Z_n^I$-representation.  Since $\rho$ is $I$-isotypic by assumption, and since the image of $\PZ_n$ under $P$ lies in $\Z_m^I \leqslant \Z_m^I \times S_{n-m}$ it follows that $V_m(\rho) \boxtimes V_{n-m}(\lambda)$ is $I$-isotypic.  The corollary is thus an immediate consequence of Proposition~\ref{prop:irred}.
\end{proof}

\subsection{Classification of representations} 
\label{sec:class}

We are almost ready to prove Theorem~\ref{thm:reps}, our classification of irreducible representations of $\Z_n$.  What remains is to distinguish between different representations of the form $V_n(\rho,\lambda)$.  The following technical lemma provides the required tools for this.

\begin{lemma}\label{lem:technical} Let $n\geq 2$ and let $I,J \subseteq [n]^{\underline 2}$.
\begin{enumerate}[leftmargin =*]

\item Let $U$ be an $I$-isotypic $\Z_n^I$ module.  There is a $J$-isotypic $\Z_n^J$-module $W$ with \newline $\Ind_{\Z_n^I}^{\Z_n} U \cong \Ind_{\Z_n^J}^{\Z_n} W$ if and only if $I$ and $J$ lie in the same $\Z_n$-orbit.

\item If $U,W$ are irreducible $I$-isotypic $\Z_n^I$-modules, $\Ind_{\Z_n^I}^{\Z_n} U \cong \Ind_{\Z_n^I}^{\Z_n} W$ if and only if $U\cong W$. 
\end{enumerate}
\end{lemma}

\begin{proof}

We begin with the first statement.  For the reverse implication, we first observe that if $I$ and $J$ lie in the same $\Z_n$-orbit, which is to say that they lie in the same $S_n$-orbit, then $\B_n^I$ and $\B_n^J$ are conjugate in $\B_n$.  It follows that $\Z_n^I$ and $\Z_n^J$ are conjugate in $\Z_n$.  The desired conclusion is then given by the first part of the first exercise in Section III.5 of Brown's book~\cite{brown}. 

For the forward implication, suppose that $\Ind_{\Z_n^I}^{\Z_n} U \cong \Ind_{\Z_n^J}^{\Z_n} W$.  Since $U$ and $W$ are $I$- and $J$-isotypic, respectively, it follows that $\Res^{\Z_n}_{\PZ_n} \Ind_{\Z_n^I}^{\Z_n} U$ and $\Res^{\Z_n}_{\PZ_n} \Ind_{\Z_n^J}^{\Z_n} W$ are each direct sums of copies of representations of the form $V_{g(I)}$ and $V_{g(J)}$ for $g\in \Z_n$ (possibly with multiplicities), respectively.  We conclude that there is a $g$ so that $V_{g(I)}$ is isomorphic to $V_J$ as $\PZ_n$-modules.  But this implies that $g(I)=J$, as desired.

We proceed to the second statement.  The reverse implication is trivial.  For the forward implication suppose that $U$ and $W$ are non-isomorphic $I$-isotypic irreducible $\Z_n^I$-modules.  It suffices to prove that 
\[
\Hom_{\Z_n}\left(\Ind_{\Z_n^I}^{\Z_n}U ,\Ind_{\Z_n^I}^{\Z_n}W\right) = 0.
\]

By Frobenius reciprocity we have an isomorphism
\[
\Hom_{\Z_n}\left(\Ind_{\Z_n^I}^{\Z_n}U ,\Ind_{\Z_n^I}^{\Z_n}W\right) \cong \Hom_{\Z_n^I}\left(U, \Res^{\Z_n}_{\Z_n^I}\Ind_{\Z_n^I}^{\Z_n}W\right). 
\]
Let $E\subset \Z_n$ be a set of representatives for the set of double cosets $\Z_n^I \backslash \Z_n/ \Z_n^I$. Using the formula $g\Z_n^Ig^{-1} = \Z_n^{g(I)}$ for $g\in \Z_n$ we have an isomorphism of $\Z_n^I$-modules 
\[
\Res^{\Z_n}_{\Z_n^I}\Ind_{\Z_n^I}^{\Z_n}W = \bigoplus_{g\in E}\Ind_{\Z_n^I \cap \Z_n^{g(I)}}^{\Z_n^I} \Res_{\Z_n^I \cap \Z_n^{g(I)}}^{\Z_n^{g(I)}}gW
\]
(see \cite[p.69 Proposition 5.6(b)]{brown}), and so
\[
\Hom_{\Z_n^I}\left(U, \Res^{\bar  \B_n}_{\Z_n^I}\Ind_{\Z_n^I}^{\bar  \B_n}W\right) \cong \bigoplus_{g\in E}\Hom_{\Z_n^I }\left(U, \Ind_{\Z_n^I \cap\bar  \B_n^{g(I)}}^{\Z_n^I} \Res_{\Z_n^I \cap\bar  \B_n^{g(I)}}^{\Z_n^{g(I)}}gW\right).
\]
Applying Frobenius reciprocity once more, we see that the right-hand side is isomorphic to 
\[
\bigoplus_{g\in E}\Hom_{\Z_n^I\cap \Z_n^{g(I)}}\left(\Res_{\Z_n^I \cap \Z_n^{g(I)}}^{\Z_n^I}U, \Res_{\Z_n^I \cap \Z_n^{g(I)}}^{\Z_n^{g(I)}}gW\right).
\]
Since any $\Z_n^I \cap \Z_n^{g(I)}$-module map between $U$ and $gW$ restricts to a $\PZ_n$-module map, the fact that $U$ is $I$-isotypic and $gW$ is $g(I)$-isotypic implies that there can be no non-trivial $\Z_n^I \cap \Z_n^{g(I)}$-module maps between them unless $g(I) = I$. Since $g$ ranges over a set of representatives for the set of double cosets $\Z_n^I \backslash \Z_n/ \Z_n^I$, the only $g$ for which this condition is satisfied is $g = id$. Thus the only nontrivial summand in the above direct sum is $\Hom_{\Z_n^I}\left(U,W\right)$, which vanishes by Schur's lemma, because $U$ and $W$ are non-isomorphic irreducible $\Z_n^I$-modules.
\end{proof}

We are finally ready to prove Theorem~\ref{thm:class}, which states that every $V_n(\rho,\lambda)$ is an irreducible $\Z_n$-representation and conversely that every irreducible $\Z_n$-representation is isomorphic to exactly one $V_n(\rho,\lambda)$.

\begin{proof}[Proof of Theorem~\ref{thm:class}]

Corollary~\ref{cor:irred} already gives that the $V_n(\rho,\lambda)$ are irreducible.  Let $V$ be an arbitrary irreducible $\Z_n$-representation.  We would like to show that $V$ is isomorphic to some $V_n(\rho,\lambda)$ as a $\Z_n$-module.  By Proposition~\ref{prop:irred}, there is an $I \subseteq [n]^{\underline 2}$ and an irreducible $I$-isotypic $\Z_n^I$-representation $W_I$ such that $V = \Ind_{\Z_n^I}^{\Z_n}W_I$.  By the first statement of Lemma~\ref{lem:technical} we may assume that the union of the elements of $I$ is $[m]$ for some $m$.  Let $K_{n,m}$ be the kernel of the map $P$, as in Lemma~\ref{lem:product}.  Since the generators for $K_{n,m}$ act trivially on $W_I$, the $\Z_n^I$-action on $W_I$ descends to an action of the quotient $\Z_m^I\times \bar S_{n-m}$.  Since $W_I$ is an irreducible representation of $\Z_n^I$, it is an irreducible representation of the quotient $\Z_m^I\times S_{n-m}$.  Since we are working over an algebraically closed field of characteristic 0 an irreducible representation of a direct product of groups decomposes as an external tensor product of irreducible representations of the two factors \cite{serre}.  In particular, there are irreducible representations $U_1$ and $U_2$ of $\B_m^I$ and $S_{n-m}$ such that $W_I\cong U_1\boxtimes U_2$ as $\Z_m^I\times S_{n-m}$-modules.  Since $W_I$ is $I$-isotypic and since $\PZ_n$ acts trivially on $U_2$ it follows that $U_1$ is $I$-isotypic.  

To complete the proof of the theorem, it remains to prove the uniqueness statement.  Suppose that $V_n(\rho,\lambda)$ and $V_n(\rho',\lambda')$ are isomorphic as $\Z_n$-modules.  By the second statement of Lemma~\ref{lem:technical}, the $\Z_n$-modules $V_m(\rho) \boxtimes V_{n-m}(\lambda)$ and $V_{m'}(\rho') \boxtimes V_{n-m'}(\lambda')$ from which $V_n(\rho,\lambda)$ and $V_n(\rho',\lambda')$ are induced must be isomorphic.  It follows from the first statement of Lemma~\ref{lem:technical} that $m =m'$.  Since the tensor products are isomorphic, it follows that the individual factors are as well (as we are working over $\C$).
\end{proof}

\subsection{A non-splitting} The following proposition ties up a loose end from the introduction.

\begin{proposition}\label{prop:nosplit}
The following extension is not split:
\[
1\rightarrow \PZ_n \rightarrow \Z_n\rightarrow S_n\rightarrow 1.
\]
\end{proposition}

\begin{proof}

The surjection $\Z_n\rightarrow S_n$ induces a surjection
\[
H_1(\bar\B_n; \ZZ)\rightarrow H_1(S_n; \ZZ) \cong \ZZ/2
\]
We claim that $H_1(\bar{\B}_n; \ZZ) \cong \ZZ/4$.  Since there is no split surjection $\ZZ/4 \to \ZZ/2$, the proposition follows from this.

Since $\Z_n$ is the quotient of $\B_n$ by $\B_n[4]$ there is an exact sequence
\[
H_1(\B_n[4]; \ZZ) \to H_1(\B_n;\ZZ) \to H_1(\Z_n;\ZZ) \to 0.
\]
We have $H_1(\B_n;\ZZ) \cong \ZZ$.  The image of $H_1(\PB_n;\ZZ)$ in $H_1(\B_n;\ZZ)$ is $2\ZZ$, since each Artin generator evaluates to 2 under the length homomorphism on $\B_n$. Since $\B_n[4]$ is $\PB_n^2$ the image of $H_1(\B_n[4]; \ZZ)$ in $H_1(\B_n;\ZZ)$ is $4\ZZ$.  The claim follows.
\end{proof}


\section{Representation stability}

In this section we prove Theorem~\ref{thm:reps}, which gives the decomposition of $H_1(\B_n[4]; \C)$ into irreducible $\Z_n$-representations, and also states that the $H_1(\B_n[4]; \C)$ satisfy uniform representation stability.

In Section~\ref{sec:homs}, we define the representations $V_3(\rho_3)$ and $V_4(\rho_4)$ of $\bar\B_3^{I_3}$ and $\bar\B_4^{I_4}$ that are used in the representations $V_n(\rho_3,0)$ and $V_n(\rho_4,0)$ from the statement of Theorem~\ref{thm:reps}.  Then we prove the isomorphisms from Theorem~\ref{thm:reps} in Section~\ref{sec:reppf} by exhibiting the given $V_n(\rho,\lambda)$ as $\Z_n$-submodules of $H_1(\B_n[4]; \C)$.  These submodules are the spans of the orbits of the elements 
\begin{align*}
x_3  = (1-T_{13})\prod_{4 \leq j \leq n}&(1+T_{1j})(1+T_{2j})\tau_{12}\hspace{.25in} x_4 = (1-T_{14})(1-T_{23})\tau_{12}.
\end{align*}
(note that $x_3$ is only defined for $n \geq 3$ and $x_4$ only for $n \geq 4$).  Finally, in Section~\ref{sec:urs} we complete the proof of Theorem~\ref{thm:reps} by showing that $H_1(\B_n[4]; \C)$ satisfies the definition of uniform representation stability.

In this section we denote the span of $x \in H_1(\B_n[4];\C)$ by $\langle x\rangle$.

\subsection{Representations of \boldmath$\Z_n$}\label{sec:homs} The representations $V_3(\rho_3)$ and $V_4(\rho_4)$ of $\Z_n^{I_3}$ and $\Z_n^{I_4}$ will both be 1-dimensional representations obtained from homomorphisms $\rho_k : \Z_k^{I_k} \to \mu_2$.  We first define maps $\omega_k : \B_n^{I_k} \to \ZZ$ (Lemma~\ref{lem:homs}) and then obtain the $\rho_k$ from the mod 2 reductions of the $\omega_k$.  

In order to define the $\omega_k$ we take a different point of view on braids, as follows.  Let $C_n(\R^2)$ be the space of configurations of $n$ distinct, indistinguishable points in $\R^2$.  Choose a base point for $C_n(\R^2)$ where the $n$ points lie on a horizontal line.  There is a natural isomorphism $\pi_1(C_n(\R^2)) \cong \B_n$.  We label the points in the base point of $C_n(\R^2)$ by $[n]$ from left to right.  A loop in $C_n(\R^2)$ induces a permutation of $[n]$, and this is the usual homomorphism $\B_n \to S_n$.  If we represent a braid by a loop in $C_n(\R^2)$, then the $i$th strand of this braid representative is the path traced out by the point labeled $i$ (the terminology is explained by considering a spacetime diagram of the loop).  Let
\[
\xi_{ij} : \B_n \to \tfrac{1}{2}\ZZ
\]
be the function that counts the total winding number of the $i$th strand with the $j$th strand.  This is well defined because of our choice of base point for $C_n(\R^2)$.

With this in hand, we define a function $\omega_3 : \B_n^{I_3} \to \ZZ$ 
by the formula
\[
\omega_3 =  \xi_{13} + \xi_{23}.
\]
We similarly we define $\omega_4 : \B_n^{I_4} \to \ZZ$
by
\[
\omega_4 = \xi_{13} + \xi_{14} + \xi_{23} + \xi_{24}.
\]
The subgroup $\B_n^{I_3}$ can alternatively be described as the subgroup of $\B_n$ preserving the subsets $\{1,2\}$ and $\{3\}$ of $[n]$.  Similarly, $\B_n^{I_4}$ can be described as the subgroup preserving the pair of sets $\{\{1,2\},\{3,4\}\}$.

A priori the functions $\omega_3$ and $\omega_4$ are not well defined, since the natural codomain is $\tfrac{1}{2}\ZZ$ in both cases.  

\begin{lemma}
\label{lem:homs}
For $k \in \{3,4\}$, the function $\omega_k$ is a well-defined homomorphism.
\end{lemma}

\begin{proof}

We begin by showing that $\omega_3$ and $\omega_4$ are well defined.    For any braid in $\B_n^{I_3}$, the 1st and 2nd strands both start and end to the left of the 3rd strand.  It follows that both $\xi_{13}$ and $\xi_{23}$ map $\B_n^{I_3}$ to $\ZZ$.  Thus, $\omega_3$ is a well-defined function to $\ZZ$.  

For any braid in $\B_n^{I_4}$ that preserves $\{1,2\}$ the numbers $\xi_{13}$, $\xi_{14}$, $\xi_{23}$, and $\xi_{24}$ are all integers, similar to the $\omega_3$ case.  Also, for any braid in $\B_n^{I_4}$ that interchanges $\{1,2\}$ and $\{3,4\}$, none of  $\xi_{13}$, $\xi_{14}$, $\xi_{23}$, and $\xi_{24}$ are integers, and so again $\omega_4$ is well defined.

To complete the proof it remains to show that $\omega_3$ and $\omega_4$ are homomorphisms.  We observe that $\B_n^{I_3}$ can alternatively be described as the subgroup of $\B_n$ preserving the subsets $\{1,2\}$ and $\{3\}$.  Similarly, $\B_n^{I_4}$ is the subgroup preserving the pair of sets $\{\{1,2\},\{3,4\}\}$.

We begin with $\omega_3$.  Let $g \in \B_n^{I_3}$.  We color the 1st and 2nd strands red and the 3rd strand blue.  Then $\omega_3(g)$ is the sum of the winding numbers of red strands with blue strands.  If $g$ and $h$ are two elements of $\B_n^{I_3}$ then the colorings of the strands in $\R^2$ for $g$ and $h$ agree with the coloring of $gh$ (defined in the same way as the one for $g$).  It follows that $\omega_3$ is a homomorphism.

The case of $\omega_4$ is similar.  In this case, given $g \in \B_n^{I_4}$ we color the 1st and 2nd strands red and we color the 3rd and 4th strands blue.  Then $\omega_4(g)$ again is the sum of the winding numbers of red strands with blue strands.  Suppose now that $h$ is another element of $\B_n^{I_4}$.  If $g$ preserves the set $\{1,2\}$ then we color the strands of $h$ in the same way that we colored the strands of $g$.  Otherwise, if $g$ interchanges $\{1,2\}$ and $\{3,4\}$ then we color $h$ in the opposite way: the 1st and 2nd strands are blue and the 3rd and 4th strands are red.  For either coloring of $h$, the number $\omega_4(h)$ counts the sum of the winding numbers of red strands with blue strands.  The chosen colorings of $g$ and $h$ agree with the coloring on $gh$.  It follows that $\omega_4$ is a homomorphism. 
\end{proof}

The homomorphisms $\omega_3$ and $\omega_4$ induce homomorphisms $\B_n^{I_3} \to \mu_2$ and $\B_n^{I_4} \to \mu_2$.  The pure braid group $\PB_n$, hence $\B_n[4]$, is contained in each $\B_n^{I_k}$.  Since $\B_n[4]$ is equal to $\PB_n^2$ the image of $\B_n[4]$ under each map is trivial.  It follows that $\omega_3$ and $\omega_4$ induce homomorphisms $\bar\B_n^{I_3} \to \mu_2$ and $\bar\B_n^{I_4} \to \mu_2$.  Further restricting to $n=3$ and $n=4$ gives the desired homomorphisms
\[
\rho_3 : \bar\B_3^{I_3} \to \mu_2 \quad \text{and} \quad \rho_4 : \bar\B_4^{I_4} \to \mu_2.
\]
These homomorphisms give rise to the representations $V_3(\rho_3)$ and $V_4(\rho_4)$ from the introduction.  

\begin{lemma}
\label{lem:iso}
Let $k \in \{3,4\}$.  The representation $V_k(\rho_k)$ is $I_k$-isotypic.
\end{lemma}

\begin{proof}

Since each $\rho_k$ defines a 1-dimensional representation, it is enough to check that the restriction of $\rho_k$ to $\PZ_k$ is equal to $\rho_{I_k}$.   For any $I$ the homomorphism $\rho_I$ can be written as
\[
\rho_I = \sum_{\{i,j\} \in I} \frac{1}{2} \xi_{ij} \mod 2.
\]
The lemma now follows from this and the expressions of the $\rho_{I_k}$ in terms of the $\xi_{ij}$.
\end{proof}


\bigskip
\bigskip
\bigskip

\subsection{The irreducible decomposition}\label{sec:reppf}

We are now in a position to prove the first part of Theorem~\ref{thm:reps}, which we state separately as Proposition~\ref{prop:rep1} below.    We require two lemmas.

In the statement of the first lemma,  $\sigma_{13}$ denotes the half-twist in $\B_n$ whose square is $T_{13}$.  In terms of the standard generators for $\B_n$, we can write $\sigma_{13}$ as $(\sigma_3 \sigma_2) \sigma_1 (\sigma_3 \sigma_2)^{-1}$.  Also, when an element is not defined we simply drop it from the proposed generating sets (for instance $T_{14}$ is not an element of $\B_3^{I_3}$).

\newpage

\begin{lemma}
\label{lem:bni gens}
For $n \geq 3$ the group $\B_n^{I_3}$ is generated by the set 
\[
\{ T_{13}, T_{14}, T_{34}, \sigma_1 \} \cup \{\sigma_i \mid i \geq 4\}. 
\]
For $n \geq 4$ the group $\B_n^{I_4}$ is generated by the set 
\[
\{ T_{23} ,T_{45}, \sigma_1, \sigma_2 \sigma_{13}^{-1}\} \cup \{\sigma_i \mid i \geq 5 \}.
\]
\end{lemma}

\begin{proof}

We begin with the case of $\B_n^{I_3}$.  To simplify the exposition we assume $n \geq 4$; the case $n=3$ is obtained by ignoring the elements $T_{14}$ and $T_{34}$. The stabilizer of $I_3$ in $S_n$ is the image of $S_2 \times S_1 \times S_{n-3}$ under the standard inclusion.  The $\sigma_i$ in the proposed generating set map to the standard generators for this subgroup.  Thus it suffices to check that every Artin generator $T_{ij} \in \PB_n$ lies in the group generated by the proposed generators.  Using the fact that $T_{13}$, $T_{14}$, and $T_{34}$ lie in the generating set and inductively applying the formulas $\sigma_j T_{ij} \sigma_j^{-1} = T_{i,j+1}$ and $\sigma_i^{-1} T_{ij} \sigma_i = T_{i-1,j}$ shows that each $T_{ij}$ is a product of the proposed generators, as desired.

We now treat  $\B_n^{I_4}$.  Again, to simplify the exposition we assume $n \geq 5$; the case $n=4$ is obtained by ignoring the element $T_{45}$. The stabilizer of $I_4$ in $S_n$ is isomorphic to $(S_2 \times S_2 \times S_{n-4}) \rtimes \ZZ/2$, where the $\ZZ/2$ factor is any element of order 2 that interchanges $\{1,2\}$ with $\{3,4\}$.  The element $\sigma_2\sigma_{13}^{-1}$ maps to $(1 4)(2 3)$, giving the $\ZZ/2$ factor.  The element $\sigma_1$ maps to the generator of the first $S_2$ factor.  Since the $\ZZ/2$ factor interchanges the  $S_2$ factors, the generator for the other $S_2$ factor also is in the image.   The $\sigma_i$ with $i \geq 5$ map to the standard generators for the $S_{n-4}$ factor.  So again the lemma reduces to the problem of exhibiting each $T_{ij}$ as a product of generators.  This is achieved in the same way as in the previous case.
\end{proof}

\begin{lemma}
\label{lem:x1x2}
Let $k \in \{3,4\}$ and let $n \geq k$.  The subspace $\langle x_k \rangle$ of $H_1(\B_n[4];\C)$ is a $\Z_n^{I_k}$-module isomorphic to $V_k(\rho_k) \boxtimes V_{n-k}(0)$.
\end{lemma}

\begin{proof}

We begin by observing that $x_3$ and $x_4$ are nonzero in $H_1(\B_n[4];\C)$.  The element $x_4$ is certainly nonzero, as it is one of the basis elements of $H_1(\B_n[4];\C)$ from Theorem~\ref{thm:main}.  To see that the element $x_3$ is nonzero, we apply the forgetful map $\PB_n \to \PB_3$ that forgets the last $n-3$ strands.  Via this map, $(1+T_{1j})$ and $(1+T_{2j})$ both map to 2 in $\Q[\PZ_3]$, and $(1-T_{13})\tau_{12}$ maps to $(1-T_{13})\tau_{12}$ in $H_1(\B_3[4];\C)$.  Since the latter is one of the basis elements for $H_1(\B_3[4];\C)$ from Theorem~\ref{thm:main} it follows that the image of $x_3$, hence $x_3$ itself, is nonzero.

As in the introduction, the action of $\Z_n^{I_k}$ on $V_k(\rho_k) \boxtimes V_{n-k}(0)$ factors through the surjection to $P : \Z_n^{I_k} \to \Z_k^{I_k} \times S_{n-k}$ from Lemma~\ref{lem:product}.  Let $P_1$ denote the composition of $P$ with projection to the first factor.

For each $k$, Lemma~\ref{lem:bni gens} gives a set of generators for $\B_n^{I_k}$.  We will show that the image $g$ of each generator in $\Z_n^{I_k}$ preserves $\langle x_k \rangle$ and moreover that $gx_k = \rho_k \circ P_1(g)x_k$.  Since the representation $V_k(\rho_k) \boxtimes V_{n-k}(0)$ is determined by $\rho_k \circ P_1$ the lemma follows from this.  In the argument we refer to an element of $\B_n^{I_k}$ and its image in $\Z_n^{I_k}$ by the same symbol.  

We begin with the case $k=3$.  Again, to simplify the exposition, we assume $n \geq 4$.  For $T_{13}$ we have $T_{13}(1-T_{13}) = -(1-T_{13})$ and so $T_{13}x_3 = -x_3 = \rho_3 \circ P_1(T_{13})x_3$, as desired.  For $T_{14}$ we have $T_{14}(1+T_{14}) = 1+T_{14}$ and so again $T_{14}x_3 = x_3 = \rho_3 \circ P_1(T_{14})x_3$.  Next we have $T_{34}\tau_{12}=\tau_{12}$ and so $T_{34}x_3 = x_3 = \rho_3 \circ P_1(T_{34})x_3$.

For $\sigma_1$ we use the following relations in $\B_n$: $\sigma_1T_{1j}\sigma_1^{-1} = T_{2j}$ for $j \geq 3$,  $\sigma_1T_{2j}\sigma_1^{-1} = T_{12}T_{1j}T_{12}^{-1}$ for $j \geq 3$,  and $\sigma_1$ commutes with $T_{12}^2$.  Since $\PZ_n$ is abelian we have $T_{12}T_{1j}T_{12}^{-1} = T_{1j}$ in $\Z_n^{I_3}$.  Using these facts and Lemma~\ref{lem:lantern} in turn we obtain that 
\begin{align*}
\sigma_1 \cdot (1-T_{13})\prod_{4 \leq j \leq n}(1+T_{1j})(1+T_{2j})\tau_{12} &= (1-T_{23})\prod_{4 \leq j \leq n}(1+T_{1j})(1+T_{2j})\tau_{12} \\ &= (1-T_{13})\prod_{4 \leq j \leq n}(1+T_{1j})(1+T_{2j})\tau_{12},
\end{align*}
which is to say that $\sigma_1 x_3 = x_3 = \rho_3 \circ P_1(\sigma_3)x_3$.  

Finally we must deal with the $\sigma_j$ with $j \geq 4$.  As elements of $\B_n$, each of these commutes with $T_{12}^2$ and $T_{13}$.  Similarly, for $i \in \{1,2\}$ and $k \notin \{1,2,j-1,j\}$ we have that $\sigma_j$ commutes with $T_{ij}$.  Also, for $i \in \{1,2\}$ we have that $\sigma_j T_{ij} \sigma_j^{-1} = T_{i,j+1}$ and $\sigma_j T_{i,j+1} \sigma_j^{-1} = T_{j,j+1}T_{ij}T_{j,j+1}^{-1}$ in $\B_n$.  As before we have the the latter is equal to $T_{ij}$ in $\Z_n^{I_3}$.  We deduce that $\sigma_j x_3 = x_3 = \rho_3\circ P_1(\sigma_j)x_3$ for $j \geq 4$.

We now treat the case $k=4$.  To simplify the exposition, we assume $n \geq 5$.  For $T_{45}$, we have that $T_{45}\tau_{12} = \tau_{12}$ and so $T_{45}x_4 = x_4 = \rho_4 \circ P_1(T_{45})x_4$.  For $T_{23}$ we have $T_{23}(1-T_{23}) = -(1-T_{23})$ and so $T_{23}x_4 = -x_4 = \rho_4 \circ P_1(T_{23})x_4$.  Since $\sigma_i$ commutes with $T_{14}$, $T_{23}$, and $T_{12}^2$ for $i \geq 5$ we also have $\sigma_i x_4 = x_4 = \rho_4 \circ P_1(\sigma_i)x_4$ for $i \geq 5$.  Using the relations $\sigma_1 T_{14} \sigma_1^{-1} = T_{24}$ and $\sigma_1 T_{23}^2 \sigma_1^{-1} = T_{12} T_{13}^2 T_{12}^{-1}$ in $\B_n$ we have
\[
\sigma_1 \cdot (1-T_{14})(1-T_{23})\tau_{12} = (1-T_{24})(1-T_{13})\tau_{12} = (1-T_{14})(1-T_{23})\tau_{12}
\]
and so $\sigma_1 x_4 = x_4 = \rho_4 \circ P_1(\sigma_1)x_4$.  Finally, we will show that $\sigma_2\sigma_{13}^{-1} x_4 = x_4 = \rho_4 \circ P_1(\sigma_2\sigma_{13}^{-1})x_4$.   The braid $\sigma_2\sigma_{13}^{-1}$ commutes with $T_{14}$ and $T_{23}$, and $(\sigma_2\sigma_{13}^{-1})T_{12}^2(\sigma_2\sigma_{13}^{-1})^{-1} = T_{34}^2$ in $\B_n$.  Applying these facts and the first equality of Lemma~\ref{lemma:key} in turn we have:
\[
\sigma_2\sigma_{13}^{-1} \cdot (1-T_{14})(1-T_{23})\tau_{12} = (1-T_{14})(1-T_{23})\tau_{34} = (1-T_{14})(1-T_{23})\tau_{12},
\]
as desired.
\end{proof}

\begin{proposition}
\label{prop:rep1}
There are $\Z_n$-equivariant isomorphisms
\[
H_1(\B_n[4]; \C)\cong 
\begin{cases}
V_2(1,(0)) & n =2 \\[1ex]
V_3(1,(0)) \oplus V_3(1,(1)) \oplus V_3(\rho_3,(0)) & n=3 \\[1ex]
V_n(1,(0)) \oplus V_n(1,(1)) \oplus V_n(1,(2)) \oplus V_n(\rho_3,(0)) \oplus V_n(\rho_4,(0))  & n\geq 4.\\
\end{cases}
\]
\end{proposition}

\begin{proof}[Proof of Proposition~\ref{prop:rep1}]

As usual, to simplify the exposition, we assume $n \geq 4$.  The other cases are obtained by ignoring the appropriate terms.

The first step is to show that we have the following isomorphisms of $\Z_n$-modules:
\[
H_1(\B_n[4]; \C) \cong 
\begin{cases}
\Ind_{\Z_2^{I_2}}^{\Z_2}\C & n=2 \\[2ex]
\Ind_{\Z_3^{I_2}}^{\Z_3}\C \oplus \Ind_{\Z_3^{I_3}}^{\Z_3}\langle x_3\rangle & n=3 \\[2ex]
\Ind_{\Z_n^{I_2}}^{\Z_n}\C \oplus \Ind_{\Z_n^{I_3}}^{\Z_n}\langle x_3\rangle \oplus \Ind_{\Z_n^{I_4}}^{\Z_n}\langle x_4\rangle & n \geq 4,
\end{cases}
\]
where $\C$ is the trivial $\Z_2^{I_2}$ module.  The second step is to identify the summands of this decomposition with the summands in the statement of the theorem.  

We begin with the first step.   The index of $\Z_n^{I_k}$ in $\B_n$ is $n \choose 2$, $3{n  \choose 3}$, and $3{n \choose 4}$ for $k$ equal to 2, 3, and 4, respectively.  Thus the dimension of the purported decomposition equals the dimension of $H_1(\B_n[4]; \C)$ as given in Theorem~\ref{thm:main}.  Since representations of finite groups are completely reducible in characteristic zero, it thus suffices to show that the representations
\[
\Ind_{\Z_n^{I_2}}^{\Z_n}\C, \quad 
\Ind_{\Z_n^{I_3}}^{\Z_n}\langle x_3 \rangle, \quad \text{and} \quad 
\Ind_{\Z_n^{I_4}}^{\Z_n}\langle x_4 \rangle
\]
appear as submodules of $H_1(\B_n[4]; \C)$ and that they pairwise intersect in the zero vector.  We deal with each summand in turn.  

We start with the first submodule $\Ind_{\Z_n^{I_2}}^{\Z_n} \C$. For $i < j$ we define 
\[
\alpha_{ij} = \prod_{r<s} (1+T_{rs}) \tau_{ij}.
\]
The $\alpha_{ij}$ are linearly independent since $\alpha_{ij}$ is detected exactly by the forgetful map $\B_n[4] \to \B_2[4]$ that forgets all but the $i$th and $j$th marked points in $\D_n$. Since $g \alpha_{ij} = \alpha_{g(i)g(j)}$ for all $g \in \Z_n$ it follows that $H_1(\B_n[4]; \C)$ contains the $\Z_n$-module
\[
\bigoplus_{[g]\in \Z_n/\Z_n^{I_2}} g\langle \alpha_{12}\rangle = \Ind_{\Z_n^{I_2}}^{\Z_n}\langle \alpha_{12}\rangle = \Ind_{\Z_n^{I_2}}^{\Z_n}\C,
\]
as desired.  

Next, we identify the submodules $\Ind_{\Z_n^{I_k}}^{\Z_n}\langle x_k\rangle$.  Fix $k \in  \{3,4\}$.  We consider the subspaces $g\langle x_k\rangle$ of $H_1(\B_n[4];\C)$ for $[g]\in \Z_n/\Z_n^{I_k}$.  Since $\langle x_k\rangle$ is a $\Z_n^{I_k}$-module (Lemma~\ref{lem:x1x2}), these subspaces do not depend on the choice of $g \in [g]$. It follows from the fact that $\langle x_k\rangle$ is $I_k$-isotypic (Lemma~\ref{lem:iso}) and Lemma~\ref{lem:symm} that the $g\langle x_k\rangle$ are mutually non-isomorphic, and so $H_1(\B_n[4];\C)$ contains the direct sum.  We may write this direct sum as 
\[
\bigoplus_{[g]\in \Z_n/\Z_n^{I_k}} g\langle x_k \rangle = \Ind_{\Z_n^{I_k}}^{\Z_n}\langle  x_k \rangle,
\]
which is the desired submodule.

Since $I_2$, $I_3$, and $I_4$ lie in different $\Z_n$-orbits, it follows from Lemma~\ref{lem:symm} that the three summands we have found have trivial intersection pairwise.  This completes the first step.  

We now proceed to the second step.   By Lemma~\ref{lem:x1x2}, the second and third summands from the first step agree with the summands $V_n(\rho_3,(0))$ and $V_n(\rho_4,(0))$ from the statement of the proposition.  It therefore remains to show that we have an isomorphism of $\Z_n$-modules 
\[
\Ind_{\Z_n^{I_2}}^{\Z_n}\C \cong V_n(1,(0)) \oplus V_n(1,(1)) \oplus V_n(1,(2)).
\]
Any representation of $\Z_n$ where $\PZ_n$ acts trivially can be naturally identified with a representation of $S_n$.  Applying this identification to the right-hand side of the above isomorphism yields the $S_n$-representation $V_n(0) \oplus V_n(1) \oplus V_n(2)$.  Recall that on the left-hand side $\C$ is the trivial $\Z_n^{I_2}$-representation and so again $\PZ_n$ acts trivially.  Applying the same identification to the left-hand side yields the $S_n$-representation $\Ind_{S_2 \times S_{n-2}}^{S_n}\C$.  It follows from the branching rule that the latter is isomorphic to $V_n(0) \oplus V_n(1) \oplus V_n(2)$, as desired.  
\end{proof}


\subsection{Uniform representation stability}
\label{sec:urs}

In this section we prove the second statement of Theorem~\ref{thm:reps}, namely, that the sequence $\{H_1(\B_n[4]; \Q)\}$ of $\Z_n$-modules is uniformly representation stable (Proposition~\ref{prop:urs} below).  

\begin{lemma}
\label{lem:inc}
The standard embedding $\B_n\to \B_{n+1}$ induces injective maps $\B_n[4] \to \B_{n+1}[4]$ and $\Z_n \rightarrow \Z_{n+1}$.
\end{lemma}

\begin{proof}

As mentioned in the introduction, Brendle and the second author proved that the group $\B_n[4]$ is equal to $\PB_n^2$.  It follows that the image of $\B_n[4]$ under $\B_n \to \B_{n+1}$ is contained in $\B_{n+1}[4]$.  Thus the standard embedding $\B_n\to \B_{n+1}$ induces a well-defined maps $\B_n[4] \to \B_{n+1}[4]$ and $\Z_n \rightarrow \Z_{n+1}$.  The first map is injective since it is the restriction of an injective map.  The injectivity of the second map is equivalent to the statement that the preimage of $\B_{n+1}[4]$ under $\B_n\to \B_{n+1}$ is contained in $\B_n[4]$.  This follows from the fact that $\PB_n^2$, hence $\B_n[4]$, is the kernel of the mod 2 abelianization of $\PB_n$ and the fact that the following square commutes, where the horizontal arrows are the mod 2 abelianizations, and the vertical maps are the inclusions:
\[
\xymatrix
   { 
   \PB_{n+1}  \ar[r] & (\ZZ/2)^{n+1 \choose 2}  & \\
   \PB_n  \ar[u] \ar[r] & (\ZZ/2)^{n \choose 2} \ar[u] & \\
   }
\]
This completes the proof.
\end{proof}

\begin{lemma}
\label{lem:bn4cons}
For each $k \geq 0$ and $n \geq 0$, the vector spaces $H_k(\B_n[4]; \Q)$ form a consistent sequence of $\Z_n$-representations with respect to the maps induced by the standard inclusions.
\end{lemma}

\begin{proof}

For each $g \in \B_n$ there is a commutative diagram of groups
\[
\xymatrix
   { 
   \B_n[4] \ar[d]_{g} \ar[r] & \B_{n+1}[4] \ar[d]^{g} & \\
   \B_n[4]  \ar[r] & \B_{n+1}[4] & \\
   }
\]
where $g$ acts by conjugation.  The lemma then follows by applying $H_k$ to all four groups, and using the fact that a group acts trivially on its homology groups.
\end{proof}

\begin{proposition}
\label{prop:urs}
The sequence $\{H_1(\B_n[4]; \Q)\}$ of $\Z_n$-modules is uniformly representation stable.
\end{proposition}

\begin{proof}

We check the three parts of the definition of uniform representation stability in turn.  The standard  inclusion map $\B_n[4] \to \B_{n+1}[4]$ from Lemma~\ref{lem:inc} is a right inverse to the surjective map $\B_{n+1}[4] \to \B_n[4]$ obtained by forgetting the last strand.  It follows that the induced maps $\varphi_n : H_1(\B_n[4];\C) \to H_1(\B_{n+1}[4];\C)$ are injective.  It follows from the first statement of Proposition~\ref{prop:span} and the fact that every $\tau_{ij}$ lies in the same $\Z_{n+1}$-orbit as $\tau_{12}$ that the $\Z_{n+1}$-span of $\varphi_n(H_1(\B_n[4];\C))$ is equal to $H_1(\B_{n+1}[4];\C)$.  Finally, the condition on the multiplicities of the irreducible components follows immediately from Proposition~\ref{prop:rep1}.
\end{proof}


\section{A non-generating set}\label{sec:nongen}

In this short section we use Lemma~\ref{boundarylemma} to prove Theorem~\ref{thm:2}, which states that if $\BI_n \leqslant G \leqslant \B_n[4]$ then $G$ is not generated by even powers of Dehn twists about curves surrounding two points.

\begin{proof}[Proof of Theorem~\ref{thm:2}]

Brendle and the second author proved that the standard forgetful maps $\B_n[4] \to \B_3[4]$ induce a surjection $G \to \B_3[4]$; see \cite[Corollary 4.4]{brendlemargalit}.  Under any such forgetful map, an even power of a Dehn twist about a curve surrounding two marked points either maps to the identity or to an even power of a Dehn twist about a curve surrounding two marked points.  Thus  it suffices to prove the result for the case $n=3$.

By Lemma~\ref{boundarylemma}, we have in $H_1(\B_3[4]; \Q)$ that 
\[
\tau_{\partial} = \frac{1}{2}\left(\tau_{12} + \tau_{13} + \tau_{23} + T_{13}\tau_{12} + T_{12}\tau_{13} + T_{12}\tau_{23}\right).
\]
On the other hand, if $T_{\partial}^2$ could be written as a product of even powers of Dehn twists about curves surrounding two points, there would exist integers $c_1,\ldots, c_6$ such that
\[
\tau_{\partial} = c_1\tau_{12} + c_2\tau_{13} + c_3\tau_{23} + c_4T_{13}\tau_{12} + c_5T_{12}\tau_{13} + c_6T_{12}\tau_{23}
\]
But this is impossible, since $\tau_{12}$, $\tau_{13}$, $\tau_{23}$, $T_{13}\tau_{12}$, $T_{12}\tau_{13}$, and $T_{12}\tau_{23}$ are exactly the elements of our basis $\S$ for $H_1(\B_3[4]; \Q)$ from Corollary~\ref{cor:3b cor}. 
\end{proof}


\section{Albanese cohomology}\label{sec:alb}

In this section, we will prove Theorems~\ref{thm:mainD}, and~\ref{thm:mainE}, which state that $H^*_{Alb}(\B_n[4]; \Q)$ is a proper subalgebra of $H^*(\B_n[4]; \Q)$ for all $n\geq 15$ and that $H^*_{Alb}(\SMod_g[4];\Q)$ is a proper subalgebra of $H^*(\SMod_g[4];\Q)$ for $g\geq 7$, respectively. We conclude the section with the proofs of Theorem~\ref{thm:smallbetti} and Proposition~\ref{prop:genus2prop}. The former gives the Betti numbers of $\B_3[4]$ and $\B_4[4]$, while the latter gives a new (large) lower bound on the top Betti number of $\Mod_2[4]$.

\subsection{Interpretations of the level 4 braid group \`a la Brendle--Margalit}\label{sec:ala} In this section, it will be advantageous to recast the group $\B_n[4]$ in two different ways.  Specifically, we will utilize the following two isomorphisms, which hold for $g \geq 1$:
\begin{align*}
\B_{2g+1}[4] &\cong \SMod_g[4] \times \ZZ \\
\B_{2g+1}[4] &\cong \PMod_{0,2g+2}^2 \times \ZZ.
\end{align*}
As in Section~\ref{3strandsection}, $\PMod_{0,n}$ denotes the pure mapping class group of a sphere with $n$ marked points and $\PMod_{0,n}^2$ is the subgroup generated by all squares.

Neither of the above isomorphisms are stated explicitly by Brendle and the second author.  However, both are easily obtained from their work, as we shall explain currently. 

We begin with the first isomorphism.   Brendle and the second author \cite[Theorem 4.2]{pointpushing} proved the analogous isomorphism $\BI_{2g+1}  \cong \SI_g \times \ZZ$ (their theorem actually refers to the hyperelliptic Torelli group $\SI_g^1$ of a surface with boundary instead of $\BI_{2g+1}$, but as explained in their introduction the groups $\SI_g^1$ and $\BI_{2g+1}$ are naturally isomorphic).  The proof of their isomorphism applies verbatim in our situation, except with the Torelli group replaced with the level 4 mapping class group.  

The second isomorphism follows from the theorem of Brendle and the second author that $\B_{2g+1}[4] \cong \PB_{2g+1}^2$ and the fact that $\PB_n$ splits as a direct product as $\PMod_{0,n+1} \times \ZZ$; see \cite[p. 252]{farbmargalit}.

We can also combine the above two isomorphisms in order to obtain the isomorphism
\[
\SMod_g[4] \cong \PMod_{0,2g+2}^2
\]
for $g \geq 1$.  Indeed, the group $\B_{2g+1}[4]$ has infinite cyclic center, and so the composition of the two isomorphisms above must identify the two given $\ZZ$-factors.

In this section we will use one other fact from the work of Brendle and the second author.  They observed \cite[Corollary 4.4]{brendlemargalit} that each of the forgetful maps $\PB_n \to \PB_m$ induces a surjective homomorphism
\[
\B_n[4] \to \B_m[4].
\]
(cf. the proof of Theorem~\ref{thm:2}).  This map is split.  For instance if the forgetful map $\PB_n \to \PB_m$ is the one obtained by forgetting the last $n-m$ marked points of $\D_n$ then the splitting is the restriction of the standard inclusion $\B_m \to \B_n$.  Both the surjectivity and the existence of the splitting follow directly from the isomorphism $\B_{2g+1}[4] \cong \PB_{2g+1}^2$.

\subsection{The proofs of Theorems~\ref{thm:mainD} and~\ref{thm:mainE}}

Our next goal is to prove Theorems~\ref{thm:mainD} and~\ref{thm:mainE}, which state that the Albanese cohomology algebras of $\B_n[4]$ and $\SMod_g[4]$ are proper subalgebras of $H^*(\B_n[4];\Q)$ and $H^*(\SMod_g[4]; \Q)$ for $n \geq 15$ and $g\geq 7$, respectively.  We give two lemmas that give the cohomological dimension and the Euler characteristic of $\SMod_g[4]$ before proceeding to the proofs of Theorem~\ref{thm:mainE} and~\ref{thm:mainD} (in that order).

For a group $G$ we denote by $\cd G$ its cohomological dimension.  

\begin{lemma}
\label{lem:cd} 
For $n \geq 3$ we have
\[
\cd \PMod_{0,n}^2 = n-3
\]
and for $g \geq 1$ we have
\[
\cd \SMod_g[4] = 2g-1.
\]
\end{lemma}

\begin{proof}

For $n \geq 3$ we have $\cd \PMod_{0,n} = n-3$.  The two statements now follow from the fact that $\PMod_{0,n}^2$ has finite index in $\PMod_{0,n}$ and the isomorphism $\SMod_g[4] \cong \PMod_{0,2g+2}^2$ from Section~\ref{sec:ala}, respectively.
\end{proof}

\begin{lemma}\label{lem:eulerlemma}
For $g\geq 1$ we have
\[
\chi(\SMod_g[4]) = -2^{{2g+1 \choose 2}-1}(2g-1)!
\]
\end{lemma}

\begin{proof}

As explained in Section~\ref{sec:ala}, the group $\SMod_g[4]$ is isomorphic to $\PMod_{0,2g+2}^2$.  We will compute the Euler characteristic of the latter.  

We claim that the index of $\PMod_{0,n}^2$ in $\PMod_{0,n}$ is $2^{{n-1 \choose 2} - 1}$.  Since for any group $G$ we have $G/G^2 \cong H_1(G;\ZZ/2)$, it follows that
\[
\PMod_{0,n}/\PMod_{0,n}^2 \cong H_1(\PMod_{0,n}; \ZZ/2) \cong (\ZZ/2)^{{n-1 \choose 2}-1}
\]
The last isomorphism follows from the splitting $\PB_{n-1} \cong \PMod_{0,n}\times \ZZ$ and the usual description of the abelianization of $\PB_{n-1}$.  The claim follows.

Harer and Zagier \cite[p. 476]{harerzagier} proved that
\[
\chi(\PMod_{0,n}) = (-1)^{n-3}(n-3)!
\]
For any group $G$ and a subgroup $G'$ of finite index we have $\chi(G') = [G:G']\chi(G')$.  The lemma follows by combining this fact with the claim.
\end{proof}

\begin{proof}[Proof of Theorem~\ref{thm:mainE}]  

By Lemma~\ref{lem:cd} we have $\cd \SMod_g[4] = 2g-1$. Therefore, in order to show that $H^*_{Alb}(\B_n[4];\Q)$ is a proper subalgebra of $H^*(\B_n[4];\Q)$ we must show that the image of the cup product map
\[
\Lambda^i H^1(\SMod_g[4]; \Q) \to H^i(\SMod_g[4]; \Q).
\]
fails to be surjective for some $2 \leq i \leq 2g-1$. 

Let $b_i$ denote the $i$th Betti number of $\SMod_g[4]$ and let $d_i$ denote the dimension of the image of the above cup product map.  Our basic strategy is to show that there is some $i$ between 2 and $2g-1$ with $b_i > d_i$.  To do this we will estimate the $d_i$ from above and the $b_i$ from below.

We first claim that
\[
d_i \leq {b_1 \choose 2g-1}
\]
for all $2 \leq i \leq 2g-1$.  As $g \geq 7$ it follows from Corollary~\ref{cor:mainC} that $2g-1 < b_1/2$, and so
\[
d_i \leq \dim \Lambda^i H^1(\SMod_g[4]; \Q) = {b_1 \choose i} \leq {b_1 \choose 2g-1}
\]
for $i \leq 2g-1$, as desired.

We next claim that
\[
b_{2k-1} > \frac{1}{g-1}\left(2^{{2g+1 \choose 2}-1}(2g-1)! - b_1 \right)
\]
for some $k$ with $2 \leq k \leq g$.  Since (as above) the cohomological dimension of $\SMod_g[4]$ is $2g-1$ we have the following immediate consequence of Lemma~\ref{lem:eulerlemma}:
\[
b_1+b_3+\cdots + b_{2g-1} > 2^{{2g+1 \choose 2}-1}(2g-1)!
\]
The claim follows.

Combining the two claims, it is now enough to show that 
\[
\frac{1}{g-1}\left(2^{{2g+1 \choose 2}-1}(2g-1)! - b_1 \right)
>
{b_1 \choose 2g-1}
\]
for $g \geq 7$.

For $g = 7$ we can verify the inequality numerically.  Direct computation shows that the right-hand side is on the order of $10^{38}$ and that the left-hand side is on the order of $10^{40}$.

We now treat the general case $g \geq 8$.   We will perform four strengthenings of the desired inequality in order to obtain an inequality that we can prove with basic calculus.  First, using the estimate ${n \choose k} \leq n^k/k!$ and the estimate $(2g-1)! > \left(\frac{2g-1}{e}\right)^{2g-1}\sqrt{2\pi(2g-1)}$ we obtain the stronger inequality
\[
\frac{1}{g-1}\left(2^{{2g+1 \choose 2}-1}\left(\frac{2g-1}{e}\right)^{2g-1}\sqrt{2\pi(2g-1)} - b_1 \right)
>
\frac{b_1^{2g-1}}{(2g-1)!}
\]
Next, adding $b_1/(g-1)$ to both sides and using the fact that $b_1/(g-1) < b_1 \leq {b_1 \choose 2g-1}$, that ${n \choose k} \leq n^k$, and that $b_1^{2g-1} < b_1^{2g}$, we obtain the even stronger inequality
\[
2^{{2g+1 \choose 2}-1}\left(\frac{2g-1}{e}\right)^{4g-2}2\pi(2g-1) 
>
2b_1^{2g}.
\]
It follows from Theorem~\ref{thm:main} and the estimate ${n \choose k} \leq n^k/k!$ that $b_1 < \frac{(2g+2)^4}{6}$.  Using this and dividing both sides of the last inequality by 2 we obtain the even stronger inequality
\[
2^{{2g+1 \choose 2}-1}\left(\frac{2g-1}{e}\right)^{4g-2}\pi(2g-1)
>
\left(\frac{(2g+2)^4}{6}\right)^{2g}.
\]
Since both sides of the last inequality are positive, we may take the logarithms of both sides in order to obtain the equivalent inequality
\begin{align*}
\left(2g^2+g-1\right)\ln 2+ &(4g-2)\left(\ln(2g-1) -1\right) +  \ln(2g-1) + \ln \pi \\
&> 8g\ln(2g+2)-2g\ln 6.
\end{align*}
Set 
\[
G(x) = \left(2x^2+x-1\right)\ln 2+(4x-2)\left(\ln(2x-1)-1\right)+\ln(2x-1) + \ln \pi
\] and 
\[
H(x) = 8x\ln(2x+2)-2x\ln 6.
\]  
The last inequality can be restated as $G(g) > H(g)$. By direct computation, the function $F(x) = G(x)-H(x)$ satisfies $F(8) > 0$ and $F'(8) > 0$. Furthermore, for $x\geq 8$ we have that 
\begin{align*}
F''(x) & =  4\ln 2 + \frac{8}{2x-1} - \frac{8}{x+1} - \frac{4}{(2x-1)^2} - \frac{8}{(x+1)^2}\\
&> 4\ln 2 + 0 - \frac{8}{9}-\frac{4}{225}-\frac{8}{81}\\
&>0
\end{align*} 
where we have used the fact that $x\geq 8$ in the first inequality. This implies that $F'(x)$ is increasing for $x \geq 8$, and therefore that $F(x)$ is increasing for all $x\geq 8$. The theorem follows.
\end{proof}

\begin{proof}[Proof of Theorem~\ref{thm:mainD}]  

We will now derive Theorem~\ref{thm:mainD} from Theorem~\ref{thm:mainE}. Because of the isomorphism $\B_{2g+1}[4]\cong \SMod_g[4]\times \ZZ$ there is a split surjective homomorphism $\B_{2g+1}[4]\rightarrow \SMod_g[4]$ induced by projection onto the first factor. Any section $\sigma: \SMod_g[4]\rightarrow B_{2g+1}[4]$ induces a surjection
\[
\sigma^*: H^*(B_{2g+1}[4];\Q)\to H^*(\SMod_g[4];\Q)
\]
Let $\sigma^*(1)$ denote the algebra homomorphism 
\[
\Lambda^*H^1(B_{2g+1}[4];\Q)\rightarrow \Lambda^*H^1(\SMod_g[4];\Q)
\]
induced by $\sigma^*$ in degree 1. This map is surjective. 

The relationships between $\sigma^*$, $\sigma^*(1)$, and the cup product are given by the following commutative diagram:
\[
\xymatrix{
 \Lambda^*H^1(B_{2g+1}[4];\Q)\ar[r]^{\ \sigma_g^*(1)\ }\ar[d]^{\smile} & \Lambda^*H^1(\SMod_g[4];\Q)\ar[d]^{\smile}\\
 H^*(B_{2g+1}[4];\Q)\ar[r]^{\sigma_g^*} & H^*(\SMod_g[4];\Q).
}
\]

We complete the proof of Theorem~\ref{thm:mainD} by first dealing with the case of $n$ odd, followed by the case of $n$ even. 

By Theorem~\ref{thm:mainE}, the rightmost cup product in the above diagram fails to be surjective for $g\geq 7$. This implies that for all $g\geq 7$ the cup product $\Lambda^*H^1(\B_{2g+1}[4];\Q)\rightarrow H^*(\B_{2g+1}[4];\Q)$ is not surjective. This proves Theorem~\ref{thm:mainD} for $n$ odd with $n\geq 15$. 

It remains to deal with the case of $n$ even. Let $\B_{2g+2}[4]\rightarrow \B_{2g+1}[4]$ be the map induced by forgetting the last marked point of $\D_n$ and let $s$ be any section, for instance the one induced by the standard inclusion $\B_{2g+1} \to \B_{2g+2}$. Replacing $\B_{2g+1}[4]$, $\SMod_g[4]$, and $\sigma$ in the diagram above with $\B_{2g+2}[4]$, $\B_{2g+1}[4]$, and $s$, respectively, and applying the odd $n$ case of Theorem~\ref{thm:mainD}, we obtain that for $g\geq 7$ the cup product $\Lambda^*H^1(\B_{2g+2}[4];\Q)\rightarrow H^*(\B_{2g+2}[4];\Q)$ is not surjective. This completes the proof. 
\end{proof}

\subsection{Higher Betti numbers}\label{sec:smallbetti}

In this section we will prove Theorem~\ref{thm:smallbetti}, which gives the Betti numbers of $\B_n[4]$ for $n=3,4$ and Proposition~\ref{prop:genus2prop}, which gives a lower bound for the top Betti number of $\Mod_2[4]$.   We begin with a lemma. 

\begin{lemma}
\label{lem:cdeuler}
For all $n\geq 1$ we have $\cd \B_n[4] = n-1$ and $\chi(\B_n[4]) = 0$.
\end{lemma}

\begin{proof}

For $n \geq 1$ we have $\cd \PB_n = n-1$; the lower bound comes from the existence of a free abelian subgroup of rank $n-1$ (generated by Dehn twists) and the upper bound comes from the decomposition of $\PB_n$ into an $(n-1)$-fold iterated semidirect product of free groups (via combing).  Since $\B_n[4]$ has finite index in $\PB_n$, the first statement follows.

It follows from Arnol'd's computation \cite[Corollary 2]{arnold} of the Poincar\'{e} polynomial of $\PB_n$ that $\chi(\PB_n) = 0$. Since $\B_n[4]$ has finite index in $\PB_n$, we obtain the second statement. 
\end{proof}

\begin{proof}[Proof of Theorem~\ref{thm:smallbetti}]

In the proof we will denote the $i$th Betti number of a group $G$ by $b_i(G)$ and we will abbreviate $b_i(\B_n[4])$ by $b_i$.  

We begin with the case of $n=3$.  By the second statement of Lemma~\ref{lem:cdeuler} we have
\[
\chi(\B_3[4]) = b_0 - b_1 + b_ 2 = 0.
\]
By the $n=3$ case of Theorem~\ref{thm:main} we have $b_1(\B_3[4]) = 6$. Since 
$b_0 = 1$, we find that $b_2(\B_3[4]) = 5$. By the first statement of Lemma~\ref{lem:cdeuler}, we have found all of the nontrivial Betti numbers of $\B_3[4]$. 

Next we treat the case $n=4$. As in Section~\ref{sec:ala} we have $\B_4[4]\cong \PMod_{0,5}^2\times \ZZ$.  From this the K\"{u}nneth theorem gives
\[
H^j(\B_4[4]; \Q) \cong H^j(\PMod_{0,5}^2; \Q)\oplus H^{j-1}(\PMod_{0,5}^2; \Q)\]
for all $j\geq 1$. Thus for all $j\geq 1$ we have
\[
b_j = b_j(\PMod_{0,5}^2) + b_{j-1}(\PMod_{0,5}^2).
\]
It follows from Lemma~\ref{lem:cdeuler} and the isomorphism $\B_{2g+1}[4] \cong \PMod_{0,2g+2}^2 \times \ZZ$ (Section~\ref{sec:ala}) that $\cd \PMod_{0,5}^2 = 2$.  Thus by Lemma~\ref{lem:eulerlemma}
\[
64 = \chi(\PMod_{0,5}^2) = 1 - b_1(\PMod_{0,5}^2) + b_2(\PMod_{0,5}^2).
\]
Since $b_1(\PMod_{0,5}^2) = 20$, we obtain $b_2(\PMod_{0,5}^2) = 83$. Thus 
\[
b_2(\B_4[4]) = b_2(\PMod_{0,5}^2) + b_1(\PMod_{0,5}^2) = 83 + 20 = 103.
\]
Finally, since $\chi(\B_4[4]) = 0$ we have
\[
b_3(\B_4[4]) = 1 - b_1(\B_4[4]) + b_2(\B_4[4]) = 1-21 +103 = 83.
\]
Since $\cd \B_4[4] = 3$, we have found all of the non-trivial Betti numbers of $\B_4[4]$.
\end{proof}

\begin{proof}[Proof of Proposition~\ref{prop:genus2prop}]

Throughout we use the equality $\SMod_2[4] = \Mod_2[4]$, which follows immediately from the equality $\SMod_2 = \Mod_2$; see \cite[Section 9.4.2]{farbmargalit}.

By Lemma~\ref{lem:eulerlemma} we have $\chi(\text{Mod}_2[4]) = -3072$ and 
by Lemma~\ref{lem:cd} we have $\cd \Mod_2[4] = 3$.  Thus
\[
1 - b_1 + b_2 -b_3 = -3072.
\]
By Corollary~\ref{cor:mainC} we have $b_1 = 54$, whence $b_3 = 3019 + b_2$.  It remains to bound $b_2$ from below. 

Since $\B_5[4] \cong \Mod_2[4] \times \ZZ$, the K\"{u}nneth theorem gives
\[
H_2(\Mod_2[4]; \Q)\oplus H_1(\Mod_2[4]; \Q) \cong H_2(\B_5[4];\Q)
\]
and therefore that 
\[
b_2 + b_1 =  \dim H_2(\B_5[4]; \Q)
\] 
Since the map $\B_5[4]\rightarrow \B_4[4]$ induced by forgetting the last marked point in $\D_5$ is split, the induced map 
\[
H_2(\B_5[4]; \Q)\rightarrow H_2(\B_4[4];\Q)
\] 
is surjective. By the $n=4$ case of Theorem~\ref{thm:smallbetti} we have
\[
\dim H_2(\B_5[4]; \Q) \geq \dim H_2(\B_4[4]; \Q) = 103
\]
It follows that $b_2 \geq 103 - 54 = 49$ and therefore that $b_3\geq 3019 + 49 = 3068$. 
\end{proof}


\section{Hyperelliptic Torelli groups}\label{hyptorellisection}

In this section prove Theorem~\ref{thm:mainF}, which states that 
\[
\dim H_1(\SI_{g};\Q)\geq \frac{1}{6}\left(20g^4+12g^3-5g^2+9g-6\right).
\]
After recalling some facts about the second Johnson homomorphism $\tau_2$, we proceed to the proof of the theorem.  At the end of the section we prove Proposition~\ref{prop:unbound}.

In this section, $\Sp_g(\ZZ)[m]$ denotes the level $m$ congruence subgroup of $\Sp_g(\ZZ)$, that is, the kernel of the mod $m$ reduction map.

\p{The second Johnson homomorphism}  Let $\pi = \pi_1(\Sigma_g, *)$ and let $\pi^{(k)}$ denote the $k$th term of the lower central series of $\pi$. We define $\L_k = \pi^{(k)}/\pi^{(k+1)}\otimes \Q$.  Let $\K_g$ denote the subgroup of $\I_g$ generated by Dehn twists about separating simple closed curves. The second Johnson homomorphism is a $\Mod_g$-equivariant homomorphism
\[
\tau_2:\K_g\rightarrow \text{Hom}(\L_1,\L_3);
\]
see the papers by Hain and Morita \cite{hainkahler, morita} for the definition. The image of $\tau_2$ is a representation of $\Sp_g(\ZZ)$. Work of Hain \cite{haininfpresentations} implies that the image is isomorphic to the restriction to $\Sp_g(\ZZ)$ of the irreducible $\Sp_g(\Q)$-representation $V(2\lambda_2)$, where $\lambda_1,\lambda_2,\ldots, \lambda_g$ is a system of fundamental weights for $\Sp_g(\Q)$ (see also \cite[p.377]{morita}). 

The group $\SI_g$ is contained in $\K_g$; see of the paper by Brendle, Putman, and the second author \cite[p. 268]{brendlemargalitputman}.  Thus we may restrict $\tau_2$ to $\SI_g$ to obtain
\[
j: \SI_g\rightarrow V(2\lambda_2).
\]
The group $\SI_g$ is normal in $\SMod_g[2]$.  Also, A'Campo proved that the symplectic representation $\text{Mod}_g\rightarrow \Sp_g(\ZZ)$ induces an isomorphism $\SMod_g[2]/\SI_g\ \cong \Sp_g(\ZZ)[2]$.  It follows that $j$ is $\SMod_g[2]$-equivariant and that it induces an $\Sp_g(\ZZ)[2]$-equivariant map
\[
j_*: H_1(\SI_g;\Q)\rightarrow V(2\lambda_2).
\]

\begin{proof}[Proof of Theorem~\ref{thm:mainF}]

Let  $i: \SI_g\rightarrow \SMod_g[4]$ denote the inclusion and consider the map
\[
\Phi: H_1(\SI_g;\Q)\rightarrow H_1(\SMod_g[4];\Q)\oplus V(2\lambda_2)
\]
defined by
\[
\Phi(x) = \left(i_*(x), j_*(x)\right).
\]
By Corollary~\ref{cor:mainC} the dimension of the first summand is 
\[
3{2g+1\choose 4} + 3{2g+1\choose 3} + {2g+1 \choose 2} -1.
\]
The dimension of the second summand is also known (see \cite[Lemma 8.5]{hainkahler}):
\[
\dim V(2\lambda_2) = \frac{g(g-1)(4g^2+4g-3)}{3}.
\] 
Since the sum of these two dimensions is the desired lower bound, it suffices to prove that $\Phi$ is surjective.  To do this, we will first show that $i_*$ and $j_*$ are surjective.

First we show that $i_*$ is surjective. It follows from the aforementioned theorem of A'Campo that the map $\text{Mod}_g\rightarrow \Sp_g(\ZZ)$ induces an isomorphism $\SMod_g[4]/\SI_g \cong \Sp_g(\ZZ)[4]$.
We therefore have an exact sequence
\[
H_1(\SI_g;\Q)\rightarrow H_1(\SMod_g[4];\Q)\rightarrow H_1(\Sp_g(\ZZ)[4];\Q)
\]
The homology group $H_1(\Sp_g(\ZZ)[m];\Q)$ is zero for $g \geq 2$ and $m \geq 0$  (see \cite[p. 3]{putmanabelianization}).  Thus $i_*$ is surjective.

We now show that the map $j_*$ is surjective for each $g\geq 2$. By the Borel density theorem, any lattice $\Gamma \leq \Sp_g(\R)$ is Zariski dense; see \cite[p. 766]{putmanabelianization}.  It follows that the irreducible $\Sp_g(\R)$-module $V(2\lambda_2)\otimes \R$ is irreducible as a $\Gamma$-module.  Since tensoring a reducible representation with $\R$ results in a reducible representation, it follows that $V(2\lambda_2)$ is irreducible as a $\Gamma$-module. As $j_*$ is $\Sp_g(\ZZ)[2]$-equivariant, it suffices to show that $j_*$ is non-zero.

Morita proved that if $c$ is any (nontrivial) separating curve in $\Sigma_g$ then $\tau_2(T_c)$ is non-zero \cite[Proposition 1.1]{morita}.  If $c$ is any separating curve in $\Sigma_g$ that is preserved by the hyperelliptic involution $s$, then $T_c$ lies in $\SI_g$.  It follows that $j$ is non-zero, and hence that that $j_*$ is non-zero, hence surjective.

Finally, we show that $\Phi$ is surjective. Let $(x,y) \in H_1(\SMod_g[4];\Q)\oplus V(2\lambda_2)$.  We will show that $(x,y)$ lies in the image of $\Phi$.  Since both $i_*$ and $j_*$ are surjective, we can choose $\widetilde{x}, \widetilde{y} \in H_1(\SI_g;\Q)$ such that 
\[
i_*(\widetilde{x}) = x \quad \text{and} \quad    j_*(\widetilde{y}) = y. 
\]
Let $y_1 = -j_*(\widetilde{x})+y$.  
Since $V(2\lambda_2)$ is an irreducible $\Sp_g(\ZZ)[4]$-representation, the corresponding space of coinvariants $V(2\lambda_2)_{\Sp_g(\ZZ)[4]}$ is trivial.  This is the same as saying that $V(2\lambda_2)$ is spanned by
\[
\left\{  (h-1)v_h \mid h\in \Sp_g(\ZZ)[4],\ v_h\in H_1(\SI_g;\Q)  \right\}.
\]
In particular, there is a finite set $\cH\subset \Sp_g(\ZZ)[4]$ such that
\[
y_1= \sum_{h\in \cH} (h-1)v_h
\]
where each $v_h$ lies in $V(2\lambda_2)$. 

For each $h \in \cH$, let $\tilde v_h$ be an element of the $j_*$-preimage of $v_h$.  Let
\[
\widetilde y_1= \sum_{h\in \cH} (h-1)\widetilde v_h.
\]
By construction, we have that $j_*(\widetilde y_1) = y_1 = -j_*(\widetilde{x})+y$.  

Since $\SI_g$ is normal in $\SMod_g[4]$, and since $\SMod_g[4]/\SI_g \cong \Sp_g(\ZZ)[4]$,  the map $i_*$ is $\Sp_g(\ZZ)[4]$-equivariant.Thus
\[
i_*(\widetilde y_1) = \sum_{h\in \cH} (h-1)i_*(\widetilde v_h) = \sum_{h\in \cH} hi_*(\widetilde v_h)-i_*(\widetilde v_h).
\]
Since $\SI_g$ is contained in $\SMod_g[4]$ there is a well-defined action of the quotient $\Sp_g(\ZZ)[4]$ on $H_1(\SMod_g[4];\Q)$.  But the action of $\SMod_g[4]$ on $H_1(\SMod_g[4];\Q)$ is trivial and so the action of $\Sp_g(\ZZ)[4]$ is trivial.  Thus our last expression for $i_*(\widetilde y_1)$ is zero.  It follows that $\Phi(\widetilde{x}+\widetilde y_1) = (x,y)$, as desired.
\end{proof}

We now prove Proposition~\ref{prop:unbound}, which states that for $n$ odd the first homology $H_1(\BI_n;\Q)$ is infinite dimensional if the sequence $\left(\dim H_1(\B_n[m];\Q)\right)_{m=1}^{\infty}$ unbounded.

\begin{proof}[Proof of Proposition~\ref{prop:unbound}]

As in the statement, let $n=2g+1$ be odd.  For each $m$ we have $\BI_{2g+1} \subset \B_{2g+1}[m]$. Indeed, $\BI_{2g+1}$ is equal to the intersection of all $\B_{2g+1}[m]$ with $m\geq 1$. It follows from the work of Brendle and the second author \cite{brendlemargalit} that for $g\geq 1$ there is an isomorphism 
\[
\B_{2g+1}[2m]/\BI_{2g+1} \cong \Sp_g(\ZZ)[2m].
\]
The fact that $H_1(\Sp_g(\ZZ)[2m];\Q) = 0$ for $g\geq 2$ (see, for example, \cite{putman}) implies that there is a surjection
\[
H_1(\BI_{2g+1};\Q) \to H_1(\B_{2g+1}[2m];\Q).
\]
From this, it follows that if the sequence $\left(\dim H_1(\B_n[2m];\Q)\right)_{m=1}^{\infty}$ were unbounded then $H_1(\BI_{2g+1};\Q)$ would be infinite dimensional.  To complete the proof, it now suffices to observe that the transfer homomorphism gives a surjection $H_1(\B_{2g+1}[2m];\Q) \to H_1(\B_{2g+1}[m];\Q)$.  So if the sequence $\left(\dim H_1(\B_n[m];\Q)\right)_{m=1}^{\infty}$ unbounded the sequence $\left(\dim H_1(\B_n[2m];\Q)\right)_{m=1}^{\infty}$ would be unbounded as well.
\end{proof}

\section{2-torsion on the characteristic varieties of the braid arrangement}
\label{sec:tors}

The goal of this section is to prove Theorem~\ref{thm:2torsion}.  We first introduce a general branching rule that gives the restriction of an irreducible $\Z_n$-representation of  to $\PZ_n$ (Lemma~\ref{lem:branch}).  To prove the theorem we apply the lemma to our description of $H_1(\B_n[4];\C)$ from Theorem~\ref{thm:reps} in order to explicitly compute all of the 2-torsion points that lie on the characteristic variety of the braid arrangement. 

\begin{lemma}\label{lem:branch}
Let $n\geq 2$. Assume that $\rho$ is an irreducible $I$-isotypic representation of $\Z_m^I$ for some full subset $I\subset [m]^{\underline 2}$ where $m\leq n$. Then we have an isomorphism of $\PZ_n$-modules
\[
\Res^{\Z_n}_{\overline{PB}_n}V_n(\rho, \lambda) \cong \left(\dim V_m(\rho)\right)\left(\dim V(\lambda)\right)\bigoplus_{g\in \Z_n/\Z_n^I}V_{g(I)}
\]
\end{lemma}

\begin{proof}

By the definition of the $V_n(\rho, \lambda)$, by the formula for the restriction of an induced representation \cite[Proposition 5.6(b)]{brown}, and by the fact that $\overline{PB}_n\subset \Z_n^J$ for every $J\subset [n]^{\underline 2}$, we have
\begin{align*}
\Res^{\Z_n}_{\overline{PB}_n}V_n(\rho, \lambda) &= \Res^{\Z_n}_{\overline{PB}_n}\Ind_{\Z_n^I}^{\Z_n}V_m(\rho)\boxtimes V(\lambda) \cong  \bigoplus_{g\in \Z_n/\Z_n^I}\Res_{\PZ_n}^{\Z_n^{g(I)}}gV_m(\rho)\boxtimes V(\lambda).
\end{align*}

To complete the proof, we observe that $V_{n-m}(\lambda)$ restricts to a direct sum of $\dim V_{n-m}(\lambda)$ copies of the trivial $\PZ_n$-module and that $V_m(\rho)$ restricts to the direct sum of $\dim V_m(\rho)$ copies of the representation $V_I$. Thus $V_m(\rho)\boxtimes V_{n-m}(\lambda)$ restricts to the direct sum of $\left(\dim V_m(\rho)\right)\left(\dim V_{n-m}(\lambda)\right)$ copies of $V_I$. Employing Lemma~\ref{lem:symm}, we see that $gV_m(\rho)\boxtimes V_{n-m}(\lambda)$ restricts to the direct sum of $\left(\dim V_m(\rho)\right)\left(\dim V_{n-m}(\lambda)\right)$ copies of $V_{g(I)}$. The result follows. 
\end{proof}

\begin{proof}[Proof of Theorem~\ref{thm:2torsion}]
A 2-torsion point of $V_d(X_n)$ is a homomorphism $\rho : \PB_n \to \mu_2$ with $\dim H^1(X_n;\C_\rho)\geq d$.  As in Section~\ref{sec:bnbar}, any such $\rho$ is equal to some $\rho_I$.  For any $I\subset [n]^{\underline 2}$, we may identify the fiber of $\C_{\rho_I}$ with the $\PZ_n$-module $V_I$, viewed as a $\PB_n$-module. The fact that $\PZ_n = \PB_n/\B_n[4]$ is a finite group implies that the Hochschild--Serre spectral sequence 
\[
E_2^{p,q} = H^p(\PZ_n; H^q(\B_n[4]; V_I))\implies H^{p+q}(\PB_n; V_I)
\] 
degenerates at the $E_2$ page. This gives isomorphisms
\[
H^1(X_n; \C_{\rho_I}) \cong H^1(\PB_n; V_I)\cong \left(H^1(\B_n[4]; \C)\otimes V_I\right)^{\PZ_n}.
\]
We conclude that $\dim H^1(X_n; \C_{\rho_I})$ is equal to the multiplicity of $V_I$ in $H^1(\B_n[4]; \C)\cong H_1(\B_n[4]; \C)$, regarded as a $\PZ_n$-module. Combining Theorem~\ref{thm:reps} with Lemma~\ref{lem:branch} we see that
\[
\Res^{\Z_n}_{\PZ_n} H_1(\B_n[4]; \C) \cong
\begin{cases}
\C^3\oplus \left(\bigoplus_{g\in \Z_3/\Z_3^{I_3}}V_{g(I_3)}\right) & n =3 \\
\C^{{n \choose 2}}\oplus \left(\bigoplus_{g\in \Z_n/\Z_n^{I_3}}V_{g(I_3)}\oplus \bigoplus_{g\in \Z_n/\Z_3^{I_4}}V_{g(I_4)}\right) & n \geq 4
\end{cases}
\]
Thus the multiplicity of any nontrivial $V_I$ in $\Res^{\Z_n}_{\PZ_n}H_1(\B_n[4]; \C)$ is at most 1. That is, for $d\geq 2$ there are no 2-torsion points on $V_d(X_n)$. Further, this decomposition shows that the 2-torsion points on $V_1(X_n)$ are exactly those of the form $\rho_{g(I_3)}$ or $\rho_{g(I_4)}$ for $g\in S_n$.  It remains only to show that these points lie on $\check{V}_1(X_n)$. 

Cohen--Suciu \cite{cohensuciu} found explicit equations for all of the components of $\check{V}_1(X_n)$.  For $i < j < k$ there is a component
\[
V_{ijk} = \{{\bf t}\in (\C^{\times})^{{n \choose 2}}: t_{ij}t_{ik}t_{jk} = 1\ \text{and}\ t_{pq}=1\ \text{if}\ |\{p,q\}\cap \{i,j,k\}|\leq 1\}
\]
and for each 4-element set $I = \{i,j,k,\ell\}$ with $i < j <k < \ell $ there is a component
\[
V_{ijk\ell} = \{{\bf t}\in (\C^{\times})^{{n \choose 2}}: t_{pq} = t_{rs}\ \text{if}\ \{p,q\}\cup \{r,s\} = I, t_{pq} = 1\ \text{if}\ \{p,q\}\not\subset I \ , \prod t_{pq} = 1 \}.
\]
We directly verify that the 2-torsion points of the form $\rho_{g(I_3)}$ lie in $V_{g(123)}$ and those of the form $\rho_{g(I_4)}$ lie in $V_{g(1234)}$. This completes the proof of the theorem. 
\end{proof}


\bibliographystyle{plain}
\bibliography{Level4}

\end{document}